\documentclass[12pt,centertags,oneside]{amsart}

\usepackage{tikz}
\usepackage[mathscr]{eucal}
\usepackage{asymptote}
\usepackage{caption}
\usetikzlibrary{arrows}
\usepackage{import}
\usepackage{graphicx,graphics,array,booktabs}

\usepackage{amsmath,amstext,amsthm,amscd,typearea,hyperref}
\usepackage{amssymb}
\usepackage{a4wide}
\usepackage[mathscr]{eucal}
\usepackage{mathrsfs}
\usepackage{typearea}
\usepackage{charter}
\usepackage{pdfsync}
\usepackage{multirow}
\usepackage{array}
\usepackage[utf8]{vietnam}
\usepackage{amsmath, amssymb, latexsym, amscd, amsthm,amsfonts,amstext}
\usepackage[mathscr]{eucal}

\usepackage[english]{babel}
\usepackage[utf8]{inputenc}

\usepackage[a4paper,width=16.2cm,top=3cm,bottom=3cm]{geometry}

\numberwithin{equation}{section}

\newtheorem{theorem}{Theorem}[section]
\newtheorem{definition}[theorem]{Definition}
\newtheorem{proposition}[theorem]{Proposition}
\newtheorem{corollary}[theorem]{Corollary}
\newtheorem{lemma}[theorem]{Lemma}
\newtheorem{remark}[theorem]{Remark}
\newtheorem{remarks}[theorem]{Remarks}
\newtheorem{example}[theorem]{Example}

\newtheorem{problem}[theorem]{Problem}

\newcommand{\cali}[1]{\mathscr{#1}}

\newcommand{\Tan}{\mathop{\mathrm{Tan}}\nolimits}
\newcommand{\Cotan}{\mathop{\mathrm{Cotan}}\nolimits}
\newcommand{\Nor}{\mathop{\mathrm{Nor}}\nolimits}

\newcommand{\Ebf}{\mathbf{E}}

\newcommand{\GL}{{\rm GL}}

\newcommand{\U}{{\rm U}}

\newcommand{\PSL}{{\rm PSL}}

\newcommand{\Aut}{{\rm Aut}}

\newcommand{\lof}{\mathop{\mathrm{{log^\star}}}\nolimits}
\newcommand{\supp}{{\rm supp}}

\newcommand{\const}{\mathop{\mathrm{const}}\nolimits}

\newcommand{\Nc}{\cali{N}}
\newcommand{\Dom}{{\rm Dom}}

\newcommand{\diam}{{\rm diam}}

\newcommand{\dist}{\mathop{\mathrm{dist}}\nolimits}
\renewcommand{\Re}{\mathop{\mathrm{Re}}\nolimits}
\renewcommand{\Im}{\mathop{\mathrm{Im}}\nolimits}
\newcommand{\vol}{\mathop{\mathrm{vol}}}

\newcommand{\Cf}{\mathfrak{C}}
\newcommand{\Df}{\mathfrak{D}}

\newcommand{\loc}{{loc}}
\newcommand{\ddc}{dd^c}
\newcommand{\ddcy}{{dd^c_{y}}}

\newcommand{\ddczw}{{dd^c_{z,w}}}
\newcommand{\dc}{d^c}

\newcommand{\dbar}{\overline\partial}
\newcommand{\ddbar}{\partial\overline\partial}
\renewcommand{\GL}{{\rm GL}}

\newcommand{\Hyp}{{\rm Hyp}}
\newcommand{\Par}{{\rm Par}}

\newcommand{\id}{{\rm id}}

\newcommand{\Ic}{\cali{I}}

\newcommand{\hol}{{\rm hol}}

\newcommand{\Sing}{{\rm Sing}}
\newcommand{\codim}{{\rm codim\ \!}}

\newcommand{\chac}{{\rm \mathbf{1}}}
\newcommand{\Ac}{\cali{A}}
\newcommand{\Bc}{\cali{B}}
\newcommand{\Cc}{\cali{C}}
\newcommand{\Dc}{\cali{D}}
\newcommand{\Ec}{\cali{E}}
\newcommand{\Fc}{\cali{F}}
\newcommand{\Gc}{\cali{G}}
\newcommand{\Hc}{\cali{H}}

\newcommand{\Lc}{\cali{L}}
\renewcommand{\Mc}{\cali{M}}
\newcommand{\Oc}{\cali{O}}
\newcommand{\Pc}{\cali{P}}

\newcommand{\Uc}{\cali{U}}

\newcommand{\Sc}{\cali{S}}

\newcommand{\Kc}{\cali{K}}

\newcommand{\A}{\mathbb{A}}
\newcommand{\B}{\mathbb{B}}
\newcommand{\C}{\mathbb{C}}
\newcommand{\D}{\mathbb{D}}
\newcommand{\N}{\mathbb{N}}
\newcommand{\Z}{\mathbb{Z}}
\newcommand{\K}{\mathbb{K}}
\newcommand{\R}{\mathbb{R}}
\renewcommand\P{\mathbb{P}}
\newcommand{\T}{\mathbb{T}}
\renewcommand{\U}{\mathbb{U}}
\renewcommand{\S}{\mathbb{S}}
\newcommand{\E}{\mathbb{E}}

\title[Ergodic theorems  for laminations and foliations]{Ergodic theorems   for    laminations and foliations: recent results and  perspectives}

 \author{Vi{\^e}t-Anh Nguy{\^e}n}
\address{Universit\'e de Lille, 
Laboratoire de math\'ematiques Paul Painlev\'e, 
CNRS U.M.R. 8524,  
59655 Villeneuve d'Ascq Cedex, 
France.}

\address{and 
Vietnam Institute for Advanced Study in Mathematics (VIASM),  157 Chua Lang Street, Hanoi, Vietnam.  }

\email{Viet-Anh.Nguyen@univ-lille.fr,
{\tt http://www.math.univ-lille1.fr/$\sim$vnguyen}}
\dedicatory{Dedicated to   My Beloved Father}

\date{May 22, 2020}
\begin{document}
\selectlanguage{english} 
\begin{abstract}
  This  report discusses    recent results as well as  new perspectives  in the  ergodic  theory for   Riemann  surface laminations,  with an emphasis    on
 singular  holomorphic  foliations by curves.
The central  notions of these  developments are  {\it leafwise Poincar\'e metric,} {\it  directed positive harmonic  currents,} 
  {\it multiplicative cocycles} and {\it Lyapunov  exponents.}  We  deal with various  
ergodic  theorems  for such  laminations: Random and Operator Ergodic Theorems, (Geometric) Birkhoff Ergodic Theorems, Oseledec Multiplicative Ergodic Theorem and Unique Ergodicity Theorems.
Applications of these theorems are  also given. In particular,  we define and study  the canonical {\it Lyapunov exponents}
for  a large  family of singular holomorphic foliations on compact projective surfaces. Topological  and algebro-geometric   interpretations of these
characteristic numbers are also  treated.
 These results highlight the  strong  similarity as well as   the  fundamental  differences  between the ergodic  theory of maps
and that  of  Riemann  surface laminations. Most  of the  results  reported here are known. However, sufficient  conditions  for abstract  heat diffusions  to coincide with the  leafwise heat diffusions  
(Subsection \ref{SS:coincidence}) are new  ones.
	\end{abstract}

\maketitle

\medskip

\noindent
{\bf Classification AMS 2020}: Primary:  37A30, 57R30; Secondary: 58J35, 58J65,
60J65, 32J25.

\medskip


\tableofcontents


 
\section{Introduction} \label{S:introduction}

\subsection{Prelude}
The goal of these notes is to explain recent ergodic theorems for  laminations by   Riemann surfaces (without and with singularities), 
and particularly those for  singular holomorphic foliations by curves. We make an
emphasis on the analytic approach to the dynamical  theory of laminations and foliations. This illustrates 
a prominent role of the theory of currents in the field.

 There is  a  natural correspondence between  the dynamics of  Riemann  surface laminations  and  those of iterations  of continuous maps.
 More  specifically in the  meromorphic category, this correspondence becomes  a connection  between   the dynamics of  singular holomorphic  foliations   in dimension $k\geq 2$ and 
 those of iterations of meromorphic maps in dimension  $k-1.$
 Ergodic  theorems  for measurable  maps is  by now   well-understood, see for instance  the  monograph of Krengel \cite{Krengel}.
 Those for  the subclass of all  meromorphic maps  have been  studied  intensively  only during the last three decades, see  the survey of Dinh-Sibony
 \cite{DinhSibony10}. So a natural question  arises whether one can obtain  analogous ergodic theorems 
  for    Riemann surface laminations, and in particular, for  the subclass  of  singular holomorphic foliations by curves.
  In this  article we  try to answer this fundamental  question  by analyzing some  known results and  by  proving some  new ones.
  It is   worthy noting that    the ergodic theory  of laminations  and foliations requires many new ideas and presents many difficulties.
  Traditional dynamical systems techniques are based on a singly generated system, and these methods require the existence of invariant measures.
  For laminations and foliations, the dynamics are defined by the local actions of the holonomy maps, which provides a more complex system, 
  and often precludes the existence of invariant measures.  Therefore, the ergodic theory  of laminations  and foliations is rich in ideas and problems, where every insight opens new avenues of exploration.

    In the next subsection  we will recall  two basic  ergodic theorems: one for measurable maps and the other for holomorphic maps.
    These theorems will serve as our starting models in order to  look for  ergodic theorems in the context of laminations-foliations.
    The last subsection outlines the organization  of the paper.
These notes  may be  considered as the continuation  of our previous survey  \cite{NguyenVietAnh18d}. However, in  the latter  article,  we are interested  in the whole ergodic  theory of
Riemann surface laminations, which is clearly a  broader topic. In the present  work, we only specialize in ergodic theorems and related matters. So  some fundamental topics such as the entropies etc. are not treated here.
It should be noted that some  progress has been  made in this  area since
the publication of  our previous survey \cite{NguyenVietAnh18d}. Namely, Problem  4.7 (Zero Lelong numbers),  Problem 5.8 (Unique ergodicity) and Problem 7.7 (Negative Lyapunov exponent)  therein  
have  recently been  solved in  \cite{NguyenVietAnh19}, \cite{DinhNguyenSibony18} and 
\cite{NguyenVietAnh18c} respectively. 
Moreover, we try  to rewrite  several parts of \cite{NguyenVietAnh18d} in  a  somewhat more general context of (not necessarily  hyperbolic) Riemann surface  laminations with singularities. 
We hope that  the ideas  reviewed  in these two surveys  will  be    developed and expanded   in the  future. 
In writing the present  article,  we are  inspired  by   the  surveys and lecture notes of Deroin   
 \cite{Deroin13}, Dinh-Sibony \cite{DinhSibony18c},
    Forn\ae ss-Sibony \cite {FornaessSibony08},  Ghys \cite{Ghys}, Hurder \cite{Hurder}, Zakeri  \cite{Zakeri} etc.   
 \subsection{Two ergodic theorems in the dynamics of iterations of maps}

 To state  the  first  ergodic theorem, we need to introduce some notations and terminology. Let $f : X \to X$ be a 
 map on a probability measure space $(X, \Ac, \mu)$  and suppose $\varphi$ is a $\mu$-integrable function, i.e.
 $\varphi\in  L^1(\mu).$
 Then we define the following averages:
\begin{itemize}
 \item[$\bullet$]   Time average (up to level $n\in\N$): This is defined as the average  over iterations of $f$ from  $0$ up to the $(n-1)$-iteration starting from some initial point $x\in X:$
$$  \frac {1}{n}\sum _{k=0}^{n-1}\varphi(f^{k} (x)). $$
This  sum can also be rewritten as  
$ \langle m^+_{x,n},\varphi\rangle,$  where $ m^+_{x,n}$  is a probability measure which is  the average of Dirac masses at forward orbit of $x$ from  time $0$ up to time $n-1:$
$$
 m^+_{x,n}:={1\over n}  \sum_{k=0}^{n-1} \delta_{f^k(x)}.
$$
Here, for $a\in X$  let $\delta_a$ denotes the Dirac mass at $a.$ The  sign $+$  in  $m^+_{x,n}$ emphasizes  that we are concerned  with the forward orbit of $x.$ 
  \item[$\bullet$]  Space average: 
      $$\langle \mu,\varphi\rangle =\int_X \varphi\,d\mu.$$
\end{itemize}
In general the limit of time averages (if exists)  as $n\to\infty$ and space average may be different. 

We say  that  $\mu$ is {\it ergodic} with respect to $f$ if for every element $A$ of the $\sigma$-algebra $\Ac$ with $f^{-1}(A) = A,$ then  either $\mu(A) = 0$ or $\mu(A) = 1.$ 
We say  that  $\mu$ is {\it invariant}  with respect to $f$ if  $f_*\mu=\mu,$ i.e. $\mu(f^{-1}(A))=\mu(
A)$ for $A\in\Ac.$ The following  theorem plays a fundamental role in the dynamics of maps-iterations.
\begin{theorem}{\rm (Birkhorff  Ergodic  Theorem  \cite{Birkhoff,Krengel,Walters})} \label{T:Birkhorff}
If $\mu$ is   invariant and ergodic, then 
$$\lim_{n\to\infty} \langle m^+_{x,n}, \varphi\rangle = \langle \mu,\varphi\rangle \qquad\text{for $\mu$-almost everywhere $x\in X$. }$$
In other words, as the time $n$ tends to infinity, the limit of  time averages  is equal to the space average $\mu$-almost everywhere.
In particular, $
m^+_{x,n}\to  \mu $ weakly as $n$  tends to infinity.

 \end{theorem}
 There are   many    ergodic  theorems for iterations  of maps (see \cite{Krengel}).
 
 Now  we turn to  the statement  of the unique ergodicity of  surjective holomorphic maps defined on compact K\"ahler manifolds.
 Let $X$  be  a  compact K{\"a}hler  manifold of dimension
$k$ and $\omega$ a K\"ahler form on $X$  so normalized  that $\omega^k$ defines a probability measure on $X$.  
 Let $f:X\rightarrow X$ be a surjective holomorphic map.  
 Let $d_p(f)$ (or $d_p$ if there is no possible confusion), $0\leq p\leq k$, be the 
dynamical degree   of  order $p$ of $f$. This is a bi-meromorphic invariant which  measures the  norm growth of the operators $(f^n)^*$
acting  on the Hodge cohomology group $H^{p,p}(X,\C)$ when $n$  tends
to infinity, that is,
\begin{equation*}
d_p=d_p(f):=\lim\limits_{n\to\infty} \| (f^n)^*\|^{1\over n},\qquad \text{where}\qquad  (f^n)^*:\  H^{p,p}(X,\C)\to H^{p,p}(X,\C).
\end{equation*}
We always have $d_0(f)=1$. The last dynamical degree $d_k(f)$ is {\it the topological degree} of $f$: it is equal to the number of points in a generic fiber of $f$.
We also denote it by $d_t(f)$ or simply by $d_t$. 

We say that $f$ is with {\it dominant topological degree}\footnote{In some references, such a map is said to be with large topological degree; we think the word ``dominant'' is more appropriate.}  
if $d_t>d_p$ for every $0\leq p\leq k-1$. 
It is well-known that for such  a map $f,$ the  following weak limit  of probability measures
 $$\mu:=\lim_{n\to\infty} {1\over d_t^n}(f^n)^*\omega^k$$
 exists.  The probability measure $\mu$ is  called  
  {\it the  equilibrium measure}  of $f.$   
It has no mass on proper analytic  subsets of $X$, is {\it totally invariant}: 
  $d_t^{-1}f^*(\mu)=f_*(\mu)=\mu$  and  is exponentially mixing. 
 The measure $\mu$ is also the unique invariant measure with maximal entropy $\log d_t$. 
 We refer the reader to \cite{DinhSibony04,DinhSibony06,Guedj} for details, see also  \cite{Vu} for a  recent result.

   The set of preimages $f^{-n}(x)$ of $f^n$ consists of $d_t^n$ points counted with multiplicity (see e.g.  \cite[Lemma 4.7]{DinhNguyenVu}). 
  So  for $n\in\N,$
 the probability  measure  $$
 m^-_{x,n}:={1\over d_t^n}  (f^n)^\ast \delta_x={1\over d_t^n}  \sum_{y\in f^{-n}(x)} \delta_y$$
 is the   average of Dirac masses at backward orbit of $x$ up to past-time $n$. The sign $-$ in the notation $ m^-_{x,n}$ emphasizes that we are concerned with the backward orbit of $x.$
 \begin{figure}[h]
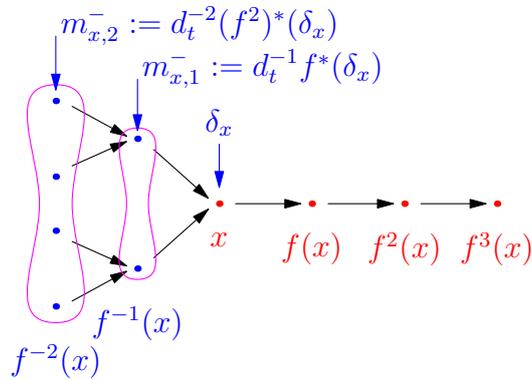

 \begin{center}
 
 \begin{asy}
size(7cm);
defaultpen(0.4);
arrowbar Fleche = Arrow(SimpleHead, 3, 30);
pen penna = black + 0.4;
pen punti = black + 2;
usepackage("amssymb");

dot((0,0),red);
label("$x$", (0, 0), 3*S, red);
draw((0.3,0)--(1.5,0), Arrow);
dot((1.7,0),red);
label("$f(x)$", (1.7, 0), 3*S,red);
draw((2,0)--(3.2,0), Arrow);
dot((3.4,0),red);
label("$f^2(x)$", (3.4, 0), 3*S,red);
draw((3.7,0)--(4.9,0), Arrow);
dot((5.1,0),red);
label("$f^3(x)$", (5.1, 0), 3*S,red);

dot((-1.5,1.2),blue);
draw((-1.2,1)--(-0.2,0.1), Arrow);
dot((-1.5,-1.2),blue);
label("$f^{-1}(x)$", (-1.5, -1.2), 4*S,blue);
draw((-1.2,-1)--(-0.2,-0.1), Arrow);

dot((-3,1.9),blue);
draw((-2.7,1.8)--(-1.7,1.25), Arrow);
dot((-3,0.5),blue);
draw((-2.7,0.7)--(-1.7,1.15), Arrow);

dot((-3,-1.9),blue);
draw((-2.7,-1.8)--(-1.7,-1.25), Arrow);
dot((-3,-0.5),blue);
draw((-2.7,-0.7)--(-1.7,-1.15), Arrow);
label("$f^{-2}(x)$", (-3, -1.9), 4*S,blue);

label("$\delta_x$", (0, 1.5), blue);
draw((0,1.1)--(0,0.3),blue, Arrow);

pair p1 = (-1.5, 1.4);
pair p2 = (-1.3, 0);
pair p3 = (-1.5, -1.4);
pair p4 = (-1.7,0);
guide U = p1{E}..{S}p2{S}..{W}p3{W}..{N}p4{N}..{E}cycle;
draw(U,magenta);
label("$m^-_{x,1}:=d_t^{-1}f^*(\delta_x)$", (-1.5, 2.5), 0.5*E, blue);
draw((-1.5,2.3)--(-1.5,1.35),blue,Arrow);

pair p5 = (-3, 2.2);
pair p6 = (-2.7, 0);
pair p7 = (-3, -2.2);
pair p8 = (-3.3,0);
guide U = p5{E}..{S}p6{S}..{W}p7{W}..{N}p8{N}..{E}cycle;
draw(U,magenta);
label("$m^-_{x,2}:=d_t^{-2}(f^2)^*(\delta_x)$", (-3, 3.3), 0.5*E, blue);
draw((-3,3.1)--(-3,2.1),blue,Arrow);
\end{asy}
 \caption{Schema to construct $m^-_{x,n}$ for $n=1,2$ of a map $f$ with $d_t=2.$}
\end{center}
\end{figure}
 The following theorem  gives a equidistribution of preimages of points by $f^n.$
 
\begin{theorem} {\rm (Unique ergodicity for holomorphic maps)} \label{T:DNT}
Let $f:X\to X$ be a surjective holomorphic map with  dominant topological  degree $d_t$ on a  compact K\"ahler manifold $(X,\omega)$ and  let $\mu$ be its equilibrium measure. 
Then  there is   a  (possibly empty) proper analytic  set $\Ec=\Ec_f$ of $X$ such that 
we have 
 $$m^-_{x,n} \to \mu \qquad \mbox{as} \quad n\to\infty$$
if and only if  $x\not\in \Ec.$ In fact, 
the exceptional set $\Ec$ is  characterized by the following two conditions: 
(1)  $f^{-1}(\Ec)\subset \Ec$; (2) any proper analytic subset of $X$ satisfying (1)  is contained in $\Ec$. Moreover, we have 
$$\Ec=f^{-1}(\Ec) =f(\Ec).$$
 \end{theorem}

 The above  result describes the  dichotomous behaviour of the equidistribution of preimages of points by $f^n.$
 Namely,  the typical case  $x\not\in \Ec$  is  characterized by the fact that  $m^-_{x,n}$ tends to the  equilibrium measure $\mu.$
 The complementary case  (i. e.  $x\in\Ec$) is  non-typical  since  the exceptional set $\Ec$ is  small: it is a (possibly empty) proper analytic set. 
 The above theorem was obtained for holomorphic endomorphisms of
$\P^k$ in \cite{BriendDuval01, DinhSibony03,FornaessSibony94}. For the case of dimension 1, see \cite{Brolin,FreireLopesMane,Lyubich,Tortrat}.
  A proof of this theorem  even for the broader class of meromorphic maps is given in \cite[Theorem 1.3]{DinhNguyenTruong}. 
 The reader is invited to consult the survey  \cite{DinhSibony10}
  for a   comprehensive  treatment of the  equidistribution of preimages by a holomorphic map in $\P^k.$ 
  
  When $X=\P^k,$ the space of all surjective holomorphic maps  with a given  algebraic degree   $d>1$
  is canonically
  identified to a Zariski open set $\Omega_{k,d}$ of some $\P^N.$ It is  well-known that for  $f\in \Omega_{k,d},$  we have  $d_k(f)=d^k,$ and  hence  $f$ is with dominant topological degree.
  Moreover, we have    $\Ec_f=\varnothing$  for $f$ belonging to a Zariski open subset of $\Omega_{k,d}.$
  In particular, for a generic surjective holomorphic map $f$ of a given  degree $d>1$ on $\P^k$ of a given degree $d>1$   we have that   $\Ec_f=\varnothing.$

 \subsection{Outline of the  article}
  In  this  article  we  undertake  the following two tasks.  First, we survey  various  ergodic theorems  in the context of laminations-foliations  which  follow  the models of the  above two ergodic theorems \ref{T:Birkhorff} and \ref{T:DNT}.  
  Second, in  Subsection \ref{SS:coincidence} we  prove  some new comparison results regarding  the heat diffusions.
  Many concepts  in the context of  maps-iterations will find their analogues  in the new  context of laminations. Some have even more than  one analogue.
  The work is  organized as follows.
  
  \medskip
  
  In Section  \ref{S:basics}, we  will recall  basic facts on Riemann  surface laminations (without  and  with  singularities),  singular holomorphic foliations.
  Throughout the paper, unless otherwise  stated we will refer to all these objects as  the common abridged name {\bf lamination}. Moreover, {\bf (singular) (holomorphic) foliations} mean 
  (singular) (holomorphic) foliations by curves.
 The hyperbolicity and the leafwise Poincar\'e metric as well as the  hyperbolic and  parabolic parts of a lamination  will be introduced. The leafwise  Poincar\'e metric is
 regarded as the  hyperbolic time  for the laminations.
Consequently, we will develop the leafwise heat diffusions
 and define various  concepts  of harmonic  measures
 for    laminations. Next, we  study companion  notions  for  harmonic measures,  namely,  directed  positive  harmonic current for  Riemann surface laminations 
 and   directed positive $\ddc$-closed  current for  holomorphically immersed  Riemann  surface laminations. The latter class of laminations  contains the family of
 singular holomorphic foliations. We investigate  the  relationship  {\it (positive quasi-) harmonic measures $\longleftrightarrow$ directed (positive) harmonic currents.}
 Historically, 
  Garnett in \cite{Garnett}   introduces the notion of  harmonic measures
  and  considers the diffusions of the  heat equation in the Riemannian context.
  Her idea is further  developed by Candel-Conlon  \cite{CandelConlon2}  and Candel \cite{Candel03}.
  The notions of directed positive $\ddc$-closed
  currents on singular holomorphic  foliations  and on singular laminations living  in a complex manifold are introduced  in the article  of Berndtsson-Sibony \cite{BerndtssonSibony} and the  survey of  Forn\ae ss-Sibony \cite {FornaessSibony08}
  respectively.
 We also  collect from \cite{DinhNguyenSibony12}  basic facts about  positive $\ddc$-closed  currents   on complex  manifolds. 
 The  sample-path space and the the holonomy  of a lamination  are  presented   in this  section. These typical  objects distinguish  the laminations  from the maps.
 A  short digression to the isolated singularities for singular holomorphic foliations is  given.
  Singular  holomorphic foliations by Riemann surfaces in $\P^k$ $(k>1)$ provide a large family of examples  where all the above notions  apply.
  We will  describe  the  properties of a generic  holomorphic foliation in $\P^k$  with  a  given degree $d>1.$
   The  section is ended with  a discussion on Sullivan's dictionary. This  is  a  kind of philosophical correpondence between the world of maps  (or roughly speaking, discrete dynamics)
   and  the world of laminations  (or more generally, continuous dynamics).
   
   \smallskip
   
     The Random Ergodic Theorem is presented in Section \ref{S:Random}.
    The first part of this  section   introduces the Wiener measure   for  the sample-path space  associated to a given point.  
    These play the  same role as  the counting measures do  for the   orbit of  a point in the context of map-iterations.
    The material for this  section is mainly  taken from  \cite{NguyenVietAnh17a}.
  
  \smallskip
  
  In the first part of Section  \ref{S:Regularity_Mass}
  we first will introduce  a function  $\eta$  which measures the ratio between the ambient metric and the leafwise  Poincar\'e metric of a lamination.
  This  function plays an important role in the study of  laminations.
  We  also  introduce  the class of  Brody hyperbolic  laminations.  This class contains not only all  compact laminations by hyperbolic Riemann surfaces,
  it also includes  many interesting singular  holomorphic  foliations.
  We then  state some recent results on the regularity  of Brody hyperbolic laminations which arise  from  our  joint-works with Dinh and Sibony in \cite{DinhNguyenSibony14a,DinhNguyenSibony14b}.  
  The second part of  Section  \ref{S:Regularity_Mass} is  devoted  to 
    the mas-distribution  of  directed  positive $\ddc$-closed  currents in both local and global settings.
    Understanding  the mass-distribution is one of the main  steps in  establishing ergodic theorems for singular holomorphic foliations.

    \smallskip
    
In Section \ref{S:Ergodic_theorems} we introduce the  abstract diffusions  of the  heat equation for two situations:

$\bullet$ Riemann surface laminations  (possibly
with singularities)   with  respect to a harmonic measure;

$\bullet$  for  holomorphically immersed Riemann surface laminations  (possibly
with singularities)    with  respect to a (not necessarily directed)  positive $\ddc$-closed current.

This approach  allows us in \cite{DinhNguyenSibony12} to extend the classical theory of  Garnett \cite{Garnett} and  Candel \cite{Candel03}
to Riemann surface laminations  (without or with singularities) or to singular holomorphic foliations with not
necessarily bounded geometry. Our  method is totally different from  those of  Garnett \cite{Garnett}, Candel-Conlon  \cite{CandelConlon2} and Candel \cite{Candel03}. We give two versions of
ergodic theorems for such currents: one associated to the abstract heat
diffusions and one of  geometric nature close to Birkhoff's averaging on orbits of a
dynamical system.   Another  consequence  of  this method  is  a  sufficient  condition  for the abstract  heat diffusions  to coincide with the leafwise heat diffusions (see Theorem \ref{T:diffusions_comparisons}).
This  result
and  its  consequences
are  new.

\smallskip

In Section   \ref{S:Unique_ergodicity}
we present some unique  ergodicity  theorems for compact Riemann surface laminations without singularities and for  singular holomorphic foliations.
In the first subsection we consider the case  when 
the lamination is compact and transversally conformal. We state a  unique  ergodicity theorem  due to Deroin-Kleptsyn \cite{DeroinKleptsyn} in this context.
The  second  subsection   is  devoted  to singular holomorphic  foliations in  $\P^2.$   
Then the works of Forn{\ae}ss-Sibony \cite{FornaessSibony10} and  Dinh-Sibony \cite{DinhSibony18}   describe   a dichotomous behaviour   of the unique  ergodicity    
when the singular holomorphic foliation  admits only hyperbolic  singularities. So  the panoramic picture in this special case is  rather  complete, 
at least when the singular holomorphic foliation  admits only hyperbolic  singularities.
In the last subsection  we state  our recent unique  ergodicity  theorems   for  Riemann surface laminations without singularities and for  singular holomorphic foliations
on compact K\"ahler surfaces. This result  is obtained in collaboration  with Dinh and  Sibony  \cite{DinhNguyenSibony18}.
Our  results  give  a  trichotomous behaviour in these general  settings.

\smallskip

In Section  \ref{S:Densities_Unique_Ergodicity} we   give a sketchy proof  of  our unique  ergodity theorems.  
We  outline  the  theory of densities for a class of  non $\ddc$-closed  currents developed in \cite{DinhNguyenSibony18}. This  theory is one of the main  ingredients
in our approach.


\smallskip

  Section \ref{S:Lyapunov} is devoted to the Lyapunov--Oseledec theory for  Riemann surface laminations (without or with  singularities).
Here  we deal with the (multiplicative) cocycles  which   are modelled  on the  holonomy cocycle of a  foliation.
The  Oseledec multiplicative ergodic theorem  for laminations is  the main result of this theory.
We  apply this  theorem  to  smooth compact laminations by hyperbolic Riemann surfaces   and to  compact singular 
holomorphic foliations. The material for this section  is taken from our memoir \cite{NguyenVietAnh17a}. 

\smallskip

Finally,   Section \ref{S:Applications} discusses  some  applications of the  theory developed here. We define and study  the canonical  Lyapunov exponents
for  a large  family of singular holomorphic foliations on compact projective surfaces. We also  study the topological  and algebro-geometric    interpretations of these
characteristic numbers. When  the lamination in question  is hyperbolic, smooth  and compact, we  characterize    geometrically the Lyapunov exponents of a  smooth cocycle with respect to a harmonic measure.
This  section  is based on   our   works in \cite{NguyenVietAnh17b,NguyenVietAnh18b,NguyenVietAnh18c}.

\smallskip

Several  open  problems develop in the  course of the  exposition.

\smallskip
\noindent
{\bf Main notation.} Throughout the paper, 
\begin{itemize}
\item[$\bullet$]
$\R^+$  (resp. $\N$) denotes  $[0,\infty)$  (resp. $\{n\in\Z:\ n\geq 0\}$);
\item[$\bullet$]
$\D$ denotes the unit disc
in $\C$, $r\D$ denotes the disc of center 0 and of radius $r,$ and
$\D_R\subset\D$ is the disc of center $0$ and of radius $R$ with
respect to the Poincar{\'e} metric $g_P$ on $\D$,
i.e. $\D_R=r\D$ with $R:=\log[(1+r)/(1-r)]$.
Recall  that 
$$
g_P(\zeta)={2\over (1-|\zeta|^2)^2} id\zeta\wedge d\bar\zeta\qquad\text{for}\qquad \zeta\in\D.
$$
\end{itemize}
Poincar{\'e} metric on
a hyperbolic Riemann surface, in particular on $\D$ and on the  hyperbolic leaves of a   Riemann  surface lamination, is
given by a positive $(1,1)$-form that we often denote by the  same symbol $g_P$. The associated distance  is denoted by $\dist_P.$ 

Given a  Riemann surface lamination with singularities $\Fc=(X,\Lc,E),$ $E$ is  the set of singularities, 
a leaf through a point $x\in X\setminus E$ is often denoted by $L_x,$ $\Hyp(\Fc)$ (resp. $\Par(\Fc)$) denotes the hyperbolic part (resp. the  parabolic part) of $\Fc.$  

 Recall that $i:=\sqrt{-1}$ and $\dc:={i\over2\pi}(\dbar-\partial)$ and $\ddc = {i\over \pi}\ddbar$.

\smallskip

\noindent{\bf Acknowledgement.} 
This paper is an expanded version of my talk  presented at the Vietnam Institute for Advanced Studies in Mathematics (VIASM) at the  Annual  Meeting on August  17, 2019.
I would like to thank   the
VIASM, and in particular Professor Ng\^o Bảo Ch\^au and Professor L\^e Minh H\`a, as well as the staff
of the Institute, for their invitation and warm hospitality.  I am also grateful to Professor Nguyễn  Tử Cường,  Professor  Ph\`ung Hồ Hải, Professor  Ph\d{a}m Ho\`ang Hi\d{\^e}p  and Professor Đỗ Đức Thái
for 
their kind support.   It is a pleasure to thank   Professor  Nessim Sibony and Professor Đinh  Tiến Cường  for their long-term collaboration with me. Finally,   I  am thankful   to   
Dr. Fabrizio Bianchi  for drawing graciously  the figures.
\section{Basic laminations and  foliations concepts} \label{S:basics}


\subsection{Riemann surface laminations and singular foliations}\label{SS:RSL}

Let $X$ be a locally compact space.  A {\it   Riemann surface lamination}     $(X,\Lc)$   is  the  data of  a {\it (lamination)  atlas} $\Lc$ 
of $X$ with (laminated) charts 
$$\Phi_p:\U_p\rightarrow \B_p\times \T_p.$$
Here, $\T_p$ is a locally compact  metric space, $\B_p$ is a domain in $\C$,  $\U_p$ is  an open set in 
$X,$ and  
$\Phi_p$ is  a homeomorphism,  and  all the changes of coordinates $\Phi_p\circ\Phi_q^{-1}$ are of the form
$$x=(y,t)\mapsto x'=(y',t'), \quad y'=\Psi(y,t),\quad t'=\Lambda(t),$$
 where $\Psi,$ $\Lambda$ are continuous  functions and $\Psi$ is  holomorphic in  $y.$


The open set $\U_p$ is called a {\it flow
  box} and the Riemann surface $\Phi_p^{-1}\{t=c\}$ in $\U_p$ with $c\in\T_p$ is a {\it
  plaque}. The property of the above coordinate changes insures that
the plaques in different flow boxes are compatible in the intersection of
the boxes. Two plaques are {\it adjacent} if they have non-empty intersection.
A  {\it  transversal} in a flow box is  a closed  set of the box which intersects every plaque  in one point.
In particular, $\Phi_p^{-1}(\{x\}\times \T_p)$ is a transversal in $\U_p$ for any point $x\in\B_p.$
For the sake of simplicity, we often identify $\T_p$  with $\Phi_p^{-1}(\{x\}\times \T_p)$  for some $x\in\B_p,$ or even identify $\U_p$   with $\B_p\times\T_p$ via the map $\Phi_p.$ 
 
A {\it leaf} $L$ is a minimal connected subset of $X$ such
that if $L$ intersects a plaque, it contains that plaque. So a leaf $L$
is a  Riemann surface  immersed in $X$ which is a
union of plaques. For every point $x\in X,$  denote  by  $L_x$   the   leaf passing  through $x.$  
  A subset $M\subset X$  is  called
  {\it leafwise  saturated} if $x\in M$ implies  $L_x\subset M.$

 We say  that a Riemann  surface lamination  $(X,\Lc)$ is {\it   $\Cc^k$-smooth}  (resp.   {\it smooth}) if
  each map $\Psi$ above is    $\Cc^k$-smooth (resp.  smooth)   with
respect to $y,$ and its partial derivatives of   any total order $\leq k$   (resp. any order)  with respect to $y$ and $\bar y$ are jointly continuous
with respect  to $(y,t).$
 
We are mostly interested in the case where the $\T_p$ are closed subsets of smooth real manifolds (resp. of some  complex manifolds)  and the functions $\Psi,\Lambda$ are $\Cc^k$-smooth (resp. smooth, 
holomorphic) in all variables. 
In this case, we say that the lamination $(X,\Lc)$ is  {\it $\Cc^k$-transversally smooth} (resp. {\it transversally smooth,}  {\it transversally holomorphic}).
If, moreover, $X$ is compact, we can embed it in an $\R^N$ in order to use the distance induced by a Riemannian metric on $\R^N$. 

 We say that a transversally smooth Riemann surface  lamination $(X,\Lc)$ is a  {\it smooth foliation} if
 $X$ is a  manifold and all leaves of $\Lc$ are Riemann surfaces immersed in $X.$

We  say that  a Riemann  surface lamination $(X,\Lc)$ is a {\it holomorphic  foliation} 
  if $X$ is a complex manifold (of dimension $k$) and
there is an atlas $\Lc$ of $X$ with (foliated) charts 
$$\Phi_p:\U_p\rightarrow \B_p\times \T_p,$$
where the $\T_p$'s  are open sets of  $\C^{k-1}$ and 
 all above maps  $\Psi,\Lambda $ are   holomorphic.

We call {\it Riemann surface lamination with singularities} the data
$\Fc=(X,\Lc,E)$ where $X$ is a locally compact space, $E$ a closed
subset of $X$ such that
 $\overline{X\setminus E}=X$  and $(X\setminus E,\Lc)$ is a Riemann surface
lamination. The set $E$ is {\it the singularity set} of the lamination.

\begin{figure}[h]%
\begin{center}
\def\svgwidth{0.6\columnwidth}
\resizebox{0.5\textwidth}{!}{\input{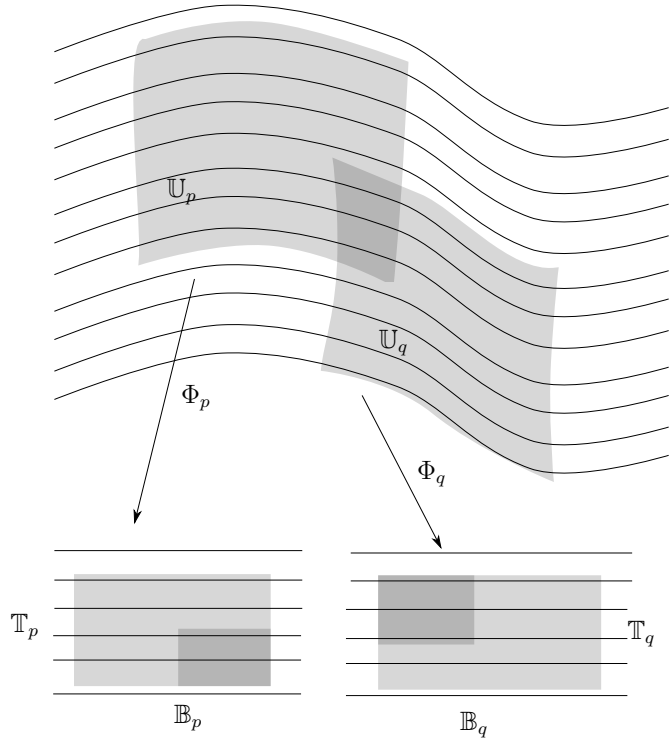}}%
\caption{Laminated charts $\Phi_p$ and $\Phi_q$ of a lamination.}
\label{fig:lamination}
\end{center}
\end{figure}

 We say that $\Fc:=(X,\Lc,E)$ is  a {\it singular  foliation} (resp.   {\it singular holomorphic foliation}) if $X$ is   a manifold (resp. a complex manifold) and $E\subset X$ is 
 a closed subset such that
 $\overline{X\setminus E}=X$ and  $(X\setminus E,\Lc)$ is a  smooth foliation (resp. a  holomorphic  foliation). $E$ is  said to be the  {\it set of singularities} of the  foliation $\Fc.$ 
 We  say that $\Fc$ is  compact if  $X$ is compact.

 Many examples of abstract compact Riemann  surface laminations are constructed in 
\cite{CandelConlon1} and \cite{Ghys}. Suspensions of  group actions give already a large class
of laminations without singularities.

 \begin{remark}\label{R:m-dim_laminations}
  \rm  In the  above definitions,  if we allow  $\B_p$ to be  a domain in $\R^N$ (resp. in $\C^N$) for 
  a fixed number $N\in\N$, then we obtain $N$-real (resp. $N$-complex) dimensional laminations/foliations.
  For a comprehensive recent exposition on $N$-dimensional laminations/foliations, the reader is invited to consult
  the textbooks by Candel-Conlon \cite{CandelConlon1,CandelConlon2}, by Walczak \cite{Walczak} etc.
 \end{remark}

\subsection{Hyperbolicity and leafwise Poincar\'e metric}\label{SS:Poincare}
 
 Let  $L$ be  an arbitrary,  not necessarily simply-connected, Riemann surface.
 Riemann's mapping theorem states that  its universal covering surface $\widetilde L,$ which is always simply connected, can be mapped conformally onto  
  a domain of exactly one of the following types:
 \begin{itemize}
 \item  the extended complex plane $\C\cup\{\infty\}=\P^1$ (the case of a Riemann surface of elliptic type);
 
 \item the finite complex plane $\C$ (is of parabolic type); 
 
 \item or — the unit disc $\D$ ( is of hyperbolic type). 
 \end{itemize}
 We  say  that $L$ is {\it uniformized } by the  corresponding  domain of  its type.
 Since the elliptic case differs from the others already from the topological point of view, the difficult problem of recognizing whether a given Riemann surface 
 is of hyperbolic or parabolic type is still left.  
 It is known that a closed Riemann surface of genus $g$ for $g=0$ is of elliptic type, for $g=1$ it is of parabolic type, and for $g>1$ of hyperbolic type; therefore,
 the problem of types is mainly important for open Riemann surfaces.

 Consider now a  Riemann surface lamination with singularities $\Fc=(X,\Lc,E)$. 
 \begin{definition}\label{D:hyperbolic}\rm 
A leaf $L$  of $\Fc$ is  said to be  {\it hyperbolic} if
it  is a   hyperbolic  Riemann  surface, i.e.  it is  of hyperbolic type, i.e. it is  uniformized   by 
$\D.$  Otherwise (i.e. when $L$ is  either of parabolic type or of elliptic type), $L$ is  called  {\it  parabolic.}   

$\Fc$   is  said to be {\it hyperbolic} if  
  its leaves   are all  hyperbolic. 
  
  The {\it hyperbolic part} of the lamination $\Fc$, denoted by $\Hyp(\Fc),$ is  the union of all hyperbolic leaves, whereas the union of all  parabolic leaves, denoted by $\Par(\Fc),$ 
  is  called the {\it parabolic part} (or equivalently, the 
  {\it non-hyperbolic part} of $\Fc$).
  These  are disjoint leafwise  saturated measureable sets of $X $ and $\Hyp(\Fc)\cup \Par(\Fc)=X\setminus E$    (see \cite[Proposition 3.1]{DinhNguyenSibony12}).
\end{definition}

  Consider also a Hermitian metric on $X$, i.e. Hermitian metrics on the
leaves of $L$ whose restriction to each flow box defines Hermitian
metrics on the plaques that depend continuously on the plaques. It is
not difficult to construct such a metric using a partition of
unity. Observe that all the Hermitian metrics on $X$ are locally
equivalent. From now on, fix a Hermitian metric on $X$.  

For every $x\in \Hyp(\Fc),$  consider a universal covering map
\begin{equation}\label{e:covering_map}
\phi_x:\ \D\rightarrow L_x\qquad\text{such that}\  \phi_x(0)=x.
\end{equation}
 This map is
uniquely defined by $x$ up to a rotation on $\D$. 
Then, by pushing   forward  the Poincar\'e metric $g_P$
on $\D$  
  via $\phi_x$ (see  Main notation), we obtain the  so-called {\it Poincar\'e metric} on $L_x$ which depends only on the leaf.  
  The latter metric is given by a positive $(1,1)$-form on $L_x$  that we also denote by $g_P$ for the sake of simplicity. 
So $g_P$ is a measurable  $(1,1)$-form defined on $\Hyp(\Fc)$  \cite[Proposition 3.1]{DinhNguyenSibony12}. 
 For a systematic exposition on the  Poincar\'e metric and its  generalizations,  see the book by Kobayashi \cite{Kobayashi}.
  
 \subsection{Leafwise heat diffusions and harmonic  measures}
 \label{ss:heat_diffusions}
     
Let  $\Fc=(X,\Lc,E)$ be  a  Riemann surface  lamination with singularities.  
     The leafwise Poincar\'e metric
$g_P$  induces  the corresponding 
Laplacian $\Delta_P$  on  hyperbolic leaves   (see formula \eqref{e:Laplacian_disc} below  for  the case of the Poincar\'e disc $(\D,g_P)$  and formula \eqref{e:Delta_commutation} for the case of a
hyperbolic leaf).  
 For  every point  $x\in \Hyp(\Fc),$
 consider  the   {\it heat  equation} on $L_x$
 \begin{equation}\label{e:heat-equation}
 {\partial p(x,y,t)\over \partial t}=\Delta_{P,y} p(x,y,t),\qquad  \lim_{t\to 0+} p(x,y,t)=\delta_x(y),\qquad   y\in L_x,\ t\in \R^+.
 \end{equation}
Here   $\delta_x$  denotes  the  Dirac mass at $x,$ $\Delta_{P,y}$ denotes the  Laplacian  $\Delta_P$ with respect to the  variable $y,$
 and  the  limit  is  taken  in the  sense of distribution, that is,
$$
 \lim_{t\to 0+
}\int_{L_x} p(x,y,t) f(y) g_P( y)=f(x)
$$
for  every  smooth function  $f$   compactly supported in $L_x.$   

The smallest positive solution of the  above  equation, denoted  by $p(x,y,t),$ is  called  {\it the heat kernel}. Such    a  solution   exists   because  $(L_x,g_P)$ is
complete and   of bounded  geometry  (see, for example,  \cite{CandelConlon2,Chavel}).  
 The  heat kernel  $p(x,y,t)$  gives  rise to   a one  parameter  family $\{D_t:\ t\geq 0\}$ of  {\it leafwise heat diffusion  operators}    defined on bounded measurable functions  on $\Hyp(\Fc)$ by
 \begin{equation}\label{e:diffusions}
 D_tf(x):=\int_{L_x} p(x,y,t) f(y) g_P (y),\qquad x\in \Hyp(\Fc).
 \end{equation}
 This  family is  a  semi-group, that is,  
\begin{equation}\label{e:semi_group}
D_0=\id\quad\text{and}\quad   D_t \mathbf{1} =\mathbf{1}\quad\text{and}\quad D_{t+s}=D_t\circ D_s \quad\text{for}\ t,s\geq 0,
\end{equation}
where $\mathbf{1}$ denotes the function which is identically equal to $1.$

We also  denote by  $\Delta_P$ the  Laplacian on the Poincar\'e disc $(\D,g_P),$ that is, for every function $f\in \Cc^2(\D),$
\begin{equation}\label{e:Laplacian_disc}
(\Delta_P  f) g_P=\pi\ddc f=i\ddbar f \qquad\text{on}\ \D.
\end{equation}
Let $\dist_P$ denote the Poincar\'e distance on $(\D,g_P).$
 For $\zeta\in\D$   write  $\rho:=\dist_P(0,\zeta).$ So
$$\rho:=\log {1+|\zeta|\over 1-|\zeta|}\cdot$$
Recall a formula in Chavel \cite[p.246]{Chavel} for the  heat kernel of the Poincar\'e disc $(\D,g_P):$
\begin{equation} \label{e:heat_kernel}p_\D(0,\zeta,t)={\sqrt{2}e^{-t/4} \over (2\pi t)^{3/2}}
\int_{\rho}^{\infty}{s  e^{-{s^2 \over 4t}}\over \sqrt{\cosh s-\cosh\rho}}ds.
\end{equation}
For every function $f\in \Cc^1(\D),$ we also denote by $|df|_P$  the length  of  the differential $df$ with respect to $g_P,$
that is,   $|df|_P=|df|\cdot g_P^{-1/2}$ on $\D,$ where $|df|$ denotes the Euclidean norm of $df.$  


Let $x\in \Hyp(\Fc).$ We often identify  the fundamental group $\pi_1(L_x)$ of  $L_x$  with the group of  
deck-transformations
   of the  universal covering map
$\phi_x:\  \D\to L_x$  given in  (\ref{e:covering_map}). 
It is  well-known that  $\pi_1(L_x)$ is  at most  countable.  

 The  Laplace operator $\Delta_P$ on the leaf $(L_x,g_P|_{L_x})$ lifts, via $\phi_x,$
to $ \Delta_P$ on the Poincar\'e disc $(\D,g_P).$ More  precisely,   
for   $x\in \Hyp(\Fc)$ and for every  $\Cc^2$-smooth   function $f$ defined on $L_x,$  
\begin{equation}\label{e:Delta_commutation}
 \Delta_P(f\circ \phi_x)=(\Delta_P f)\circ\phi_x, \qquad \textrm{on}\qquad  \D. 
\end{equation}
The heat kernel $p(x,y,t)$  for $(L_x,g_P)$  is  related   to $p_\D(\tilde x,\tilde y,t)$  for  $(\D,g_P)$ by
\begin{equation}\label{e:heat_kernel2}
p(x,y,t)=\sum_{\gamma \in \pi_1(L_x)} p_\D(\tilde x, \gamma \tilde y,t) ,
\end{equation}
where $\tilde x$ (resp.  $\tilde  y$) is  a preimage of   $x$ (resp. $y$)  by the map $\phi_x.$
Moreover, 
$p_\D$ is  invariant  under deck-transformations,
 that is,
\begin{equation}\label{e:heat_kernel3}
 p_\D(\gamma\tilde x, \gamma \tilde y,t)=  p_\D(\tilde x,  \tilde y,t)
\end{equation}
for  all $\gamma \in \pi_1(L_x)$ and $\tilde x,\tilde y\in \D$ and $t\geq 0.$
As an immediate consequence of identity (\ref{e:heat_kernel2}), we obtain the following relation  between $D_t$ given in  \eqref{e:diffusions} and the heat
diffusions $ D_t$ on $\D.$ 
For   $x\in X$ and for every  bounded measurable  function $f$ defined on $L_x,$   we have 
\begin{equation}\label{e:commutation}
 D_t(f\circ \phi_x)=(D_tf)\circ\phi_x, \qquad \textrm{on $\D$  for all $t\in\R^+.$} 
\end{equation}
See \cite[Proposition 2.7]{NguyenVietAnh17a} for a proof.

 \subsection{Directed differential forms, directed positive forms, directed currents   and harmonic  measures}
 \label{ss:forms_measures}

We recall now the notion of currents on a manifold.
Let $M$ be a real oriented manifold of dimension $m$ (resp. a complex manifold of dimension $k$). We fix an atlas
of $M$ which is locally finite. Up to reducing slightly the charts, we
can assume that the local coordinate system associated to
each chart is defined on a neighbourhood of the closure of this chart. For $0\leq p
\leq m$ (resp. for $0\leq p,q\leq k$) and $l\in\N$, denote by $\Dc^p_l(M)$ (resp. $\Dc^{p,q}_l(M)$) the space of $p$-forms  (resp. $(p,q)$-forms)  of
class $\Cc^l$ with compact support in $M,$ and $\Dc^p(M)$ (resp.  $\Dc^{p,q}(M)$) their
intersection for $l\in\N.$ If $\alpha$ is a $p$-form (resp.  $(p,q)$-form) on $M$, denote by
$\|\alpha\|_{\Cc^l}$ the sum of the $\Cc^l$-norms of the coefficients
of $\alpha$ in the local coordinates. These norms induce a
topology on $\Dc^p_l(M)$ and $\Dc^p(M)$ (resp.  on $\Dc^{p,q}_l(M)$ and $\Dc^{p,q}(M)$). In particular, a
sequence $\alpha_j$ converges to $\alpha$ in $\Dc^p(M)$ (resp. in  $\Dc^{p,q}(M)$) if
these forms are supported in a fixed compact set and if
$\|\alpha_j-\alpha\|_{\Cc^l}\rightarrow 0$ for every $l$. 

 Let $M$ be a real oriented manifold of dimension $m.$ A {\it current of degree $p$} (or equivalently,   a {\it current  of dimension $m-p$} on $M,$ or a {\it $p$-current}
for short) is a continuous linear form
$T$ on $\Dc^{m-p}(M)$ with values in $\C$. 
The value of $T$ on a test form $\alpha$ in  $\Dc^{m-p}(M)$ is
denoted by $\langle T,\alpha\rangle$ or $T(\alpha)$.
The current $T$ is
of order $\leq l$ if it can be extended to a continuous linear form on 
$\Dc^{m-p}_l(M)$. The order of $T$ is the minimal integer $l\geq 0$ satisfying this condition.
It is not difficult to see that the restriction of $T$ to a relatively compact open
set of $M$ is always of finite order. Define
$$\|T\|_{-l,K}:=\sup\Big\{|\langle T,\alpha\rangle|,\quad
\alpha\in\Dc^{m-p}(M),\quad \|\alpha\|_{\Cc^l}\leq 1,\quad
\supp(\alpha)\subset K\Big\}$$
for $l\in\N$ and $K$ a compact subset of $M$. 
This quantity may be infinite when the order of $T$ is larger than $l$. 

Let $M$ be  a complex manifold of dimension $k.$ 
 A {\it $(p,q)$-current} on  $M$ (or equivalently, a {\it current  of bidegree $(p,q),$} or equivalently,   a {\it current  of bidimension  $(k-p,k-q)$}) is
 a  continuous linear form $T$  on $\Dc^{k-p,k-q}(M)$ with values in $\C.$ 

Consider now a Riemann surface lamination with singularities $\Fc=(X,\Lc,E).$
The notion of differential forms on manifolds can be extended
to laminations, see Sullivan \cite{Sullivan}. A {\it (directed) $p$-form} (resp. a {\it (directed) $(p,q)$-form}) on $\Fc$ can be seen on the flow box
$\U\simeq \B\times\T$ as
a $p$-form  (resp.  $(p,q)$-form) on $\B$ depending on the parameter $t\in \T$. 
For $0\leq p\leq 2$ (resp. for $0\leq p,q\leq 1$),  denote by $\Dc^p_l(\Fc)$  (resp. $\Dc^{p,q}_l(\Fc)$) the space of 
$p$-forms (resp.  $(p,q)$-form) $\alpha$ with compact support in $X\setminus E$ satisfying the following property: 
$\alpha$ restricted to each flow box $\U\simeq \B\times\T$ is a 
$p$-form (resp.  $(p,q)$-form) of class $\Cc^l$ on the plaques whose coefficients and
all their derivatives up to order $l$ depend continuously on the
plaque. The norm $\|\cdot\|_{\Cc^l}$ on
this space is defined as in the case of real manifold using a locally finite
atlas of $\Fc$. We also define $\Dc^p(\Fc)$ (resp.  $\Dc^{p,q}(\Fc)$)  as the intersection of
$\Dc^p_l(\Fc)$ (resp.  $\Dc^{p,q}_l(\Fc)$) for $l\geq 0$.
A {\it (directed) current of degree $p$} (or equivalently, a {\it  (directed) current  of dimension $2-p$})  
on $\Fc$ is a continuous
linear form on the space $\Dc^{2-p}(\Fc)$ 
with values in $\C$. A
$p$-current is of order $\leq l$ if it can be extended to a linear
continuous form on $\Dc^{2-p}_l(\Fc)$. The restriction of a current to a
relatively compact open set of $X\setminus E$ is always of finite order. 
The norm $\|\cdot\|_{-l,K}$ on currents is defined as in the case of
manifolds.  
  We often write for short $\Dc(\Fc)$ instead of the space of functions $\Dc^0(\Fc).$
  A (directed) $(p+q)$-current is said to be  of
{\it bidegree $(p,q)$} (or equivalently,  of bidimension $(1-p,1-q)$)  if it vanishes on forms of bidegree $(1-p',1-q')$ for $(p',q')\not=(p,q).$
  
   A form $\alpha \in \Dc^{1,1}(\Fc)$ is  said to be {\it positive} if its restriction to every plaque
 is  a  positive measure in the  usual  sense, that is, in every flow box  $\U\simeq \B\times \T,$
$$
\alpha(z,t)=a(z,t)i dz\wedge  d\bar z\quad\text{for}\quad  (z,t)\in \B\times \T,
$$
 where $a$ is  a  nonnegative-valued function.

 \begin{definition}\label{D:harmonic_measure}\rm 
 Let  $\Delta_P$  be   the  Laplacian  on $\Fc,$ that is, the  aggregate  of the  leafwise Laplacians  $\{\Delta_{P,x}\},$  where  $x\in\Hyp(\Fc)$ 
 (see \eqref{e:Laplacian_disc} and \eqref{e:commutation}).
 Let  $\mu$ be a  locally finite\footnote{ A  real-valued signed Borel  measure $\mu$ on a topological space $X$  is  said to be  {\it locally finite} if   every point $x\in X$   has a neighbourhood 
 $\U_x$
 such that  $|\mu|(\U_x)$ is  finite,  where $|\mu|$  is  the variation of $\mu.$ }
  real-valued signed  Borel measure  on $X$  whose variation  $|\mu|$   gives no mass to $\Par(\Fc)\cup E.$
 \begin{enumerate}
\item $\mu$ is {\it quasi-harmonic}  if
$$
\int_X  \Delta_P f \,d\mu=0
$$ 
 for all  functions  $f\in \Dc(\Fc).$
 
\item  $\mu$  is   called {\it very weakly   harmonic} 
(resp.  {\it weakly harmonic}) if  $\mu$ is finite positive and  the following    property  is satisfied
for $t=1$ (resp.  for all $t\in\R^+$):
$$\int_X  D_t f d\mu=\int_X fd\mu
$$   
for all  bounded  measurable functions $f$ defined on $X.$

 \item  $\mu$ is  said to be {\it harmonic }  if it is both weakly  harmonic and   quasi-harmonic.  
 \end{enumerate}
\end{definition}


\subsection{Directed positive harmonic currents vs harmonic  measures}\label{SS:Directed_hamonic_currents}

  Let $\Fc=(X,\Lc,E)$ be  a Riemann surface lamination with singularities.
  For each chart $\Phi:\U\rightarrow \B\times \T,$ the  complex   structure on $\B$ induces  a complex  structure on the leaves of $X.$
Therefore,  the operator $\partial,$ $\dbar,$ $d$  and $\dc$  can be  defined  so that they act leafwise on  forms
as in the case of complex manifolds.  Let
$T$ be  a directed $(p,q)$-current, then $\partial T$ and $\dbar T$
are defined  as follows.

If $T$  is a  $(0,q)$-current, then
$$\langle\partial T,\alpha\rangle:=(-1)^{q+1}\langle T,\partial \alpha\rangle
\quad \mbox{for all test forms } \alpha\in \Dc^{0,1-q}(\Fc).$$

If $T$ is a  $(1,q)$-current, then $\partial T:=0.$

If $T$ is a $(p,0)$-current, then 
$$\langle\dbar T,\alpha\rangle:=(-1)^{p+1}\langle T,\dbar \alpha\rangle
\quad \mbox{for all test forms } \alpha\in \Dc^{1-p,0}(\Fc).$$

If $T$ is a  $(p,1)$-current, then $\dbar T:=0.$

 So we get easily   that
 \begin{eqnarray*}d&:=&\partial +\dbar:\  \Dc^p(\Fc)\to\Dc^{p+1}(\Fc)\quad \text{for}\quad 0\leq p\leq 2\quad \text{with}\quad \Dc^3(\Fc)\equiv \{0\};\\ 
\ddc&:=& {i\over \pi}\ddbar:\  \Dc(\Fc)\to\Dc^{1,1}(\Fc).
\end{eqnarray*}
 \begin{definition}\label{D:Directed_hamonic_currents} {\rm (Garnett \cite{Garnett}, see also  Sullivan \cite{Sullivan}).}
\rm  Let $T$ be a directed current of bidimension $(1,1)$ on $\Fc.$ 
    
  $\bullet$  $T$ is  said to be  {\it positive} if  $T(\alpha)\geq 0$ for all positive forms $\alpha\in \Dc^{1,1}(\Fc).$

$\bullet$ $T$ is  said to be  {\it   closed}   if  $d T=0$ in the  weak sense (namely,  $T(d f)=0$ for all directed forms  $f\in  \Dc^1(\Fc)$).
    
$\bullet$ $T$ is  said to be  {\it   harmonic}   if  $\ddc T=0$ in the  weak sense (namely,  $T(\ddc f)=0$ for all functions  $f\in  \Dc(\Fc)$).
\end{definition}
  We have the  following decomposition.
  \begin{proposition}{\rm  \cite[Proposition  2.3]{DinhNguyenSibony12} and \cite[Proposition 2.10]{NguyenVietAnh17a}}\label{P:decomposition}
   Let $T$ be  a directed harmonic current on $\Fc.$  Let $\U\simeq \B\times \T$ be  a flow  box which is relatively compact in $X.$
   Then, there is a positive Radon measure $\nu$ on $\T$ and for $\nu$-almost every $t\in\T,$ there is a  harmonic   function $h_t$ on $\B$
   such that  if $K$  is compact in $\B,$  the integral $\int_K \|h_t\|_{L^1(K)} d\nu(t)$ is  finite and
   $$
   T(\alpha)=\int_\T\big( \int_\B  h_t(y)\alpha(y,t)  \big)d\nu(t)
   $$
   for every  form $\alpha\in\Dc^{1,1}(\Fc)$ compactly  supported  on $\U.$ 
   If  moreover, $T$ is  positive, then    for $\nu$-almost every $t\in\T,$ the  harmonic   function $h_t$ is  positive on $\B.$
    If  moreover, $T$ is  closed, then    for $\nu$-almost every $t\in\T,$ the  harmonic   function $h_t$ is  constant on $\B.$
  \end{proposition}
  
 \begin{definition} \label{D:diffuse}\rm 
  A directed positive harmonic current $T$ on $\Fc=(X,\Lc,E)$  is said to be  {\it diffuse}  if   for any  decomposition of $T$ in any flow box  $\U\simeq \B\times \T$
  as in Proposition \ref{P:decomposition}, the  measure $\nu$ has no mass on each single point of the transversal $\T.$
 \end{definition}

  \begin{definition}\label{D:ergodicity}\rm
Recall that a positive finite  measure  $\mu$ on the $\sigma$-algebra of Borel sets in $X$ with $\mu(E)=0$  is  said  to be  {\it ergodic}
if for every  leafwise  saturated  Borel measurable set $Z\subset X,$
   $\mu(Z)$ is  equal to either $\mu(X)$ or $0.$  
   
   A   directed positive harmonic  current $T$ is said to be {\it extremal}
   if  $T=T_1+T_2$ for   directed positive harmonic  current $T_1,T_2$ implies that $T_1=\lambda T$  for some $ \lambda \in  [0, 1].$  
   
   Similarly,  a positive   quasi-harmonic  measure $\mu$ of finite mass is said to be {\it extremal}
   if  $\mu=\mu_1+\mu_2$ for positive   quasi-harmonic  measures $\mu_1,\mu_2$ implies that $\mu_1=\lambda \mu$  for some $ \lambda \in  [0, 1].$  
   \end{definition}
   \begin{definition}\rm \label{D:Phi}
  Let $\Fc=(X,\Lc,E)$ be  a  Riemann surface  lamination with singularities.  Consider  the map
$T\mapsto \Phi(T):= \mu$  which is defined  by the following  formula on the  convex  cone of all  directed positive harmonic  currents $T$ of $\Fc$ giving  no mass to $\Par(\Fc)$ 
 \begin{equation}\label{e:mu}  \mu :=T\wedge g_P\quad\text{on}\quad  X\setminus (E\cup\Par(\Fc))\quad\text{and}\quad  \mu(E\cup\Par(\Fc))=0.
  \end{equation}
   We call {\it Poincar{\'e}  mass} of $T$ the mass of $T$ with respect to
Poincar{\'e}  metric $g_P$ on $X\setminus E$, i.e. the mass of the
positive measure $\mu=\Phi(T)$.
\end{definition}
A priori, Poincar{\'e}  mass may be
infinite near the singular points. 
The  following result, which is  implicitly proved in  \cite{DinhNguyenSibony12}, relates the notions of harmonic  measures and  directed positive harmonic currents 
(see also \cite{NguyenVietAnh18b}). 
\begin{theorem}\label{thm_harmonic_currents_vs_measures} 
We keep the  assumption and notation of Definition \ref{D:Phi}.
\begin{enumerate}
\item 
 $\Phi$ is  a bijection   from the  convex  cone of  all  directed positive harmonic  currents $T$  of $\Fc$ giving no mass to $\Par(\Fc)$  onto   
the convex  cone of all positive quasi-harmonic  measures $\mu$. 


\item
If $T$ is a    directed positive harmonic current such that $T$ gives no mass to $\Par(\Fc)$ and that its Poincar\'e mass  is  finite, then
$\mu:=\Phi(T)$  is an extremal positive quasi-harmonic measure iff $\mu$ is  ergodic
iff $T$ is   extremal.

\item 
  A positive quasi-harmonic    measure is  harmonic if and only if it is  finite.
\end{enumerate}
\end{theorem} 

Assertions (1) and (2) follows  from the definitions.

By definition, a  harmonic measure  is necessarily  finite. Therefore,
to complete  the proof of assertion (3), we need to show  that  a finite positive quasi-harmonic measure is weakly harmonic. 

For this  purpose, we introduce the following
 operator $A_R:\ L^\infty(\Hyp(\Fc))\to  L^\infty(\Hyp(\Fc))$ given by
\begin{equation} \label{e:A_R} A_Ru(x):=\frac{1}{\mathcal M_R}\int_{\D_R} (\phi_x)^*
(u g_P) \quad\text{for}\quad  x\in\Hyp(\Fc),\quad \mbox{where}\quad \mathcal M_R:=\int_{\D_R} (\phi_x)^*(g_P).
\end{equation}
Note that $\mathcal M_R$ is the Poincar{\'e} area of $\D_R$ which is also the Poincar\'e area of $\phi_x(\D_R)$ counted with multiplicity. 
It is  immediate from  the definition that the norm of $A_R$ is equal to $ 1.$
Since $\mu(E\cup\Par(\Fc))=0,$ we extends the domain of definition  of $A_R$ in  \eqref{e:A_R} in a natural way  so that  $A_R:\ L^\infty(\mu)\to  L^\infty(\mu)$  with norm $1.$

The following  result is  needed.
\begin{lemma}\label{L:mu_Dt_invariant} If  $\mu$ is a  finite positive quasi-harmonic measure, then 
\begin{equation*}
\int (A_R u) d\mu =\int u d\mu\quad\text{for}\quad u\in L^1(\mu). 
\end{equation*}
\end{lemma}
Taking  for granted  the above lemma, we arrive at the
\proof[End  of the proof of  Theorem \ref{thm_harmonic_currents_vs_measures}.] 
Let $\mu$ be a    finite  positive quasi-harmonic measure. We  only need  to show  that $\mu$ is  weakly harmonic. Fix  an  arbitrary function $u$  in $\Dc(\Fc)$ and  a  time $t>0.$
So  it is sufficient  to show that 
$$
\int_X D_t u d\mu=\int_X ud\mu.
$$
For $z\in \D,$  write $r=|z|$ and  let  $R$  be defined  by  \eqref{e:radii_conversion}, that is,  $\D_R=r\D.$
Write the Poincar\'e metric in $\D$ as  follows:
\begin{equation}\label{e:polar-form}
dg_P(z)=d\sigma_R(z)dR,
\end{equation}
where  $d\sigma(z)$ is  the Poincar\'e-length  form  on  $\partial\D_R$  which is  the  circle of  center $0$ and of  Poincar\'e radius $R.$
It follows  from   \eqref{e:polar-form}  that the derivative  of $\mathcal M_R$ with respect to $R$  is  equal to
$$
\mathcal M'_R:=\lim\limits_{s\to 0}  {\mathcal M_{R+s} -\mathcal M_R\over s}=\int_{\partial \D_R}d\sigma_R(z) 
$$
Since
$p_\D(0,\cdot,t)$ is a radial  positive function  smooth  on $\D\setminus \{0\}$ and it satisfies
$\int_\D p_\D(0,y,t) g_P (y)=1$  for every $t\in\R^+_*,$  we  infer from  the last  line and \eqref{e:polar-form} that
\begin{equation}\label{e:M'_R=1}
\int_0^\infty    p_\D(0,r_R,t)\mathcal M'_R dR=1,
\end{equation}
where $R$ and $r_R$ are related by \eqref{e:radii_conversion}.
On the  one hand, since $u\in\Dc(\Fc),$  we deduce  from  \eqref{e:polar-form} that for $x\in\Hyp(\Fc),$ the  derivative of $(\mathcal M_R  A_R)  u(x)$ with respect to $R$ is  equal to
\begin{eqnarray*}
(\mathcal M_R  A_R)'  u(x)&:=&\lim\limits_{s\to 0}  {\mathcal M_{R+s} A_{R+s} u(x) -\mathcal M_R A_R u(x)\over s}=\lim\limits_{s\to 0}{1\over s}\int_{ \D_{R+s}\setminus \D_R}\phi_x^*(ug_P) \\
&=&\int_{ \partial \D_R}\phi_x^*(u \sigma_R).
\end{eqnarray*} 
Moreover,  by Lemma \ref{L:mu_Dt_invariant}, 
\begin{eqnarray*}
\int_X(\mathcal M_R  A_R)'  u(x)d\mu(x)&=&\lim\limits_{s\to 0}  {\int_X \mathcal M_{R+s} A_{R+s} u(x)d\mu(x) -\int_X\mathcal M_R A_R u(x)d\mu(x)\over s}\\
&=&
\lim\limits_{s\to 0}{\mathcal M_{R+s}- \mathcal M_R\over s}\cdot\,\int_X u(x)d\mu(x)\\
&=&  \mathcal M'_R  \int_X u(x)d\mu(x).
\end{eqnarray*}
On  the other hand, by  \eqref{e:commutation}, we have 
$$
(D_tu)(x):=\int_\D  p_\D(0,\cdot,t) (u\circ \phi_x) g_P\qquad\text{for}\qquad x\in \Hyp(\Fc).
$$ 
Therefore,  using    \eqref{e:polar-form} again  and  the   last expression for $(\mathcal M_R  A_R)'  u(x),$ we get that
$$
(D_tu)(x)=\int_0^\infty    p_\D(0,r_R,t)  \big( \int_{ \partial \D_R}\phi_x^*(u \sigma_R) \big) dR=\int_0^\infty    p_\D(0,r_R,t)  \big( (\mathcal M_R  A_R)'  u(x) \big) dR.
$$  
Integrating  the last equalities with respect to $d\mu$ and using Fubini theorem,  
we obtain 
$$
\int_X (D_tu)(x)d\mu(x)= \int_0^\infty    p_\D(0,r_R,t)  \big(\int_X (\mathcal M_R  A_R)'  u(x)d\mu(x) \big) dR.
$$
Using  the last expression  for   the inner  integral on   the  right hand side,  
it follows that
$$
\int_X (D_tu)(x)d\mu(x)= \big(\int_0^\infty    p_\D(0,r_R,t) \mathcal M'_RdR\big)  \big(\int_X  u(x)d\mu(x) \big).
$$
By  \eqref{e:M'_R=1}, the right hand side  is  equal to $\int_X  u(x)d\mu(x).$ 
Hence, 
$\mu$ is a  weakly  harmonic  measure.
\endproof

In order to prove Lemma \ref{L:mu_Dt_invariant}, we need some preparations.
Consider now a flow box $\Phi:\U\rightarrow \B\times \T$ as
above. Recall that for simplicity,  we identify $\U$ with $\B\times\T$
and $\T$ with the transversal
$\Phi^{-1}(\{z\}\times\T)$ for some point $z\in \B$.
We have the following result.

\begin{lemma} \label{L:Lusin1}  {\rm \cite[Proposition 3.2]{DinhNguyenSibony12} }
Let  $\nu$ be a positive Radon 
  measure on $\T$. Let
  $\T_1\subset \T$ be a measurable set such that $\nu(\T_1)>0$ and $L_x$ is
  hyperbolic for any $x\in\T_1$. Then, for every $\epsilon>0$ there
  is a compact set $\T_2\subset \T_1$ with $\nu(\T_2)>\nu(\T_1)-\epsilon$ 
and a family of universal covering maps $\phi_x:\D\rightarrow L_x$
with $\phi_x(0)=x$ and $x\in \T_2$ that depends continuously on $x$.
\end{lemma}
 For  $x\in\Hyp(\Fc)$ and $R>0,$ denote
$L_{x,R}:=\phi_x(\D_R)$, where  we recall  from  \eqref{e:covering_map} that  $\phi_x:\D\rightarrow L_x$  is a universal  covering map  with $\phi_x(0)=x$, and $\D_R\subset \D$ is the disc of
center 0 and of radius $R$  (with respect to the
Poincar{\'e} metric $g_P$ on $\D$). Since $\phi_x$ is unique up to a
rotation on $\D$, $L_{x,R}$ is independent
of the choice of $\phi_x$. We will need the following result.

\begin{lemma} \label{L:Lusin2} {\rm \cite[Corollary 3.3]{DinhNguyenSibony12} }
Let $R>0$ be a positive constant.
Then, under the hypothesis of Lemma
  \ref{L:Lusin1}, there is  a countable family of compact sets $\S_n\subset \T_1$, $n\geq1$, with
  $\nu(\cup_n \S_n)=\nu(\T_1)$ such that $L_{x,R}\cap \S_n=\{x\}$ for every $x\in \S_n$. Moreover, there are universal covering maps $\phi_x:\D\rightarrow L_x$ with $\phi_x(0)=x$ 
which depend continuously on $x\in\S_n$.
\end{lemma}

\proof[Proof of Lemma \ref{L:mu_Dt_invariant}]   (see \cite[Proposition 7.3]{DinhNguyenSibony12})
We can assume that $u$ is positive and using a partition of unity, 
we can also assume that $u$ has support in a compact set of a flow box
$\U\simeq \B\times\T.$ Let $\T_1$ be the set of $\T\cap \Hyp(\Fc).$
 We will use the decomposition of $T$ and the notation
as in Proposition 
\ref{P:decomposition}. By hypothesis, we can assume that the measure $\nu$ has total
mass on $\T_1$. 
Now, we apply Lemma \ref{L:Lusin2} to  $\nu$ and 
$4\lambda R$ instead of $R$ for a fixed constant $\lambda$ large enough. 
Let $\Sigma_n(R)$ denote the union of $L_{x,R}$
for $x\in \S_n$. 
Define by induction 
the function $u_n$ as follows: $u_1$ is the restriction of $u$ to
$\Sigma_1(R)$ 
and $u_n$ is the restriction of
$u-u_1-\cdots-u_{n-1}$ to $\Sigma_n(R)$. We have $u=\sum u_n$. So, it
is enough 
to prove the proposition for each $u_n$. 

We use now the properties of $\S_n$ given in Lemma
\ref{L:Lusin2}. 
The set $\Sigma_n(4\lambda R)$ is a smooth lamination and the restriction
$T_n$ 
of $T$ to $\Sigma_n(4\lambda R)$ is a positive harmonic current. 
Observe that $A_Ru_n$ vanishes outside $\Sigma_n(\lambda R)$ and
does not depend on the restriction of $T$ to $X\setminus\big(E\cup\Par(\Fc)\cup \Sigma_n(2\lambda R)\big)$. 
Since there is a natural projection from $\Sigma_n$ to the transversal $\S_n$, 
the extremal positive harmonic currents on $\Sigma_n$ are supported by
a leaf and defined by a harmonic function. Therefore, we can reduce the
problem 
to the case where $T=h[L_{x,4\lambda R}]$ with $x\in \S_n$ and $h$ is positive
harmonic on $L_{x,4\lambda R}$. 

Define $\widehat u:=u_n\circ \phi_x$, $\widehat h:=h\circ \phi_x$ and
$\widehat{A_R u}:=(A_R u)\circ\phi_x$. The function $\widehat h$ is
harmonic on $\D_{4\lambda R}$. 
Choose a 
measurable set $\Theta\subset \D_{2\lambda R}$ such that $\phi_x$ defines a
bijection 
between $\Theta$ and $L_{x,2\lambda R}$.   We first observe that 
$$\widehat{A_R u}(0):={1\over \mathcal M_R}\int_{\dist_P(\zeta,0)<R}\widehat u(\zeta)g_P(\zeta).$$
If $\eta$ is a point in $\D$ and $\tau:\D\to\D$ is an automorphism such that $\tau(0)=\eta$, then $\phi_x\circ\tau$ is also a covering map of $L_x$ but it sends 0 to $\phi_x(\eta)$. 
We apply the above formula to this covering map. Since $\tau$ preserves $g_P$ and $\dist_P$, we obtain
$$\widehat{A_R u}(\eta):={1\over \mathcal M_R}\int_{\dist_P(\zeta,\eta)<R}\widehat u(\zeta)g_P(\zeta).$$
Hence, we have to show the following identity
$$\int_\Theta \Big[{1\over \mathcal M_R}\int_{\dist_P(\zeta,\eta)<R} \widehat
u(\zeta)g_P(\zeta)\Big]\widehat h(\eta)g_P(\eta) =
\int_\Theta\widehat u(\zeta)\widehat h(\zeta)g_P(\zeta).$$

Let $W$ denote the set of points $(\zeta,\eta)\in\D^2$ such that
$\eta\in \Theta$ and $\dist_P(\zeta,\eta)<R$. Let $W'$ denote the symmetric
of $W$ with respect to the diagonal, i.e. the set of $(\zeta,\eta)$
such that $\zeta\in \Theta$ and $\dist_P(\zeta,\eta)<R$. Since $\widehat h$ is
harmonic, we have
$$\widehat h(\zeta)= {1\over \mathcal M_R}\int_{\dist_P(\zeta,\eta)<R} \widehat
h(\eta)g_P(\eta).$$
Therefore, our problem is to show that the integrals of $\Phi:=\widehat
u(\zeta)\widehat h(\eta) g_P(\zeta)\wedge g_P(\eta)$ on $W$ and
$W'$ are equal.

Consider the map 
$\phi:=(\phi_x,\phi_x)$ from $\D^2$ to $L_x^2$. The 
fundamental group $\Gamma:=\pi_1(L_x)$ can be identified with a group of
automorphisms of $\D$. Since, $\Gamma^2$ acts on
$\D^2$ and preserves the form $\Phi$, our problem is equivalent to
showing that each fiber of $\phi$ has the same number of points in $W$
and in $W'$. 
We only have to consider the fibers of points in $L_{x,2\lambda R}\times L_{x,2\lambda R}$ since $A_Ru_n$ is supported on $L_{x,\lambda R}$. 
Fix a point $(\zeta,\eta)\in \Theta^2$ and consider the
fiber $F$ of $\phi(\zeta,\eta)$. By definition of $\Theta$, the numbers of
points in $F\cap W$ and $F\cap W'$ are respectively equal to
$$\#\big\{ \gamma\in\Gamma,\ \dist_P(\gamma\cdot\zeta,\eta)<R\big\}
\quad \mbox{and}\quad \#\big\{ \gamma\in\Gamma,\ \dist_P(\zeta,\gamma\cdot\eta)<R\big\}.$$
Since $\Gamma$ preserves the Poincar{\'e} metric $g_P$ on $\D$, the first set is
equal to 
$$\{ \gamma\in\Gamma,\ \dist_P(\zeta,\gamma^{-1}\cdot\eta)<R\big\}.$$
It is now clear that the two numbers are equal. This
completes the proof.
\endproof

\begin{remark}
 \rm Theorem  \ref{thm_harmonic_currents_vs_measures} (3) gives a very effective necessary and  sufficient criterion for a positive quasi-harmonic  measure  to be  harmonic:
 the finiteness of the measure. This is a  percular  property of the leafwise Poincar\'e metric $g_P.$
 For  general families of leafwise  metrics  on  Riemann surface laminations  with singularities, or more  generally, for $N$-real-or complex dimensional laminations,
  it is  an important question to determine  when    a positive finite quasi-harmonic  measure  is  harmonic.
\end{remark}

 \begin{problem}
  \rm
  For  general families of leafwise  metrics  on  Riemann surface laminations  with singularities, or more  generally, for $N$-real-or complex dimensional laminations,
find  sufficient and  effective conditions  for  a   positive finite quasi-harmonic  measure  to be  harmonic.
 \end{problem}

 \subsection{Positive $\ddc$-closed currents on complex manifolds}
 \label{SS:Hamonic_currents}

 Let $M$ be a  complex manifold of dimension $k.$
A $(p,p)$-form on  $M$  is {\it
  positive} if it can be written at every point as a combination with
positive coefficients of forms of type
$$i\alpha_1\wedge\overline\alpha_1\wedge\ldots\wedge
i\alpha_p\wedge\overline\alpha_p$$
where the $\alpha_j$ are $(1,0)$-forms. A $(p,p)$-current or a $(p,p)$-form $T$ on $M$ is
{\it weakly positive} if $T\wedge\varphi$ is a positive measure for
any smooth positive $(k-p,k-p)$-form $\varphi$. A $(p,p)$-current $T$
is {\it positive} if $T\wedge\varphi$ is a positive measure for
any smooth weakly positive $(k-p,k-p)$-form $\varphi$. 
If $M$ is given with a Hermitian metric $\beta$ and $T$ is  a positive  $(p,p)$-current on $M,$
$T\wedge \beta^{k-p}$ is a positive measure on
$M$. The mass of  $T\wedge \beta^{k-p}$
on a measurable set $A$ is denoted by $\|T\|_A$ and is called {\it the mass of $T$ on $A$}.
{\it The mass} $\|T\|$ of $T$ is the total mass of  $T\wedge \beta^{k-p}$ on $M.$
A $(p, p)$-current $T$ on $M$  is {\it strictly positive} if  we have locally $T \geq \epsilon \beta^p ,$ i.e., $T -\epsilon \beta^p$
is positive, for some constant $\epsilon > 0.$ The definition does not depend on the
choice of $\beta.$

 A $(p,p)$-current on  $M$  is {\it
closed} if  $d T=0$ in the  weak sense (namely,  $T(d \alpha)=0$ for every  form $\alpha\in \Dc^{k-p,k-p-1}(M)\oplus\Dc^{k-p-1,k-p}(M)$). 
 A $(p,p)$-current on  $M$  is {\it
$\ddc$-closed} if  $\ddc T=0$ in the  weak sense (namely,  $T(\ddc \alpha)=0$ for every  form $\alpha\in \Dc^{k-p-1,k-p-1}(M)$). 

  For  every $r>0$ let $\B_r$ denote the ball of
  center $0$ and of radius $r$ in $\C^k$. 
The following local property of positive $\ddc$-closed
currents  has been discovered  by Skoda \cite{Skoda}.

\begin{proposition} {\rm (Skoda \cite{Skoda}).} \label{P:Skoda}
Let $T$ be a positive $\ddc$-closed $(p,p)$-current in
  a ball $\B_{r_0}$.
Define $\beta:=\ddc\|z\|^2$
  the standard K{\"a}hler form where $z$ is the canonical
  coordinates on $\C^n$.
Then the function $r\mapsto \pi^{-(k-p)} r^{-2(k-p)}\|T\wedge\beta^{k-p}\|_{\B_r}$
is increasing on $0<r\leq r_0$. In particular, it is bounded on 
$(0,r_1]$ for any  $0<r_1<r_0$.
\end{proposition}

\begin{definition}\label{D:Lelong}\rm
Under the  hypothesis and notation  of Proposition \ref{P:Skoda},
the limit of the above function when $r\rightarrow 0$ is called {\it
  the Lelong number} of $T$ at $0,$  and is  denoted by $\nu(T,x).$
  \end{definition}
 By Proposition  \ref{P:Skoda},  Lelong
number always  exists and is finite non-negative.

The  next  simple  result    allows  for extending  positive $\ddc$-closed  currents of bidimension $(1,1)$    through isolated points.

\begin{proposition} \label{P:extension}{\rm  (Dinh-Nguyen-Sibony \cite[Lemma 2.5]{DinhNguyenSibony12},  Forn{\ae}ss-Sibony-Wold \cite[Lemma 17]{FornaessSibonyWold})}
Let $T$ be a positive current of bidimension $(1,1)$ with compact support on a complex manifold $M$. Assume that $\ddc T$ is a negative measure on $M\setminus E$ where $E$ is a finite set. 
Then  $T$ is a positive  $\ddc$-closed  current on $M.$
\end{proposition}

Now  we come to the notion of holomorphically immersed laminations. 
  
   \begin{definition}\label{D:immersed-lamination}   \rm 
   Let  $\Fc=(X,\Lc,E)$   be  a  Riemann surface lamination  with  singularities and  let $M$ be a complex  manifold.  
    We say that  $\Fc$ is  {\it holomorphically immersed in  $M$}  if 
$X$ is a closed
subset of $M$  and    the leaves of $(X\setminus E,\Lc)$ are Riemann surfaces holomorphically immersed in
$M$.  
\end{definition}
If $\Fc:=(X,\Lc,E) $ is   a  singular holomorphic foliation, then  it is clearly a Riemann surface lamination  with  singularities  which is holomorphically immersed in $X.$
\begin{remark}\rm  \label{R:immersed-lamination}
 Let $j:\  X\hookrightarrow M$   be   the canonical injection  from $X$ into $ M. $ The   condition in Definition \ref{D:immersed-lamination} means  the following properties (i)-(ii):
 \begin{itemize}
 \item[(i)]  $X$ is a closed
subset of $M$ and  $j$ is  continuous;   
 \item[(ii)]  the  restriction $j_x$ of $j$  on  each leaf $L_x,$ with $x\in X\setminus E,$  is a holomorphic immersion from $L_x$  into  $M.$
 \end{itemize}
In particular property (i)   implies that  the topology of the lamination $X$ coincides with the  topology induced on the closed set $X$ from  $M.$
\end{remark}
The aggregate of  the  pull-back  via $j_x,$ with $x\in X\setminus E,$  of  each test form $\alpha\in  \Dc^{1,1}(M\setminus E)$  defines
a form in  $\Dc^{1,1}(\Fc)$ denoted by $j^*\alpha.$
So  we obtain  a canonical map $$j^*:\   \Dc^{1,1}(M\setminus E)\to \Dc^{1,1}(\Fc)\qquad\text{given by}\qquad  \alpha\mapsto  j^*\alpha.$$
We see  easily that  the image $\Ic$  of $j^*$  is dense in  $\Dc^{1,1}(\Fc).$

The original  notions  of  directed positive $\ddc$-closed currents for  singular  holomorphic foliations (resp. for singular laminations which are holomorphically  immersed in a complex  manifold)
with  a small set of singularities were  introduced by  Berndtsson-Sibony \cite{BerndtssonSibony} and  Forn{\ae}ss-Sibony \cite{FornaessSibony05,FornaessSibony08}) respectively.
  We give  here another   notion of directed positive $\ddc$-closed currents for singular Riemann surface  laminations.
  Our  notion coincides with the previous ones  when  the lamination is $\Cc^2$-transversally smooth.  The  advantage  of our notion is that  it is  relevant even  when the set of singularities is  not small.

 \begin{definition}\label{D:Directed_hamonic_currents_with_sing}   \rm 
Let  $\Fc=(X,\Lc,E)$   be  a  Riemann surface lamination  with  singularities   which is holomorphically immersed in a  complex  manifold $M.$
A {\it directed positive $\ddc$-closed  current}  (resp.  a  {\it directed  positive closed current}) on $\Fc$ is  a  positive $\ddc$-closed current $T$  (resp.   a  positive closed current $T$)
of bidimension $(1,1)$ on $M$ such that 
the following properties (i)-(ii)-(iii) are satisfied:
\begin{itemize} \item [(i)] the  support of $T$  is contained in $X;$ 
 \item [(ii)]  $T$  does not give mass to $E,$  i.e. the mass $\|T\|_E$  of $T$ on $E$  is  zero;
\item[(iii)]
 $T$   is  a  directed positive harmonic  current (resp.  a  directed positive closed  current) on $\Fc$  in the sense of Definition \ref{D:Directed_hamonic_currents}.
\end{itemize}
Moreover, we say that $T$ is  {\it diffuse}  if  it is  diffuse in the sense of Definition \ref{D:diffuse} as   a  directed positive harmonic  current (resp.  a  directed positive closed  current) on $\Fc.$ 
  \end{definition}
  
  \begin{remark}\rm  \label{R:explanation}  
  Property (iii) of Definition \ref{D:Directed_hamonic_currents_with_sing}
 means  the  following two properties (iii-a)-(iii-b):
\begin{itemize}\item[(iii-a)] $\langle T,\alpha\rangle=\langle T,\beta\rangle$ for $\alpha,\beta\in  \Dc^{1,1}(M\setminus E)  $ such  that $j^*\alpha=j^*\beta;$
 so the current $$\widetilde T:\ \Ic\to\C\quad\text{given by}\quad \langle \widetilde T, j^*\alpha\rangle:=  \langle T,\alpha\rangle\quad \text{for}\quad \alpha\in \Dc^{1,1}(M\setminus E),  $$
  is  well-defined;
 \item[(iii-b)]  the current $\widetilde T$ defined in  (iii-a) can be uniquely   extended from $\Ic$ to   $\Dc^{1,1}(\Fc)$ by continuity  (as $\Ic$ is  dense in  $\Dc^{1,1}(\Fc)$)  to a current
 $ \widehat T $ of order zero, and $\widehat T$ is 
 a   directed positive harmonic  current (resp.   directed positive closed  current) on $\Fc$  in the sense of Definition \ref{D:Directed_hamonic_currents}.
 \end{itemize}
 Property (iii-b)  holds  automatically since  $T$ is a positive $\ddc$-closed current (resp.  positive closed current)  on $M.$ So property (iii) is  equivalent to  the  single property (iii-a).
 If there is no confusion,   we often denote $\widetilde T$ and $\widehat T$ simply by $T.$
  \end{remark}
   
  \begin{remark}\rm  \label{R:defi_BFS}
    In the original definitions  of directed positive $\ddc$-closed  current (resp.   directed  positive closed current)  which have been
    introduced by  Berndtsson-Sibony \cite{BerndtssonSibony} and  Forn{\ae}ss-Sibony \cite{FornaessSibony05,FornaessSibony08}),
    these  authors  assume  the following  conditions (i)'-(ii)'-(iii)' (see  \cite[Definition 7]{FornaessSibony08}:
    \begin{itemize}\item[(i)'] $T$ is  a  positive $\ddc$-closed current (resp.    positive closed current)    on $M$ with  support in $X;$ 
     \item[(ii)'] $E$ is  a  {\it small set} in the sense that 
     $\Lambda_2(E)=0,$  where $\Lambda_2$ denotes the two dimensional Hausdorff measure;
 \item[(iii)']  a decomposition of $T$ as in Proposition \ref{P:decomposition} is 
    valid in any flow box outside the singularities.
    \end{itemize}
   In fact,  the assumption (ii)' ensures that $T$ does not give  mass to $E.$
    The reader may  consult the proof of Theorems 10 and 23 in \cite{FornaessSibony08} in order to  see that our definition of directed  positive closed  currents (resp. directed  positive 
    $\ddc$-closed  currents) is  equivalent to theirs when the lamination $\Fc$
     satisfies condition (ii)' and  $\Fc$ is  $\Cc^1$-transversally smooth (resp.  $\Cc^2$-transversally smooth).
    
    It is worthy noting that  these authors also introduce  a  notion of positive $\ddc$-closed  currents {\it weakly directed}  by a  lamination
    in a complex  manifold  (see \cite[Section 5]{FornaessSibony08}). Roughly speaking, this weaker notion means that  the considered current is  directed by a collection of continuous forms defining
    the considered lamination.
  \end{remark}

  The  existence of nonzero directed  positive harmonic  currents for  compact   nonsingular laminations was   proved by Garnett \cite{Garnett}.
   The case  of compact singular  holomorphic foliations was  proved by  
   Berndtsson-Sibony \cite[Theorem 1.4]{BerndtssonSibony} under  reasonable assumptions.  The  more general case of compact $\Cc^2$-transversally smooth  Riemann surface  laminations with singularities
   was proved by  Forn{\ae}ss-Sibony  \cite[Theorem 23]
   {FornaessSibony08},  see also Sibony  \cite{Sibony2} for  the existence of positive $\ddc$-closed currents directed by a Pfaff system.  
   
   Recall  that   a subset $E$ of  a complex  manifold $M$  is  said to be  {\it locally pluripolar} if, for every $a\in M,$ there is a plurisubharmonic function $u$ in some open neighborhood
   $\U$ of $a$ in $M$  such that  $\{u=-\infty\}\subset E\cap \U .$  Moreover, $E$ is  said to be 
    {\it locally complete  pluripolar} if,  for every $a\in M,$ there is a
   plurisubharmonic function $u$ in some open neighborhood
   $\U$ of $a$ in $M$  such that  $\{u=-\infty\}=E\cap \U .$
   Here is  a synthesis of the  above-mentioned existence results. 
   \begin{theorem}\label{T:existence_harmonic_currents} 
   Let $\Fc=(X,\Lc,E)$ be a singular  Riemann surface lamination which is  holomorphically immersed in a complex manifold $M.$
   Assume moreover that $X$ is  compact  
   and  we are  in  at least  one  of the  following two situations:
   
 \begin{enumerate}
   \item  {\rm  (Forn{\ae}ss-Sibony  \cite[Theorem 23]{FornaessSibony08})} $\Fc$ is   $\Cc^2$-transversally smooth  and  $E$ is  locally complete  pluripolar in $M;$
   \item {\rm  (Berndtsson-Sibony \cite[Theorem 1.4]{BerndtssonSibony})} $\Fc$ is  a singular holomorphic  foliation (so $M=X$)  and  $E$ is  locally pluripolar in $M.$
     \end{enumerate}
    Then there is a nonzero positive $\ddc$-closed  current $T$ of bidimension $(1,1)$ supported on $X$ such that 
    the restriction of $T$ on $X\setminus E$ defines a directed positive harmonic  current on  $(X\setminus E,\Lc).$
   In particular,  if  there is  no  nonzero  positive $\ddc$-closed  current of bidimension $(1,1)$  which  gives mass to $E$ (e.g. if  $\Lambda_2(E)=0,$
   where $\Lambda_2$ denotes the two dimensional Hausdorff measure),
   then $T$ is a nonzero  directed   positive  $\ddc$-closed   current in the sense of Definition  \ref{D:Directed_hamonic_currents_with_sing}.
   \end{theorem}
   When  a leaf $L_x$  is   hyperbolic, an exhaustion-average on $L_x$ (using the Nevanlinna current $\tau_{x,R}$ given in  formula \eqref{e:m_x,R})
   was introduced by Forn{\ae}ss-Sibony \cite{FornaessSibony05}  (see  also  \cite[Corolary 3]{FornaessSibony08}). 
   It provides another construction of  directed positive  $\ddc$-closed currents.
   
   By Burns-Sibony \cite{BurnsSibony}, if  a  singular holomorphic foliation  on a  compact complex manifold admits a  parabolic leaf in the sense of the potential theory,\footnote{Here,  a Riemann surface 
  $L$ is called  {\it parabolic in the sense of the potential theory} (or equivalently  {\it parabolic in the sense of Ahlfors}),   if bounded subharmonic functions on $L$ are constant.  If $L$ is parabolic
  (see Definition \ref{D:hyperbolic}) then it is parabolic in the sense of the potential theory.  But the converse statement is in general  not true.}
  then there is  a nonzero  positive closed  current weakly  directed by the  foliation (see Remark \ref{R:defi_BFS} for the notion of weakly-directed positive currents).  If  moreover, $\Fc$ satisfies 
  assumption (2) of Theorem \ref{T:existence_harmonic_currents}, then  this current is also  directed by the foliation  in the sense of Definition \ref{D:Directed_hamonic_currents_with_sing}.
   The reader is  invited  to  consult  P\u{a}un-Sibony \cite{PaunSibony} for a
  fruitful  discussion on the link between value distribution theory and     positive closed  currents directed by singular holomorphic foliations.

The  following notion will be   needed  later on.
\begin{definition}\label{D:invariant-curve}\rm
 Let $\Fc:=(X,\Lc,E) $ be a Riemann surface lamination  with  singularities which is  holomorphically immersed in a complex manifold $M.$
 A pure $1$-dimensional complex analytic set $L$ 
 in  $M\setminus E$  is  called an {\it invariant (analytic) curve} of $\Fc$ if $L$  is  itself  a leaf of $\Fc.$
 A pure $1$-dimensional complex analytic set $L$ 
 in  $M$  is  called an {\it invariant closed  (analytic) curve} of $\Fc$ if $L\setminus E$  is  itself  a leaf of $\Fc.$
 \end{definition}
 
If  the current of integration $[L]$ associated to an  invariant closed  analytic curve $L$ of $\Fc$ (as  it is a pure $1$-dimensional complex analytic set) does not give mass to $E,$ then $[L]$ is  a directed positive closed current which is clearly not diffuse
(see Definitions \ref{D:Directed_hamonic_currents_with_sing} and \ref{D:diffuse}). 

 By     Forn{\ae}ss-Sibony \cite[Proposition 3]{FornaessSibony08}, if $X$ is  compact and either $\Lambda_1(E)=0$  or  $\Lambda_2(E)=0$ and $E$ is locally complete pluripolar, then
 the closure $\overline L$ in $M$  of an  invariant (analytic) curve $L$  of $\Fc$ is  an invariant closed  (analytic) curve of $\Fc.$  Here, $\Lambda_k$ denotes the $k$-dimensional Hausdorff measure.
 
When $M$ is an algebraic manifold, an  invariant analytic curve  is  also called {\it invariant algebraic curve}.
\subsection{Sample-path space and shift-transformations}
 \label{SS:Sample-path_space}

Let $\Fc=(X,\Lc,E)$ be  a   Riemann surface lamination with singularities.
Let  $\Omega:=\Omega(\Fc) $  be  the space consisting of  all continuous  paths  $\omega:\ [0,\infty)\to  X$ with image fully contained  in a  single   leaf. This  space  is  called {\it the sample-path space} associated to  $\Fc.$
  Observe that
$\Omega$  can be  thought of  as the  set of all possible paths that a 
Brownian particle, located  at $\omega(0)$  at time $t=0,$ might  follow as time  progresses.
For each $x \in X\setminus E,$
let $\Omega_x=\Omega_x(\Fc)$ be the  space  of all continuous
leafwise paths  starting at $x$ in $X\setminus E ,$ that is,
\begin{equation}\label{e:Omega_x}
\Omega_x:=\left\lbrace   \omega\in \Omega:\  \omega(0)=x\right\rbrace.
\end{equation}
  Consider 
  the   shift-transformations
  $\sigma_t:\  \Omega\to\Omega$  defined by 
   \begin{equation}\label{e:shift}  \sigma_t(\omega)(s):=\omega(s+t),\qquad  \omega\in \Omega,\  s\in\R^+.
   \end{equation}

\subsection{Holonomy and monodromy} \label{SS:Hol}

Let $\Fc=(X,\Lc,E)$ be a  Riemann surface  lamination with singularities and let $L$ be a leaf of $\Fc.$ Fix  two points $x,y\in L$
and  consider small transversals $\T_x,$ $\T_y$ to $L$ at $x,y$ respectively. If $\Fc$ is a singular holomorphic  foliation  and $X$ is a complex manifold of dimension $k,$ we can choose
$\T_x,\T_y\simeq  \D^{k-1}.$ Let $\gamma:\ [0,1]\to L$ be a continuous path with $\gamma(0)=x$ and $\gamma(1)=y.$
For each $z\in \T_x$  near $x$ one can travel on $L_z$ over $\gamma[0,1]$  to reach $\T_y$ at some point $z'.$ More precisely,  let $\{(\U_p,\Phi_p)\}_{0\leq p\leq n-1}$ be laminated charts of $\Fc$ and
$0=t_0 <t_1<\cdots<t_{n-1}<t_n=1$ be  a partition of $[0,1]$ such that if $\U_p\cap \U_q\not=\varnothing$  then $\U_p\cup\U_q$ is contained in a laminated chart, and $\gamma[t_p,t_{p+1}]\subset \U_p$
for $0\leq p\leq n-1.$  For each $1\leq p\leq n-1,$ choose a transversal $\T_p$ to $L$ at $\gamma(t_p),$ and set $\T_0:=\T_x$ and $\T_n:=\T_y.$
Then for each $z\in \T_p$ sufficiently close to $\gamma(t_p)$, the plaque of $\U_p$ passing through $z$ meets $\T_{p+1}$ in a unique point $f_p(z),$  and $z\mapsto f_p(z)$ is  homeomorphic with 
$f_p(\gamma(t_p))=\gamma(t_{p+1}).$ This map is smooth/holomorphic if $\Fc$ is a transversally smooth lamination/singular holomorphic  foliation. We see that  the composition
$$
\hol_\gamma:=f_{n-1}\circ\cdots\circ f_0
$$
  is well-defined for $z\in\T_x$ near $x,$ with $\hol_{\gamma}(x)=y$  (see Figure \ref{F:auxiliary}). 
  \begin{definition}\label{D:holonomy_map}\rm 
   The map $\hol_\gamma$ is called the {\it holonomy} associated with $\gamma.$
  \end{definition}

\begin{figure}[h]%
\begin{center}
\def\svgwidth{0.9\columnwidth}
\resizebox{0.8\textwidth}{!}{\input{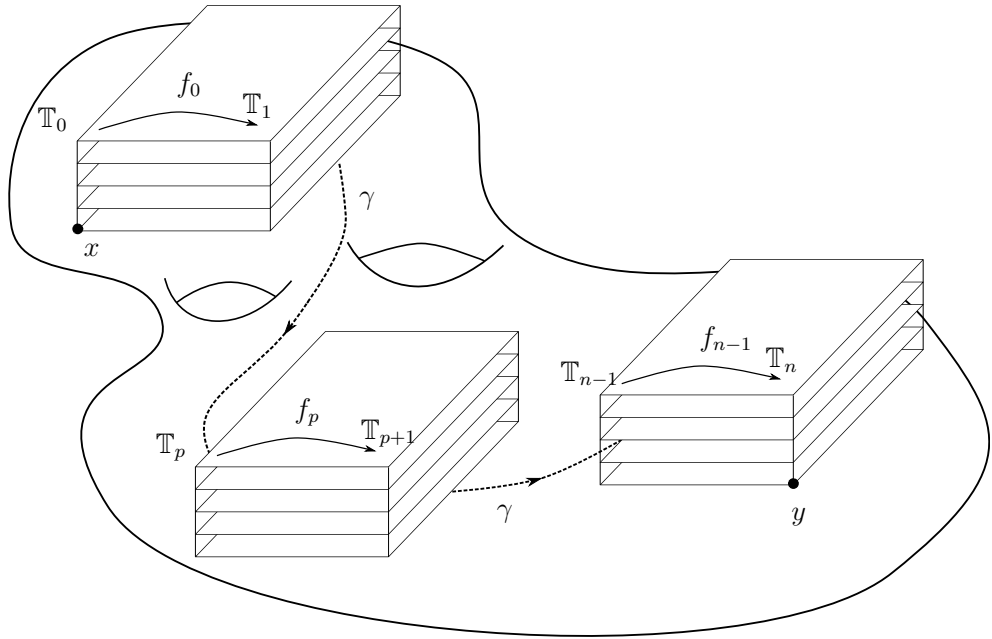}}%
\caption{The holonomy $\hol_\gamma$ associated  with  a path $\gamma$ going from $x$ to $y.$}  
\label{F:auxiliary}
\end{center}
\end{figure}

The following properties of the  holonomy  can be checked directly from the definition.

\begin{remarks}\rm
$\hol_\gamma$ is  independent of the chosen transversals $\T_p,$  $1\leq p\leq n-1,$ and the laminated/foliated  charts $\U_p.$ Hence, $\T_x,$ $\T_y$ and $\gamma$ determine the germ of $\hol_\gamma$
at $x.$

The germ of $\hol_\gamma$ at $x$ depends only on the homotopy class of $\gamma$ with  fixed end-points. More  concretely, if $\alpha:[0,1]\to L$ is another continuous  path  with $\alpha(0)=x$ and 
$\alpha(1)=y$ which is homotopic  to $\gamma$ in $L,$ then the germ of $\hol_\alpha$ coincides with that of $\hol_\gamma.$

If $\gamma^{-1}(t):=\gamma(1-t),$ then $\hol_{\gamma^{-1}}=(\hol_\gamma)^{-1}.$ Consequently, $\hol_\gamma$ represents  the germ of a local homeomorphism/diffeomorphism/biholomorphism.

Let $\T'_x$ and $\T'_y$ be  other transversals to $L$ at $x$ and $y,$ respectively.
Let $h:\ \T_x\to\T'_x$ and $\tilde h:\ \T_y\to\T'_y$ be projections along the plaques of $\Fc$ in a neighborhood of $x$ and $y,$ respectively. Then the holonomy
$\hol'_\gamma:\ \T'_x\to \T'_y$ satisfies  $\hol'_\gamma=\tilde h\circ \hol_\gamma\circ h^{-1}.$

\end{remarks}

\begin{figure}[h]%
\begin{center}
\def\svgwidth{0.6\columnwidth}
\resizebox{0.5\textwidth}{!}{\input{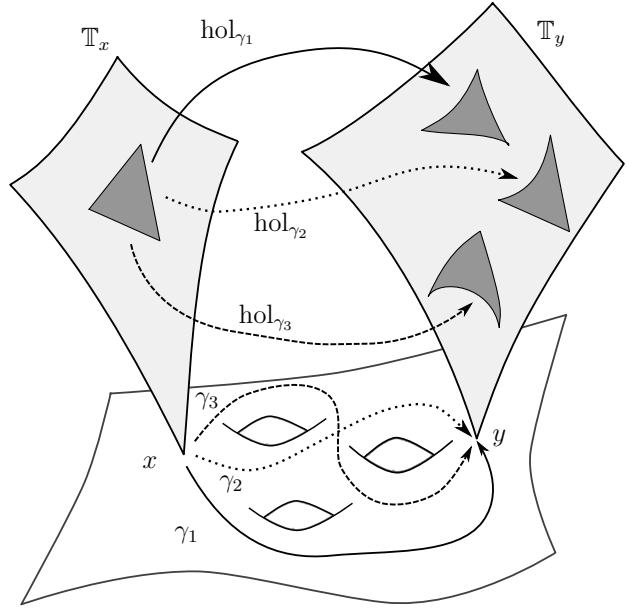}}%
\caption{The holonomy maps  associated with  pairwise non-homotopic  paths $\gamma_1,$ $\gamma_2$ and $\gamma_3$ are  all  different.}  
\label{fig:holonomy}
\end{center}
\end{figure}

Consider  the special case $x=y$ and $\T_x=\T_y.$  We obtain a generalization of  the Poincar\'e first return map.
The holonomy map $\hol_\gamma$ is called the {\it monodromy map} of $L$ associated with $\gamma.$ By the above remarks,
the germ of $\hol_\gamma$ at $x$ depends only on the homotopy class $[\gamma].$  We see that $[\gamma]\mapsto \hol_\gamma$ is a homomorphism from the first fundamental group
$\pi_1(L,p)$ into the group of germs of homeomorphisms/diffeomorphisms/biholomorphisms of $\T_x$ fixing $x:$ $[\gamma\circ\alpha]\mapsto \hol_\gamma\circ\hol_\alpha.$

\subsection{Holomorphic vector fields and  singular holomorphic  foliations} \label{SS:Local-theory}

 In order to review  the local theory of singular holomorphic foliations  we start with  holomorphic  vector fields.
 \begin{definition}
\label{D:singularities}\rm
  Let $Z=\sum_{j=1}^k F_j(z){\partial\over\partial  z_j}$ be  a holomorphic vector field  defined in a neighborhood $\U$ of $0\in\C^k.$ Consider  the holomorphic map
  $F:=(F_1,\ldots,F_k):\  \U\to\C^k.$ We say that $Z$ is 
\begin{enumerate}
\item {\it  singular  at $0$} if  $F(0)=0.$
\item {\it generic linear} if
it can be written as 
$$Z(z)=\sum_{j=1}^k \lambda_j z_j {\partial\over \partial z_j}$$
where $\lambda_j$ are non-zero complex numbers. The  $k$  hyperplanes $\{z_j=0\}$ for $1\leq j\leq k$ are said to be the {\it invariant  hypersurfaces.}

\item     {\it with non-degenerate singularity at $0$} if $Z$ is  singular at $0$ and the eigenvalues $\lambda_1,\ldots,\lambda_k$ of the Jacobian matrix $DF(0)$
 are all  nonzero. We say  that the singularity is in the Poincar\'e domain if the convex hull in $\C$ of
$\{\lambda_1 , \ldots , \lambda_k \}$ does not contain the origin, it is in the Siegel domain otherwise.

\item    {\it with weakly hyperbolic  singularity at $0$} if $Z$ is  singular at $0$ and the eigenvalues $\lambda_1,\ldots,\lambda_k$ of the Jacobian matrix $DF(0)$
 are all  nonzero and  there  are  some $1\leq j\not=l\leq k$ with $\lambda_j/\lambda_l\not\in\R.$   

\item    {\it with hyperbolic  singularity at $0$} if $Z$ is  singular at $0$ and the eigenvalues $\lambda_1,\ldots,\lambda_k$ of the Jacobian matrix $DF(0)$
 are all  nonzero and $\lambda_j/\lambda_l\not\in\R$  for all $1\leq j\not=l\leq k.$
\end{enumerate}
The integral curves of $Z$ define a  singular holomorphic foliation on
$\U$.
The condition $\lambda_j\not=0$ implies that the foliation   
has an isolated singularity at 0.
\end{definition}

 Let $\Fc=(X,\Lc,E)$ be a singular holomorphic foliation  such that $E$ is an  analytic subset of $X$  with $\codim(E)\geq 2.$ Then $\Fc$ is    given locally  by holomorphic  vector fields and its
 leaves  are locally, integral curves  of these vector fields, and the set of non-removable singularities of $\Fc$  coincide  with the union of the zero sets of these vector fields.
 Here  a point $z\in E$ is  called  a {\it removable singularity} of $\Fc$  if there is  a singular holomorphic  foliation $\Fc'=(X,\Lc',E')$    such that $E'\subset E\setminus \{z\}$ and that
 $\Lc=\Lc'|_{X\setminus E}.$  A point $a\in E$  is  said to be a {\it non-removable singularity} if   it is  not a removable singularity.
 So such a  foliation $\Fc$   is given by  an open covering $\{\U_j\}$ of $X$ and  holomorphic vector fields $v_j\in H^0(\U_j,\Tan(X))$  such that
$$
v_j=g_{j\ell}v_k\qquad\text{on}\qquad \U_j\cap \U_\ell
$$
for some non-vanishing holomorphic functions $g_{j\ell}\in H^0(\U_j\cap \U_\ell, \Oc^*_X).$
Its leaves are locally integral curves, of these vector fields.  The set of non-removable singularities  of $\Fc$ is  precisely the union of the zero  sets of these local vector fields.

 We say that a
singular point $a\in E $ is {\it linearizable} (resp.  {\it weakly hyperbolic}  or  {\it hyperbolic}) if 
there is a local holomorphic coordinate system of $X$ near $a$ on which the
leaves of $\Fc$ are integral curves of a generic linear vector field  (resp.  of a holomorphic vector field admitting $0$ as  a weakly  hyperbolic singularity  or a  hyperbolic singularity).
Clearly,  $a$ is  an  isolated point of $E.$ Moreover,  if $a$  is  a  hyperbolic singularity then  it is  clearly   a weakly hyperbolic singularity.
The converse is  true  only 
in   dimension $k=2$ (i.e.  $\dim X=2$).

Now  we focus  on the dimension $k=2.$  If  $a$ is  a hyperbolic singularity, then there is a local holomorphic coordinates system of $X$ near $a$ on which the
leaves of $\Fc$ are integral curves of a  vector field $Z(z_1,z_2) = \lambda_1 z_1 {\partial\over \partial z_1}
+ \lambda_2 z_2{\partial\over \partial z_2},$ where  $\lambda_1,\lambda_2$ are some nonzero complex numbers with $\lambda_1/\lambda_2\not\in \R.$ 
In particular, $a$ is a linearizable  singularity. Clearly, in dimension 2 
a hyperbolic singularity is always in the Poincar\'e domain.
The   analytic curves (invariant hypersurfaces) $\{  z_1=0\} $ and   $\{  z_2=0\}$ are  called  {\it separatrice}  at $a.$

The  following  result says roughly that for a   hyperbolic  singularity, the topological type of the foliation  around this  point is  determined by the eigenvalues of its linear part. 
\begin{theorem}{\rm  (Chaperon \cite{Chaperon})} Let $Z$ be a germ of a holomorphic vector field in $(\C^k , 0).$ If
$0$ is a hyperbolic singularity, then $Z$ is topologically linearizable.
This means that there is a homeomorphism  $\Phi:\ (\C^k , 0) \to (\C^k , 0)$ sending
the 
foliation defined by $Z$ to the foliation defined by the vector field $Z_0 = \sum\lambda_j {\partial\over \partial z_j}$
where the $\lambda_j $ are the eigenvalues of $DZ(0).$
\end{theorem}

Now we discuss the biholomorphic   type  of a holomorphic foliation near a singularity. 
When  we write the formal conjugation of a holomorphic vector field near a singularity,
to its linear part, we have to divide by the quantities $\langle \lambda, m\rangle -\lambda_j .$ Here, 
$$ \lambda:= (\lambda_1,\ldots,\lambda_k)\in\C^k ,\quad m = (m_1 , \ldots, m_k ) \in
\N^k ,\quad |m| =\sum_{j=1}^k m_j\geq 2,$$ and  $\langle \lambda, m\rangle$ denotes the inner product
$\sum_{j=1}^k\lambda_j m_j .$ To
prove the convergence we need that these quantities are non zero and not too close
to zero.

The {\it resonances} of $\lambda \in \C^k$ are defined by
$$\mathcal R := \left\lbrace  (m, j ):\quad  m = (m_1 , \ldots , m_k )\in\N^k,\quad |m| \geq 2, \langle \lambda, m\rangle -\lambda_j = 0\right\rbrace.$$
Notice that the set $\{\langle \lambda, m\rangle -\lambda_j :\ |m| \geq 2\}$ has zero as a limit point if and only if $\lambda$ belongs to
the Siegel domain.

We  are in the position to  state some classical results of  the local theory of singular holomorphic  foliations  near an isolated  singular point in any dimension.

\begin{theorem} {\rm  (Poincar\'e \cite{Arnold})} A germ of a singular holomorphic vector field in $(\C^k , 0)$
with a non resonant linear part (i.e., $\mathcal R$ is empty) such that $\lambda$ is in the Poincar\'e domain
is holomorphically equivalent to its linear part.
\end{theorem}
To get linearization for $\lambda$ in the Siegel domain, the  following fundamental  result assumes the so called
{\it Brjuno condition}: condition (B) [1, 7]. Define for $n \in \N$
$$\Omega(n) := \inf\limits_{1\leq j\leq k} \left\lbrace |\langle \lambda, m\rangle -\lambda_j |:\quad m = (m_1 , \ldots , m_k ), |m| \leq 2^{n+1} \right\rbrace.$$

Condition (B) is satisfied when $\mathcal R$ is  empty and 
$$\sum_{n\geq 1}{\log{(1/\Omega(n))}\over 2^n}
< \infty.$$
 In dimension 2, if $\lambda\in \R$ and $\lambda<0,$ then 
the Brjuno condition becomes $\sum_{n\geq 1} {\log{ q_{n+1}}\over  q_n}  < \infty$  where ${p_n /q_n }$ is the nth approximant
of $-\lambda.$

\begin{theorem}{\rm (Brjuno \cite{Arnold,Brjuno})} A germ of a singular holomorphic vector field in $(\C^k , 0)$
with non resonant linear part and which satisfies the Brjuno condition is holomorphically linearizable.
\end{theorem}

\subsection{Examples} \label{SS:Examples}


\begin{example}\label{Ex:Suspensions}\rm 
 {\bf Suspension:} The simplest examples of laminations are obtained by the
process of suspension.  Now  we present these examples in our context \cite{CamachoNetoSad85}. Let $S$ be a compact Riemann surface of genus $g \geq 2.$ Consider a homomorphism $h :\  \pi_1 (S)\to\Aut(\P^1 ).$
Let $\phi:\ \D\to S$ be a universal covering map. Note that     $\pi_1(S)$  may be considered as   a subgroup of  $\Aut(\D).$ Thus $h$ induces a natural
action $\tilde h$ on $\D\times \P^1 .$ More precisely:
\begin{eqnarray*} \tilde h: \pi_1  (S)&\to& \Aut(\D \times \P^1 ),\\
\tilde h [\alpha](z, w) &:=& ([\alpha] \cdot z, h[\alpha] \cdot w).
\end{eqnarray*}
Observe that the action of $\tilde h$ on $\D\times \P$ is free and properly discontinuous. Therefore, we can  consider the
manifold $M = M_h := \D \times \P^1 /\tilde h$ which is the quotient by the above action. The  map  $\phi:\ \D\to S$  induces a
natural projection $\pi : M \to S.$ Moreover,  the trivial foliation on $\D \times \P^1 ,$ with leaves $\D \times \{w\}$
induces a foliation $\Fc_h$ on $M,$ whose leaves are coverings of $S,$ hence hyperbolic
Riemann surfaces. So to any representation $h$ of $\pi_1 (S)$ into $\PSL(2, \C)$ corresponds
a foliation.  E. Ghys \cite{Ghys} shows  that $M_h$ is a projective surface.

\end{example}

\begin{example}\rm 
 {\bf Levi flats (or  equivalently  Cauchy-Riemann foliations):}
 Let $M$ be a complex surface, $\Tan (M)$ be its tangent bundle and $J$ be the
endomorphism of $\Tan( M)$ given by the complex structure (satisfying $J^2 = -\id$).
 Given a real
hypersurface  $X$
in  $M,$ we define the Cauchy-Riemann distribution on
 $X$ as follows. To each point
$x\in X$
we associate  the unique complex
line contained in $\Tan_x(X),$ i.e. the distribution $\Tan(X)\cap 
J\Tan(X).$ The hypersurface $X$
is called Levi-flat if the Cauchy-Riemann distribution is integrable in the sense of Frobenius. This means that through any point of
$X$
passes
a non-singular holomorphic curve of $M$
that is completely contained in $X.$ These curves
correspond then to the leaves of a foliation on $X,$ called the Cauchy-Riemann foliation or CR foliation. The condition that ensures a real hypersurface to be
Levi-flat can be characterized by the vanishing of the Levi form of $X.$

\end{example}


\begin{example}\rm 
 {\bf Singular holomorphic  foliations  on $\P^k$ with $k\geq 2$ :}  Let $\pi:\ \C^{k+1}\setminus\{0\}\to\P^k$ denote  the  canonical projection.  Let $\Fc$ be a
 singular holomorphic  foliations  on $\P^k,$ 
 It can be shown that  $\pi^*\Fc$ is  a singular foliation on $\C^{k+1}$ associated to a vector field $Z$  of the  form
 $$
 Z:=\sum_{j=0}^k F_j(z){\partial\over\partial  z_j},
 $$
 where the  $F_j$ are homogeneous  polynomials of degree $d\geq 1.$ We call  $d$ the {\it degree} of the  foliation. 
 Here $\pi^*\Fc$ should be understood  as a foliation  such that  each leave  is contained in  the preimage of a leaf of $\Fc$ by  $\pi,$ and  that  the foliation
 is   invariant under  all  dilatation  maps $A_\lambda:\ \C^{k+1}\setminus\{0\}\to\C^{k+1}\setminus\{0\}$ 
 defined by  $z\mapsto\lambda z,$  with $\lambda\in\C\setminus \{0\}.$

 In dimension $k=2$, the number of tangencies of a generic line with a foliation is  exactly  its degree.

 For $d\geq 2,$ let $\Fc_d(\P^k)$ be the  space of singular holomorphic  foliations of degree $d$ in $\P^k.$
 Using  the above  form of $Z,$  we can show that  $\Fc_d(\P^k)$  can be canonically identified  with a Zariski open subset of $\P^N,$ where
 $N:=(d+k+1){(d+k-1)!\over (k-1)! d!} -1 $ (see \cite{Brunella06}).

 A point $x\in \P^k$
 is a   singularity of $\Fc$ if $F(x)$ is colinear with $x,$ i.e., if $x$ is either an indeterminacy point or a fixed  point  of
 $f=[F_0:\ldots: F_k]$ as a meromorphic map in $\P^k.$
 If $ f$ is non
holomorphic, then its indeterminacy set is analytic of codimension $\geq  2,$ it can be of
positive dimension when $k \geq  3.$ Assume that $f$ is holomorphic. To count the fixed
points we only need to apply the B\'ezout theorem to the equations $F_j (z) - t^{d-1} z_j = 0$
in $\P^{k+1}$ with homogeneous coordinates $[z : t]$ and observe that $[0 : \ldots : 0 : 1]$ is a
solution. The number of fixed points counted with multiplicity is ${d^{k+1}-1\over  d-1} -1.$ So the singularity set, $\Sing(\Fc),$ of any holomorphic foliation $\Fc\in \Fc_d(\P^k)$ is always nonempty.

 \end{example}

 The  next result describes  some typical properties of a generic foliation $\Fc\in\Fc_d(\P^k).$
 
 \begin{theorem}\label{T:generic} Let $d,k>1$.

  \begin{enumerate}
  \item {\rm (Jouanolou \cite{Jouanolou}, Lins Neto-Soares \cite{NetoSoares}).}
  There is a real Zariski dense open set $\mathcal H(d)\subset  \Fc_d(\P^k)$ such that for every $\Fc\in \mathcal H(d),$
  all  the singularities of $\Fc$ are    hyperbolic and  $\Fc$ does not possess any invariant algebraic  curve.
  \item       {\rm (Glutsyuk  \cite{Glutsyuk},  Lins Neto \cite{Neto}).}
  If all  the singularities of a  foliation $\Fc\in \Fc_d(\P^k)$ are    non-degenerate, then $\Fc$ is hyperbolic.
   \item  {\rm (Brunella \cite{Brunella06}).}   If 
  all  the singularities of a  foliation $\Fc\in \Fc_d(\P^k)$ are    hyperbolic and  $\Fc$ does not possess any invariant algebraic  curve, then $\Fc$ admits no  nontrivial  directed positive closed current.
  \end{enumerate}
\end{theorem}
  
  Moreover, Loray-Rebelo \cite{LorayRebelo}  constructed a nonempty  open  set  $\mathcal U(d)$ of $\Fc_d(\P^k)$
   such that every leaf of $\Fc\in \mathcal U(d)$ is  dense.
 By  Theorem \ref{T:generic}, Theorem \ref{T:existence_harmonic_currents}  applies to every generic foliation in  $\P^k$ with a given degree $d>1.$

 \subsection{Sullivan's dictionary}\label{SS:Sullivan}

   Here is  part of the  correspondence between the  world of maps and that of laminations/foliations.


  \begin{tabular}{|l|l|l|l|}
  \hline
  \multicolumn{3}{|c|}{Dictionary} \\
  \hline
  Notion  & Maps &  Laminations/Foliations   \\ 
  \hline
  \multirow{2}{*}{Orbit of  a point $x\in X\setminus E$ } & the sequence $x,\ f(x), f^2(x),$            & Leaf  $L_x$ \\
    & $\ldots,f^n(x),\ldots$& Leafwise continuous paths $\Omega_x$ \\
    & (only one orbit) & (infinitely many paths $\omega\in\Omega_x$) \\
    \hline
  \multirow{2}{*}{Orbit of  a point $x\in X\setminus E$}
    & Dirac mass at $x$ & Wiener measure $W_x$  on $\Omega_x$ \\
    &  &  $\quad$ at $x\in\Hyp(\Fc)$\\
      \hline
  Counting time & linear time $t\in \N$ & hyperbolic $t\in\D$ \\ \hline
  \multirow{2}{*} {Nonsingular}& continuous map & Riemann surface lamination \\
  &&\\
   objects & smooth map  & smooth  lamination  \\
 \hline
  \multirow{2}{*} {Singular holomorphic}& meromorphic map & singular holomorphic foliation \\
  &&\\
   objects & indeterminacy point/set & singular point/set \\
 \hline
 \multirow{3}{*} {Invariant dynamical}& invariant measures & directed posi. closed currents \\
  &&\\
   objects &  & directed posi. harmonic currents \\
   \ldots &\ldots&\ldots\\
 \hline
\end{tabular}
 
 \medskip
 
 By  this  dictionary,  the notion  of  {\it orbit  of a  point $x$} has only one entry in the world of maps (that is, the unique  sequence of points  $x,\ f(x), f^2(x),$  $\ldots,f^n(x),\ldots$). 
 However, this notion
 has two entries in the world of laminations/foliations. The first  one is the  whole  leaf $L_x$  which is  a Riemann surface. The second  entry is  the space $\Omega_x$
 of leafwise continuous  paths $\omega:\ \R^+\to L_x$  starting  from $x,$  i.e.   $\omega(0)=x.$
 
 Similarly,   the  entry {\it invariant positive measures} in the world of maps  possesses 
 two entries in the world of laminations/foliations. The first  one is  the concept of directed positive closed currents, the second entry is the notion of
 directed positive harmonic currents. Althouth  the  first entry is very close  to the original notion of  invariant measures, it turns out that the second entry is
 more appropriate because  laminations/foliations which have invariant measures are rather scarce.  Moreover,   directed positive harmonic currents provide
 a natural framework for  a good theory to be developed.
  
 \section{Random and Operator Ergodic Theorems}
 \label{S:Random}
In this  section we follow the   expositions  given in Sections 2.2, 2.4 and 2.5 in \cite{NguyenVietAnh17a} and in Subsection 2.4 in \cite{NguyenVietAnh18b}. We are partly inspired by the constructions
given in 
\cite{CandelConlon2, Candel03}.
The $\sigma$-algebra  generated  by a family $\Sc$ of subsets  of $\Omega$ is, by definition, the  smallest  $\sigma$-algebra  
containing  $\Sc.$ 
  \subsection{Wiener measures}\label{SS:Wiener}


Let $\Fc=(X,\Lc,E)$ be  a   Riemann surface lamination with singularities endowed  with the leafwise Poincar\'e metric $g_P.$
Recall  from Subsection \ref{SS:Sample-path_space}  that  $\Omega:=\Omega(\Fc) $  is    the sample-path space associated to  $\Fc$
  and that  for each $x \in X\setminus E,$
 $\Omega_x=\Omega_x(\Fc)$ denotes the  space  of all continuous
leafwise paths  starting at $x$ in $X\setminus E .$ 
Garnett developed in \cite{Garnett} a  theory of leafwise  Brownian  motion in this  context  by constructing
a $\sigma$-algebra  $(\Omega,\widetilde \Ac)$ together  with a family of Wiener measures (see also \cite{Candel03, CandelConlon2}).
Now  recall briefly her construction. 
A {\it  cylinder  set (in $\Omega$)} is a 
 set of the form
$$
C=C(\{t_i,B_i\}:1\leq i\leq m):=\left\lbrace \omega \in \Omega:\ \omega(t_i)\in B_i, \qquad 1\leq i\leq m  \right\rbrace,
$$
where   $m$ is a positive integer  and the $B_i$ are Borel subsets of $X\setminus E,$ 
and $0\leq t_1<t_2<\cdots<t_m$ is a  set of increasing times.
In other words, $C$ consists of all paths $\omega\in  \Omega$ which can be found within $B_i$ at time $t_i.$
For  each point $x\in \Hyp(\Fc),$ let
\begin{equation}\label{eq_formula_W_x_without_holonomy}
W_x(C ) :=\Big (D_{t_1}(\chac_{B_1}D_{t_2-t_1}(\chac_{B_2}\cdots\chac_{B_{m-1}} D_{t_m-t_{m-1}}(\chac_{B_m})\cdots))\Big) (x),
\end{equation}
where $C:=C(\{t_i,B_i\}:1\leq i\leq m)$ as  above,  $\chac_{B_i}$
is the characteristic function of $B_i$ and $D_t$ is the leafwise diffusion operator
given  by  (\ref{e:diffusions}).
Let  $\widetilde\Ac=\widetilde\Ac (\Fc)$ be  the $\sigma$-algebra   generated by all cylinder sets.
It can be proved that $W_x$ extends  to a probability measure on  $(\Omega,\widetilde\Ac).$
  
In the monograph \cite{NguyenVietAnh17a} we  introduce  another   $\sigma$-algebra  $\Ac$ on $\Omega,$ which is bigger (finer) than $\widetilde\Ac.$ In fact, 
$\Ac$  takes into account the  holonomy phenomenon,  whereas $\widetilde\Ac$ does  not so.  Here is our construction in the present context.
The  {\it covering lamination} $\widetilde \Fc=(\widetilde X,\widetilde\Lc)$  of  a  Riemann surface lamination with singularities $\Fc$  is, in some  sense, its  universal cover.
More  specifically,
for  every leaf $L$ of $\Fc$ (hyperbolic or  parabolic) and every point  $x\in L,$  let $\pi_1(L,x)$ denotes   the first  fundamental  group of
 all continuous closed paths $\gamma:\  [0,1]\to L$ based  at $x,$ i.e. $\gamma(0)=\gamma(1)=x.$  Let   $[\gamma]\in \pi_1(L,x)$ be  the   class of   a    closed path $\gamma$ based  at $x.$
 Then the pair  $(x,[\gamma])$ represents      a point  of   $ \widetilde X.$
  Thus
the set of points  $ \widetilde X$ of $\widetilde \Fc$  is well-defined. The  leaf  $\widetilde L$ passing through a given point $(x,[\gamma])\in\widetilde X,$ is  by definition, the set
$$
 \widetilde L:=\left\lbrace (y, [\delta]):\ y\in L_x,\  [\delta]\in \pi_1(L,y)  \right\rbrace,
$$
which is the universal cover of $L_x.$
 We  put the following  topological structure on  $\widetilde X$ by describing
  a  basis of open  sets. Such a  
   basis  consists of all  sets $\Nc(\U,\alpha).$  Here,
  $\U$ is an open subset of  $X\setminus E$ and $\alpha:\ \U\times [0,1]\to X\setminus E$ is   a  continuous  function such that  $\alpha_x:=\alpha(x,\cdot)$ 
  is a  closed  path in $L_x$ based at  $x$ for each $x\in \U,$ and  
  $$
  \Nc(\U,\alpha):=  \left\lbrace (x,[\alpha_x]):\ x\in \U \right\rbrace.
  $$
   The projection $\pi : \widetilde X\to  X\setminus E$ is defined by $\pi(x,[\gamma]) :=x. $  It is  clear that $\pi$ is  locally homeomorphic and  is  a leafwise map.
   By pulling-back the lamination atlas $\Lc$ of $ \Fc $ via  $\pi,$   we obtain  a  natural  lamination atlas $\widetilde \Lc$ for 
   the Riemann surface lamination
$  \widetilde \Fc .$ 
 Denote  by $\widetilde \Omega$ the sample-path space $\Omega( \widetilde \Fc)$ associated with the  Riemann surface lamination   
 $ \widetilde \Fc .$   Similarly, by pulling-back the the leafwise Poincar\'e metric $g_P$  defined on $\Hyp(\Fc)$  via  $\pi,$   we obtain  a  natural  leafwise metric $\pi^*g_P$  defined  
   on the hyperbolic part $\Hyp(\widetilde \Fc)$ of 
$  \widetilde \Fc .$

Let 
$ x\in X\setminus E$ and  $\tilde x$ an arbitrary point in $\pi^{-1}(x)\subset \widetilde X.$
 Similarly as in (\ref{e:Omega_x}), let   $\widetilde\Omega_{\tilde x}=\Omega_{\tilde x}( \widetilde\Fc)$ be the  space  of all  paths  
in $\widetilde \Omega$
starting at $\tilde x.$ 
Every  path $\omega\in \Omega_x$ lifts uniquely  to
a path $\tilde\omega\in \widetilde\Omega_{\tilde x}$  in the sense  that $\pi\circ \tilde\omega=\omega.$
In what follows this bijective lifting  is  denoted by $\pi^{-1}_{\tilde x}:\  \Omega_x\to \widetilde\Omega_{\tilde x}.$ So  $\pi\circ (\pi^{-1}_{\tilde x}(\omega))=\omega,$
 $\omega\in \Omega_x.$ 
\begin{definition} \label{defi_algebras_Ac} Let   
  $ \Ac=\Ac(\Fc)$ be  the  $\sigma$-algebra generated by
all sets  of  following family
 $$ \left  \lbrace \pi\circ \tilde A:\ \text{cylinder set}\ \tilde A\  \text{in} \ \widetilde \Omega    \right\rbrace,$$ 
  where  $\pi\circ \tilde A:= \{ \pi\circ \tilde \omega:\ \tilde\omega\in \tilde A\}.$
    \end{definition}
 Observe that    $\widetilde \Ac\subset  \Ac$ and that the  equality holds if  every   leaf  of  the foliation is
homeomorphic to the  disc $\D.$

 Now  we  construct  a family  $\{W_x\}_{x\in \Hyp(\Fc)}$ of probability Wiener measures   on $(\Omega,\Ac).$ 
  Let $x\in \Hyp(\Fc)$ and 
$C$ an element of $\Ac.$ Then  we  define the  so-called {\it  Wiener measure}  $W_x$ by the following formula
\begin{equation}\label{eq_formula_W_x}
 W_x(C):= W_{\tilde x}( \pi^{-1}_{\tilde x} C),
  \end{equation}
  where   
   $\tilde x$ is  an arbitrary point in $\pi^{-1}(x),$ and  
   $$\pi^{-1}_{\tilde x}C:=\left\lbrace \pi^{-1}_{\tilde x}\omega:\  \omega\in C\cap \Omega_x\right\rbrace,
$$
and  $ W_{\tilde x}$ is the  probability measure  on  $(\widetilde \Omega,\widetilde \Ac(\widetilde \Fc)) $ which was  defined by (\ref{eq_formula_W_x_without_holonomy}).
Given  a positive  finite Borel  measure $\mu$  on $X$ which gives no mass to $\Par(\Fc)\cup E,$  consider the   measure  $\bar\mu$ on $(\Omega,\Ac)$
defined by   
   \begin{equation}\label{eq_formula_bar_mu}
   \bar\mu(A):=\int_X\left ( \int_{\omega\in A\cap   \Omega_x}  dW_x \right ) d\mu(x),\qquad  A\in\Ac. 
\end{equation}
The  measure $\bar\mu$ is  called  the  {\it Wiener measure with initial distribution $\mu$.}

 Here  are some  important  properties of $\bar\mu.$
\begin{proposition}\label{P:Wiener_measure}
We keep the above  hypotheses and notation. 
\begin{enumerate}
\item [(i)]    The value  of $W_x(C)$ defined  in (\ref{eq_formula_W_x}) is  independent of the choice
 of $\tilde x.$  Moreover, $W_x$ is a  probability  measure     
on $(\Omega,\Ac).$

\item[(ii)]   $\bar\mu$ given in (\ref{eq_formula_bar_mu}) is  a positive finite  measure on $(\Omega,\Ac)$ and
$\bar\mu(\Omega)=\mu(X).$
\end{enumerate}
 
 \end{proposition}     
  \begin{proof}
 Assertion (i)  has been proved in    \cite[Theorem 2.15]{NguyenVietAnh17a}. Assertion (ii)  has been established in
  \cite[Theorem 2.16]{NguyenVietAnh17a}.
  
\end{proof}
  
  \subsection{Random and operator ergodic  theorems}
  \label{SS:Random_ergodic}

Let $\Fc=(X,\Lc,E)$ be  a   Riemann surface lamination with singularities.  Recall  from \eqref{e:shift} the shift-transformations  $\sigma_t,$ $t\in\R^+$.

 \begin{theorem}\label{T:Random}
  
  \begin{enumerate}
\item [(i)] If  $\mu$ is a very weakly  harmonic measure (resp.  weakly harmonic measure), then  $\bar\mu$ is  unit time-invariant (resp.  time-invariant), that is,
$$
\int_\Omega F(\sigma_t(\omega))d\bar\mu(\omega)=\int_\Omega F(\omega)d\bar\mu(\omega),
$$
for $t=1$ (resp. for all $t\in\R^+$) and $F\in L^1(\Omega,\bar\mu).$  
   \item[(ii)]
   If  $\mu$   is a very  weakly harmonic measure, then $\mu$ is ergodic if and only if $\bar \mu$ is ergodic for $\sigma_1.$
   If moreover $\mu$ is  weakly harmonic  and  is ergodic, then  $\bar \mu$ is ergodic for  all   $\sigma_t$ with $t \in \R^+ \setminus  \{0\}.$

  \end{enumerate}
 \end{theorem}
\proof
 Part (i) follows from  \cite[Theorem  2.20]{NguyenVietAnh17a}  where  the general case of an  $N$-real or complex dimensional  lamination  with a general leafwise metric was treated.
 
Part (ii)  is a consequence of   \cite[Theorem 4.6]{NguyenVietAnh17a}.
 \endproof
   The following Operator  Ergodic Theorem     may be  regarded as Akcoglu's ergodic theorem %
   (see Theorem 2.6 in \cite[p. 190]{Krengel})
in the  context of laminations.
\begin{theorem}\label{T:Akcoglu}
 Let $\mu$ be a    very weakly  harmonic  measure which  is  ergodic. Then the following properties hold.
 \begin{enumerate}
\item[(i)] If  $D_1f=f$  $\mu$-almost everywhere for some $f\in L^1(X,\mu),$  then  $f=\const$ $\mu$-almost everywhere.

\item[(ii)] For every $f\in L^1(X,\mu),$   ${1\over n}\sum_{j=0}^{n-1} D_jf$ converges  to   $\int_X f d\mu $  $\mu$-almost everywhere.
\end{enumerate}
If  moreover   $\mu$ is  weakly  harmonic, then  the following two properties hold for all $t_0>0.$
 \begin{enumerate}
\item[(i')] If  $D_{t_0}f=f$  $\mu$-almost everywhere for some $f\in L^1(X,\mu),$  then  $f=\const$ $\mu$-almost everywhere.

\item[(ii')] For every $f\in L^1(X,\mu),$   ${1\over n}\sum_{j=0}^{n-1} D_{jt_0}f$ converges  to   $\int_X f d\mu $  $\mu$-almost everywhere.
\end{enumerate}
\end{theorem}
\proof
It  follows from  \cite[Theorem B.16]{NguyenVietAnh17a}   where  the general case of an  $N$-real or complex dimensional  lamination  with a general leafwise metric was investigated.
\endproof

\begin{problem}
 \rm  It  seems of interest to find  sufficient conditions to ensure that  a very weakly harmonic measure (resp. a weakly harmonic measure) is  harmonic.   
\end{problem}

\section{Regularity of the leafwise Poincar\'e metric and mass-distribution  of currents}
\label{S:Regularity_Mass}

\subsection{Regularity of the leafwise Poincar\'e metric} \label{SS:Regularity_Poincare}

Let $\Fc=(X,\Lc,E)$ be a   Riemann surface  lamination  with singularities. Let $g_P$ be  as usual   the leafwise Poincar\'e metric for the  lamination   given in Subsection  
\ref{SS:Poincare}.
Let $g_X$ be a Hermitian metric on the leaves which is transversally continuous. We can construct such a metric on flow boxes and glue them using a partition of unity.
When  $\Fc$ is  holomorphically  immersed in  a complex manifold $M,$ we often fix  an ambient   Hermitian  metric $g_M$ on $M$ and consider its restriction   to the leaves.
We  denote by $g_X$  the Hermitian metric on the leaves obtained  by this   way.
Consider the function $\eta:\  X\setminus E\to [0,\infty]$  given by
\begin{equation}\label{e:eta_bis} \eta(x)=\sup\big\{\|(D\phi)(0)\|,\quad \phi:\D\to L_x \mbox{ holomorphic such that } \phi(0)=x\big\}.
 \end{equation}
 Here, and for the norm of the differential $D\phi$ we use the Poincar\'e metric on $\D$ and the Hermitian metric $g_X$ restricted to $L_x$.
Using a map sending $\D$ to a plaque, we see that the function $\eta$ is locally bounded from below on $X\setminus E$ by a strictly positive constant. 
Moreover, when $\Fc$ is  holomorphically  immersed  in a complex  manifold,   we can show that $\eta$ is  lower-semi continuous  on $X\setminus E$ 
(see \cite[Theorem 20]{FornaessSibony08}).
Note  that  $\{x\in  X\setminus E:\ \eta(x)=\infty\}=\Par(\Fc).$
When $X$ is  compact and $\Par(\Fc)=E=\varnothing,$ the classical Brody lemma (see \cite[p.100]{Kobayashi}) implies that $\eta$ is also bounded from above.

The extremal property of the Poincar\'e metric implies that, for $x\in\Hyp(\Fc),$
\begin{equation}\label{e:eta} g_X=\eta^2g_P\quad \mbox{where}\quad 
\eta(x):=\|D\phi_x(0)\| \quad\mbox{(see  \eqref{e:covering_map}).} 
\end{equation}
Given  a point $x\in \Hyp(\Fc)$ and  a differentiable function $f:\ L_x\to\C,$ we define $|df|_P:\ L_x\to \R^+$ by
\begin{equation}  \label{e:length-eucl-vs-Poincare}|df|_P(y):=  \eta(y) |df(y)|\qquad\text{for}\qquad y\in L_x,
\end{equation}
 where for the norm $|df(y)|$  we use the   Hermitian metric $g_X$ on $L_x$ and the  Euclidean norm $|\cdot|$ of $\C.$

The continuity of the function  $\eta$ was studied by Candel, Ghys, Verjovsky,   see \cite{Candel,Ghys,Verjovsky}. The survey \cite{FornaessSibony08} establishes this  result as a consequence of Royden's lemma. 
Indeed with his lemma, Royden proved the upper-semicontinuity of the infinitesimal Kobayashi metric in a Kobayashi hyperbolic manifold (see \cite[p.91 and p.153]{Kobayashi}).
The  following theorem gives refinements of  the previous results.

\begin{theorem} \label{T:Poincare}{\rm (Dinh-Nguyen-Sibony \cite{DinhNguyenSibony14a}). }
Let $\Fc=(X,\Lc)$ be a transversally smooth compact lamination by hyperbolic Riemann surfaces.
Then the Poincar\'e metric on the leaves  is H\"older continuous, that is, the function $\eta$ defined in \eqref{e:eta} is H\"older continuous on $X$. 
Moreover, the exponent of   H\"{o}lder continuity can be estimated in geometric terms.
\end{theorem}

The main tool of the proof of Theorem  \ref{T:Poincare}  is
to use  Beltrami's equation in order to
compare universal covering maps of  any leaf $L_y$ near a given leaf $L_x$.
More precisely, 
for  $R>0$ let $\D_R$ be  the  disc of center $0$ with  radius $R$  with respect to  the  Poincar\'e metric on $\D$ (see Main Notation). 
We first  construct a non-holomorphic parametrization $\psi$ from $\D_R$ to $L_y$ which is close to a universal covering map 
$\phi_x:\D\to L_x$ for each $R$ large enough. Next, precise geometric estimates on $\psi$ allow us to modify it, using Beltrami's equation. 
We then obtain a holomorphic map that we can explicitly compare with a universal covering map $\phi_y:\D\to L_y$.

 Next,    we investigate the regularity  of the leafwise Poincar\'e metric $g_P$  of a compact singular holomorphic  foliation. Here  
  an important difficulty   emerges: a leaf of the foliation may visit singular flow boxes without any obvious rule. 
  We introduce  the following class of laminations.
 
\begin{definition}\label{D:uniform_hyperbolic_laminations} {\rm (Dinh-Nguyen-Sibony \cite{DinhNguyenSibony14b}). }  \rm
 A  hyperbolic Riemann surface  lamination  with singularities   $\Fc=(X,\Lc,E)$  with $X$ compact
 is  said  to be {\it Brody hyperbolic}  if  there is  a constant  $c_0>0$  such that
$$\|D\phi(0)\| \leq  c_0$$ 
for all  holomorphic maps  $\phi$ from $\D$  into  a leaf,
in other  words,  if the  function $\eta$ is  uniformly  bounded  from above.
\end{definition}

 \begin{remark}\rm \label{R:Brody_sufficiency} 
It is clear that if the lamination is Brody hyperbolic then its leaves are hyperbolic in the sense of Kobayashi.
Conversely, the  Brody hyperbolicity is  a consequence of
the non-existence  of  holomorphic non-constant maps
$\C\rightarrow X$  such that out of  $E$ the image of $\C$ is  locally contained in leaves, 
see \cite[Theorem 15]{FornaessSibony08}.

On the other hand,   Lins Neto proved  in   \cite{Neto} that  for every   holomorphic foliation of degree larger than 1  in $\P^k$, with non-degenerate singularities (see Definition
\ref{D:singularities}),  
there is a smooth metric with negative curvature on its tangent bundle, see also Glutsyuk \cite{Glutsyuk}.
Hence, these foliations are Brody hyperbolic.
Consequently,  holomorphic foliations  in $\P^k$  are generically  Brody hyperbolic, see Theorem \ref{T:generic} (1). The reader may find  in \cite{SibonyWold}
a nice discussion on the topology and the conformal  structures of leaves of a singular holomorphic  foliation which is  Brody hyperbolic.
\end{remark}

 Denote by   $\lof(\cdot):=1+|\log(\cdot)|$ a log-type function, and  by $\dist$  the distance on $X$ induced by  the Hermitian metric $g_X.$
The following result is   a  counterpart of Theorem \ref{T:Poincare} in the context of singular holomorphic foliations. 

\begin{theorem} \label{T:Poincare_bis} {\rm (Dinh-Nguyen-Sibony \cite{DinhNguyenSibony14b}). } 
Let $\Fc=(X,\Lc,E)$ be a Brody hyperbolic singular holomorphic foliation  on a Hermitian compact complex manifold $X$. Assume that the singular set $E$  is 
finite  and  that all  points of $E$  are  linearizable. 
 Then, there are constants  $c>0$ and $0<\alpha<1$   such that
 $$|\eta(x)-\eta(y)|\leq  c\Big ( {\max\{\lof \dist(x,E), \lof\dist(y,E)\}\over \lof\dist(x,y)}\Big)^\alpha$$
 for all $x,y$ in $X\setminus E$. 
\end{theorem}

  To prove  this  theorem, we analyze the behavior and get an explicit estimate on the modulus of continuity of the Poincar\'e metric on leaves.
The following estimates are crucial in our method.  They  are also  useful in other  problems.

\begin{proposition} \label{P:Poincare}{\rm (Dinh-Nguyen-Sibony \cite{DinhNguyenSibony14b}). }
Under the hypotheses
of Theorem  \ref{T:Poincare_bis},  there exists  a  constant $c_1>1$  such that 
$$c_1^{-1} s \lof s \leq\eta(x)  \leq c_1  s \lof s$$
for $x\in X\setminus E$  and $s:=\dist(x,E)$.  
\end{proposition}

The  Poincar\'e metric on  the leaves of a hyperbolic  foliation is a fundamental  object  which is
extremely delicate  to understand.   As  we  see in Theorem \ref{T:Poincare_bis}, the  regularity in the direction  transverse to the foliation is  quite  weak.
This is  partly  due to   the presence  of the  singularities.
We  end the subsection with the  following open  question.
\begin{problem}  \rm  
Let $(X,\Lc,E)$ be a compact  singular  holomorphic  foliation.
Assume that every point $a\in E$ is   a non-degenerate singularity. 
Study the regularity  of the function $\eta.$
In case  $\dim(X)=2,$  we may investigate the  problem   where the singularities are not necessarily non-degenerate. 
\end{problem}

 \subsection{Mass-distribution of undirected and directed positive $\ddc$-closed currents}\label{SS:Mass}

   Let  $\Fc=(X,\Lc,E)$   be  a  Riemann surface lamination  with  singularities which is  holomorphically immersed in a
complex Hermitian manifold $(M,g_M).$  
Let 
 $T$ be a  positive $\ddc$-closed  current of bidimension $(1,1)$  on $M$  whose   support is in $X.$ Here, $T$ may or may not be directed by $\Fc.$ If $T$  is  not directed
 by $\Fc,$ we say  that it is  undirected. 
 Clearly, the mass of $T$ with
respect to the  Hermitian metric $g_M$ (i.e. the mass of the positive measure $T\wedge g_M$)  is locally finite on $X\setminus E. $ 
If moreover,  $T$ can be  extended trivially through $E$ to a undirected  (resp. directed) positive $\ddc$-closed  current    (see Definition \ref{D:Directed_hamonic_currents_with_sing}),
then  its mass is locally finite on $X .$  Consider the  following concept  (see Definition  \ref{D:Phi} for a similar notion in the context of directed positive harmonic currents).
\begin{definition} \rm \label{D:Poincare_mass}
   Consider  the map
$T\mapsto \Phi(T):= \mu$  which is defined  by the following  formula on the  convex  cone of all   positive $\ddc$-closed currents $T$ of bidimension $(1,1)$ on $M$  whose support is in $X$ : 
 \begin{equation}\label{e:mu_bis}  \mu :=T\wedge g_P\quad\text{on}\quad  X\setminus (E\cup\Par(\Fc))\quad\text{and}\quad  \mu(E\cup\Par(\Fc))=0.
  \end{equation}
   We call {\it Poincar{\'e}  mass} of $T$ the mass of $T$ with respect to
Poincar{\'e}  metric $g_P$ on $X\setminus E$, i.e. the mass of the
positive measure $\mu=\Phi(T)$.
\end{definition}

For many  ergodic problems, the local finiteness mass of   the Poincar\'e mass of $T$  
is  very important.
The following proposition
gives us a criterion for this local finiteness. It can be applied
to generic foliations in $\P^k$  (see Theorem  \ref{T:generic}).
 
\begin{proposition} \label{P:Poincare_mass} {\rm (Dinh-Nguyen-Sibony \cite{DinhNguyenSibony12}).}
Let $\Fc=(X,\Lc,E)$ be a singular  holomorphic  foliation. If $a\in E$ is a linearizable singularity, 
then any positive $\ddc$-closed  current on $X$ has locally finite
Poincar{\'e}  mass near $a$.
\end{proposition}

The proof of this  result is based on the finiteness of the Lelong number of $T$  at $a$  (see Proposition \ref{P:Skoda}).
We have  a more precise result when   $T$ is  directed and the singular point is weakly  hyperbolic (see Definition \ref{D:singularities}). 
\begin{theorem} \label{T:Lelong}
{\rm (Nguyen \cite{NguyenVietAnh18a,NguyenVietAnh19}).} Let $\Fc=(X,\Lc,E)$ be a singular  holomorphic  foliation  with $\dim X=k\geq 2.$ 
If $a\in E$ is a linearizable  singularity which is also a weakly hyperbolic singularity, 
then for any directed positive $\ddc$-closed  current $T$ on $X$  which does  not give mass to any of the $k$ invariant  hypersurfaces at $a,$    the   Lelong  number of $T$ at $a$  vanishes.  
 \end{theorem}

An  immediate consequence  of Theorem \ref{T:Lelong} is the following result on the Lelong numbers of  a directed positive $\ddc$-closed current.  
\begin{corollary}\label{C:Lelong}
Let $\Fc=(X,\Fc,E)$    be  
a     singular holomorphic foliation   with $X$ a    compact complex manifold. 
Assume that all the singularities  are not only linearizable  but also   hyperbolic and  that the foliation  has no invariant analytic  curve.
Then for every   positive $\ddc$-closed   current $T$  directed by $\Fc,$   the   Lelong  number of $T$ vanishes everywhere   in $X.$  
 \end{corollary}

The  above   corollary can be   applied  to  every generic  foliation in $\P^k$ with a given degree $d>1$
(see Theorem  \ref{T:generic}).

  \begin{remark}
   \rm 
   Theorem \ref{T:Lelong}   answers positively  Problem  4.7 raised in \cite{NguyenVietAnh18d}.
   In fact,   the  special case  $\dim X= 2$  of this theorem was  proved in  \cite{NguyenVietAnh18a}.
  \end{remark}

When the singularities  are  linearizable but not    weakly hyperbolic singularity, the study  of Lelong numbers seems  difficult, see  Chen's recent article \cite{Chen} for a partial result.

We  end the section with the  following open  question.
\begin{problem}  \rm  
Let $\Fc=(X,\Lc,E)$ be a compact  singular  holomorphic  foliation and let
 $T$ be a directed positive $\ddc$-closed  current for $\Fc.$
Find  sufficient conditions  on the nature of the set of  singularities $ E$  to ensure that  the Poincar\'e mass of $T$ is  finite. 
\end{problem}

\section{Heat  equation and  ergodic theorems} \label{S:Ergodic_theorems}

Let $\Fc=(X,\Lc,E)$ be  a  Riemann surface lamination with singularities. 
In collaboration with  Dinh and Sibony  \cite{DinhNguyenSibony12},  we introduce the heat equation relative to 
\begin{enumerate}
\item[$\bullet$] a  harmonic measure  $\mu$  of $\Fc;$
 \item [$\bullet$]
a positive $\ddc$-closed  current $T$  on  a complex  manifold $M$   
in the case where  $\Fc$ is  holomorphically immersed  in $M,$  the current $T$ is not necessarily  directed,  but its support is assumed  to be  in $X$ and its  Poincar\'e mass is  assumed to be finite.

\end{enumerate}   This permits us  to construct the abstract heat diffusion with respect to
various Laplacians that could be defined
almost everywhere with respect to the  quasi-harmonic measure/positive $\ddc$-closed  current. 
In this  section we follow  closely  the exposition of  \cite{DinhNguyenSibony12}.  Note however  that there are two differences.
The  first one is  that in the  present article  we only   consider laminations   by  Riemann  surfaces  and their  leafwise Poincar\'e metric, whereas  in \cite{DinhNguyenSibony12}
the case of  $N$-dimensional  laminations endowed with a general leafwise Riemannian metrics was  studied. So this difference limits the scope of the present article.
The  second difference  is  that the  Riemann  surface laminations  considered in this  article  may be  not compact and their  singularities may be neither  isolated    nor finite.
 This  new  situation 
 leads us to introduce  some   spaces of test functions slightly more general than  those given in   \cite{DinhNguyenSibony12}.

Recall some classical results of functional analysis.
 The reader will find an exposition in Brezis \cite{Brezis}. A
linear operator $A$ on a Hilbert space $L$ is called {\it monotone} if
$\langle Au,u\rangle\geq 0$ for all $u$ in the domain $\Dom(A)$ of
$A$. Such an operator is {\it maximal monotone} if moreover for any
$f\in L$ there is a $u\in\Dom(A)$ such that $u+Au=f$. In this case,
the domain of $A$ is always dense in $L$ and the graph of $A$ is closed. 

A family $S(t):L\rightarrow L$, $t\in\R_+$, is {\it a semi-group of
  contractions} if $S(t+t')=S(t)\circ S(t')$ and if $\|S(t)\|\leq 1$
for all $t,t'\geq 0$. We will apply the following theorem 
to our Laplacian operators.  
It says that any maximal 
monotone operator is the infinitesimal generator of a semi-group of contractions.

\begin{theorem}[Hille-Yosida] \label{th_hille_yosida}
Let $A$ be a maximal monotone operator on a Hilbert space $L$. Then
there is a semi-group of contractions $S(t):L\rightarrow L$,
$t\in\R^+$, such that for $u_0\in\Dom(A)$, $u(t,\cdot):=S(t)u_0$ is
the unique  function 
$$ u\in\Cc^1(\R^+,L)\cap \Cc(\R^+,
\Dom(A))$$ which satisfies
$${\partial u(t,\cdot)\over \partial t}+Au(t,\cdot)=0\quad
\mbox{and}\quad u(0,\cdot)=u_0.$$ 
When $A$ is self-adjoint and $u_0\in L$, then 
$$u\in\Cc(\R^+,L) \cap \Cc^1(\R^+_*,L)\cap 
\Cc(\R^+_*,\Dom(A)),$$
where $\R^+_*:=\R^+\setminus\{0\},$ and we have the estimate
$$\Big\|{\partial u\over\partial t}\Big\|\leq {1\over t}\|u_0\|\quad \mbox{for } t>0.$$ 
\end{theorem}

In order to check that our operators are maximal monotone, we will
apply the following result. 

\begin{theorem}[Lax-Milgram] \label{th_lax_milgram}
Let $e(u,v)$ be a continuous bilinear form on a Hilbert space
$H$. Assume that $e(u,u)\geq \|u\|^2_H$ for $u\in H$. Then for every $f$
in the dual $H^*$ of $H$ there is a unique $u\in H$ such that 
$e(u,v)=\langle f,v\rangle$ for $v\in H$. 
\end{theorem}
 
\subsection{Heat  equation  on  Riemann surface laminations with singularities} 
\label{SS:heat_equation_CHRSL}

Consider a  Riemann surface lamination with singularities $\Fc=(X,\Lc,E)$  
 and  a positive quasi-harmonic  measure $\mu$.

 In a flow box
$\U\simeq \B\times\T$, by Proposition \ref{P:decomposition}, the current $T$ can be
written as 
\begin{equation} \label{e:local_presentation}T=\int h_a [\B\times\{a\}]d\nu(a),
\end{equation}
where $h_a$ is a  positive harmonic function on $\B$  and $\nu$ is
a  positive  measure  on the transversal $\T$.
 
 By  Theorem  \ref{thm_harmonic_currents_vs_measures}  (1), there is a unique  directed positive harmonic current $T$ giving no mass to $\Par(\Fc)$  such that 
 \begin{equation}\label{e:m-T} \mu=T\wedge g_P\quad\text{on}\quad \Hyp(\Fc)\quad\text{and}\quad \mu=0\quad\text{on}\quad \Par(\Fc)\cup E.\end{equation}
 It follows  from  \eqref{e:Laplacian_disc} and  \eqref{e:Delta_commutation} that  the following  identity  holds  for  $u\in\Dc(\Fc),$   
\begin{equation}\label{e:Laplacian_bis}
\int_X (\Delta_P u)d\mu=\int_X i\ddbar u\wedge T . 
 \end{equation}
 
In what follows, the differential operators $\nabla$, $\Delta_P$ and $\widetilde\Delta_P$ are considered in $L:=L^2(\mu)$. We introduce the  Hilbert  space $H:=H^1(\mu)$ as
the completion of $\Dc(\Fc)$ with  respect to  the  norm
\begin{equation}\label{e:norm_H1} \|u\|_{H}^2:=\int  |u|^2 d\mu+\int | \nabla u|^2 d\mu. 
\end{equation}
Recall that the gradient $\nabla$ is defined by 
\begin{equation}\label{e:grad}\langle \nabla u,\xi\rangle_{g_P}=du(\xi)
\end{equation}
for all tangent vector $\xi$ along a leaf and for $u\in\Dc(\Fc)$. 
 In  comparison  with \eqref{e:length-eucl-vs-Poincare}, we  see that 
 \begin{equation}\label{e:nabla-vs-dP}
 |\nabla u|=|du|_P.
 \end{equation}
We consider $\nabla$ as a operator in $L^2(\mu)$ and $H^1(\mu)$ is  its domain of definition.

Define in a flow box $\U\simeq \B\times \T$ as above the Laplace type operator
\begin{equation}\label{e:widetilde-Laplacian}\widetilde \Delta_P u= \Delta_P u+\langle h_a^{-1}\nabla h_a,\nabla u\rangle_{g_P}=\Delta_P u+F u,
 \end{equation}
where $F$ is a vector field.
The uniqueness of $h_a$ and $\mu$ implies that $F$ does
not depend on the choice of the flow box. Therefore, $Fu$ and
$\widetilde\Delta_P u$ are defined globally $\mu$-almost everywhere when $u\in\Dc(\Fc)$.

\begin{remark}
 \label{R:comparison-two-laplacians}
 {\rm
 We say  that a  quasi-harmonic measure $\mu$ is {\it invariant} if   for  every flow box $\U$  the functions $h_a$ given  in \eqref{e:local_presentation} are constants $c(a)$
 for $\nu$-almost every $a\in\T.$ If  $\mu$ is  invariant  then the two laplacians $\Delta_P $  and $\widetilde\Delta_P $ coincide.  
 }
\end{remark}

Define for $u,v\in\Dc(\Fc)$
$$q(u,v):=- \int (\Delta_P u)v d\mu, \qquad
e(u,v) := q(u,v)  +\int uvd\mu$$
and
$$\widetilde q(u,v):=-\int (\widetilde{\Delta}_Pu)v d\mu=q(u,v)-\int
(Fu)vd\mu, \qquad
\widetilde e(u,v):=\widetilde q(u,v)+\int uvd\mu.$$
Note that these identities still hold for $v\in L^2(\mu)$ and $u$ in the
domain of $\Delta_P$ and of $\widetilde\Delta_P$ that we will
define later.

The  main properties  of these  bilinear forms  are  described  in the following   lemma. In   particular,   the lemma says that $\widetilde\Delta_P $ is  self-adjoint. 
This  is  the main  advantage  of $\widetilde\Delta_P $ over $\Delta_P.$
\begin{lemma} \label{lemma_e_tilde}
We have 
for $u,v\in\Dc(\Fc),$
\begin{eqnarray*}\widetilde q(u,v)&=&\int \langle\nabla u,\nabla v\rangle_{g_P} d\mu=\int i\partial u\wedge \dbar v\wedge T,\\  
 \int (\widetilde\Delta_P u)v d\mu &=& \int
u(\widetilde\Delta_P v) d\mu.
\end{eqnarray*}
In particular, $\widetilde q(u,v)$ and $\widetilde e(u,v)$ are symmetric in $u,v.$ 
Moreover,
$$\int \widetilde\Delta_P u d\mu=\int\Delta_P u d\mu=\int Fu d\mu=0 \quad \mbox{for}\quad u\in\Dc(\Fc).$$
\end{lemma}
\proof
Using a partition of unity, we can assume that $u$ and $v$ have
compact support in a flow box as above. Using \eqref{e:local_presentation} it is then enough to consider
the case where $T$ is supported by a plaque $\B\times\{a\}$ and given
by a harmonic function $h_a$. Using    \eqref{e:m-T},  \eqref{e:grad} and \eqref{e:Laplacian_disc}, we have that
\begin{eqnarray*}
 \widetilde q(u,v)&=&q(u,v)-\int
(Fu)vd\mu=-\int_{\B\times\{a\}} (i\ddbar u)v h_a-\int_{\B\times\{a\}}  iv\partial u \wedge \dbar h_a\\
&=&-\int_{\B\times\{a\}} (i\ddbar u)v h_a-\int_{\B\times\{a\}}  iv\partial u\wedge  \dbar h_a  -\int_{\B\times\{a\}}  id\big(vh_a \partial u  \big),
\end{eqnarray*}
where the last  equality  holds  because  the last integral  in the  last line  is  equal  to $0$ by Stokes' theorem.
After expanding  the  differential expression $d\big(vh_a \partial u  \big)$ in the  last term  in the last line, and  then simplifying   the last line,  we get  that
 $$\widetilde q(u,v)=\int_{\B\times\{a\}} ih_a\partial u\wedge \dbar v.$$ 
 This  proves  the first  identity of the lemma.

It follows from this identity  that $\widetilde q$
and $\widetilde e$ are
symmetric.

  The second identity of the lemma  ($\widetilde\Delta_P$ is  self-adjoint) is an immediate  consequence of the first one.
 Applying   the  second identity  to the case where  $v=1$  on a neighborbood of the support of $u,$  so   $ u\widetilde\Delta v=0$ and we obtain  the  other identities of the lemma.

Note that the lemma still holds for
$u,v$ in the domain of $\Delta$ and $\widetilde\Delta$ that we will
define later.
\endproof

\begin{lemma} \label{lemma_e_e_tilde}
The bilinear forms $\widetilde q$ and $\widetilde e$ extend
  continuously to $H^1(\mu)\times H^1(\mu)$. If  the  measure $\mu$ is  finite,  then  $q$ and $e$ also extend
  continuously to $H^1(\mu)\times H^1(\mu)$.  Moreover, we have
$q(u,u)=\widetilde q(u,u)$ and $e(u,u)=\widetilde e (u,u)$ for $u\in H^1(\mu)$. 
\end{lemma}
\proof
The first identity in Lemma \ref{lemma_e_tilde} implies that
$\widetilde q$ and $\widetilde e$
extend continuously to $H^1(\mu)$ and the identity is still valid for
the extension of $\widetilde q$. 
In order to prove the same property
for $q$ and $e$  when $\mu$ is  of finite mass, it is enough to show that $q-\widetilde q$ is bounded on
$H^1(\mu)\times H^1(\mu)$. For this we  follow  the proof of  Lemma \ref{lemma_e_e_tilde_c} below.
\endproof

\begin{definition}\label{D:Dom-Delta}\rm 
Define the domain $\Dom(\pm\Delta_P)$ of $\pm\Delta_P$
(resp. $\Dom(\pm\widetilde\Delta_P)$ of $\pm\widetilde\Delta_P$) as the space of $u\in
H^1(\mu)$ such that $q(u,\cdot)$ (resp. $\widetilde q(u,\cdot)$) extends to a linear
continuous form on $L^2(\mu)$. 
\end{definition}
\begin{remark}\label{R:Dom-Delta} \rm 
 Using     Lemmas  \ref{lemma_e_tilde} and \ref{lemma_e_e_tilde}  we can show that when the measure $\mu$ is  finite,
$\Dom(\pm\Delta_P)=\Dom(\pm\widetilde\Delta_P)$  (see  also Remark  \ref{R:Dom}). 
Moreover,  we can prove that  $\Dom(\pm\Delta_P)$ is equal to  $ H_P(\mu),$  where  $H_P(\mu)$ is the completion of $\Dc(\Fc)$ for the norm 
\begin{equation}\label{e:norm_P}
\|u\|_{H_P(\mu)}:=\sqrt{\|u\|^2_{L^2(\mu)}+\|\Delta_P u\|_{L^2(\mu)}^2}. 
\end{equation}
For more  details see \eqref{e:Delta_P_monotone-bis} and  \eqref{e:Delta_P_monotone} in the proof of  Proposition \ref{prop_delta_max} below.
It is clear that if $u\in\Dom(-\Delta_P)$ then $\Delta_P u$ in the sense of distributions with respect to $\Dc(\Fc)$ as test functions, is in $L^2(\mu)$. This allows us to extend Lemma 
\ref{lemma_e_tilde} to $u,v$ in $\Dom(-\Delta_P),$  or more generally to $u\in \Dom(-\Delta_P) $ and $v\in H^1(\mu).$ 
\end{remark}
The existence of abstract heat diffusions are given by  the following proposition.

\begin{proposition} \label{prop_delta_max}
Let $\Fc=(X,\Lc,E)$ be a   Riemann surface lamination with singularities   endowed
with the leafwise Poincar\'e metric $g_P.$ Let $\mu$ be
a positive quasi-harmonic  measure.
Then 
the associated operator $-\widetilde\Delta_P$ is maximal monotone on $L^2(\mu)$. 
If, moreover,  the  measure  $\mu$ is  finite,  then the associated operator 
$-\Delta_P$ is also maximal monotone on $L^2(\mu)$. In particular,
they are infinitesimal generators of semi-groups of contractions on $L^2(\mu)$ and their graphs are closed.
\end{proposition}
\proof
The last assertion is the consequence of the first one, Theorem
\ref{th_hille_yosida} and the properties of maximal monotone operators. 
So, we only have to prove the first assertion.
By  Lemma  \ref{lemma_e_tilde}
we have  
\begin{equation}\label{e:Delta_P_monotone-bis} \langle - \widetilde \Delta_P u ,u \rangle_{g_P}=\tilde q(u,u)  
=\int i\partial u \wedge \dbar u \wedge T  \geq 0.
\end{equation}
This, combined  with Lemma  \ref{lemma_e_e_tilde}, yields that
\begin{equation}\label{e:Delta_P_monotone}
\langle - \Delta_P u ,u \rangle_{g_P}  = q(u,u)=\tilde q(u,u)=\int i\partial u \wedge \dbar u \wedge T  \geq 0.
\end{equation}
By continuity, we can extend the inequalities to $u$ in $\Dom(-\Delta_P)=\Dom(-\widetilde\Delta_P)$.
So, $-\Delta_P$ and $-\widetilde\Delta_P$ are monotone. 
Pick $u\in \Dc(\Fc).$
By Lemmas \ref{lemma_e_tilde}
 and \ref{lemma_e_e_tilde}, we have for $u\in H^1(\mu)$
\begin{equation} \label{e:lower-bound-e} \widetilde e(u,u)\geq \|u\|^2_{H^1} \quad \mbox{and}\quad
e(u,u)\geq \|u\|^2_{H^1}.\end{equation}
By Theorem \ref{th_lax_milgram}, for any $f\in L^2(\mu)$, there is $u\in
H^1(\mu)$ such that
$$e(u,v)=\langle f,v\rangle_{L^2(\mu)} \quad \mbox{for}\quad v\in
H^1(\mu).$$
So, $u$ is in $\Dom(\Delta_P)$ and the last equation is equivalent
to $u-\Delta_P u=f$. 
Hence, $-\Delta_P$ is maximal monotone. The case of $-\widetilde\Delta_P$ 
is treated in the same way. Note that since $-\widetilde\Delta_P$ is
symmetric and maximal monotone, it is self-adjoint but $-\Delta_P$ is not symmetric. 
\endproof
 
When the measure  $\mu$ is a  finite,   we obtain  the following ergodic theorem for abstract heat diffusions which is    stronger  than  Proposition  \ref{prop_delta_max}.

\begin{theorem} \label{th_heat_real} 
Let $\Fc=(X,\Lc,E)$ be a   Riemann surface lamination with singularities   endowed
with the leafwise Poincar\'e metric $g_P.$ Let $\mu$ be
a harmonic  measure.  Let $S(t)$,
$t\in\R^+$, denote  the semi-group of contractions associated with the operator
$-\Delta$ or  $-\widetilde\Delta$ which is given by the Hille-Yosida
theorem. Then 
the  
measure $\mu$ is  $S(t)$-invariant and
$S(t)$ is a positive contraction in $L^p(\mu)$ for all $1\leq p\leq\infty.$
\end{theorem}
 
 \proof
We prove that $\mu$ is invariant, that is 
$$\langle \mu,S(t)u_0\rangle = \langle \mu,u_0\rangle \quad\mbox{for}\quad u_0\in \Dc(\Fc).$$
We will see later that this identity holds also for $u_0\in L^1(\mu)$
because $S(t)$ is a contraction in $L^1(\mu)$ and $\Dc(\Fc)$ is dense in $L^1(\mu)$. Define 
$u:=S(t)u_0$ and 
$$\eta(t):=\langle \mu,S(t)u_0\rangle=\langle \mu,u(t,\cdot)\rangle\quad \text{for}\quad t\in\R^+.$$
We deduce from Theorem \ref{th_hille_yosida} that $\eta$ is of class $\Cc^1$ on $\R^+$ and that
$$\eta'(t)=\langle \mu,S'(t)u_0\rangle =\langle \mu,Au(t,\cdot)\rangle$$
where $A$ is the operator $-\Delta_P$ or $-\widetilde\Delta_P$. By Lemma \ref{lemma_e_tilde},
the last integral vanishes. So, $\eta$ is constant and hence $\mu$ is invariant.

In order to prove the positivity of $S(t)$, it is enough to show the following {\it maximum principle}: if $u_0$ is a function in $\Dc(\Fc)$ such that $u_0\leq K$ for some constant $K$, then $u(t,x)\leq K$. 
To show  the maximum  principle  we  use a trick due to  Stampacchia \cite{Brezis}.  Fix a smooth bounded function
 $G:\R\rightarrow \R^+$ with bounded first derivative such that  $G(t)=0$ for  $t\leq 0$
and   $G'(t)>0$ for  $t>0.$  Put
$$H(s):=\int_0^s G(t) dt. $$
Consider  the  non-negative  function $\xi:\R^+\rightarrow \R^+$  given by 
$$\xi(t):=\int H(u(t,\cdot)-K)d\mu.$$
Here we make use of the  assumption that $\mu$ is finite.
By Theorem \ref{th_hille_yosida}, $\xi$ is of class $\Cc^1$.
  
We  want to show that it is  identically zero. Define $v(t,x):=u(t,x)-K$. We have $Av(t,x)=Au(t,x)$. Using in particular that $G$ is bounded, 
we obtain
\begin{eqnarray*}
\xi'(t)&=&\int  G(u(t,\cdot)-K)\frac{\partial u(t,\cdot)}{\partial t} d\mu\\
&=&-\int  G(u(t,\cdot)-K) A u(t, \cdot) d\mu\\
&=&-\int  G(v(t,\cdot))Av(t,\cdot)d\mu.
\end{eqnarray*}
When $A=-\widetilde \Delta_P$, by Theorem \ref{th_hille_yosida} $ v(t,\cdot)\in H^1(\mu),$ hence by Lemma \ref{lemma_e_tilde}, the last integral is equal to
$$-\int \langle \nabla  G(v),\nabla v\rangle_{g_P} d\mu =-\int G'(v)|\nabla v|^2d\mu\leq 0.$$
Thus, $ \xi'(t)\leq  0$. This, combined  with  $\xi(0)=0$ and  $\xi(t)\geq 0$ for $t\in\R^+,$ implies that $\xi=0.$ Hence
$u(x,t)\leq  K$. 

When $A=-\Delta_P$, by Theorem \ref{th_hille_yosida} $ v(t,\cdot)\in H^1(\mu),$  the considered integral is equal to
$$-\int G'(v)|\nabla v|^2d\mu+\int G(v)Fvd\mu=-\int G'(v)|\nabla v|^2d\mu+\int FH(v)d\mu.$$
By Lemma \ref{lemma_e_tilde}, the last integral vanishes.
So, we also obtain that $\xi'(t)\leq 0$. This completes the proof of
the maximum principle which implies the positivity of $S(t)$. 

The positivity of $S(t)$ together with the invariance of $\mu$ imply that 
$$\|S(t)u_0\|_{L^1(\mu)}\leq \|u_0\|_{L^1(\mu)} \quad \mbox{for} \quad u_0\in\Dc(\Fc).$$
It follows that $S(t)$ extends continuously to a positive contraction in $L^1(\mu)$ since $\Dc(\Fc)$ is dense in $L^1(\mu)$. 
The uniqueness of the solution in  Theorem \ref{th_hille_yosida} implies that $S(t)\chac=\chac$. This together with 
the positivity of $S(t)$ imply that $S(t)$ is a contraction in $L^\infty(\mu)$. Finally, the classical theory of interpolation between the Banach spaces $L^1(\mu)$ and $L^\infty(\mu)$ 
implies that $S(t)$ is a contraction in $L^p(\mu)$ for all $1\leq p\leq\infty$, see Triebel \cite{Triebel}.
\endproof

An important consequence of Theorem  \ref{th_heat_real} is  the following ergodic theorem.

\begin{theorem} \label{T:heat-real-ergodic}
Under the hypothesis of Theorem \ref{th_heat_real}, for all $u_0\in  L^p(\mu),$  $1\leq p<\infty$, the average 
$$\frac{1}{R}\int_0^R S(t)u_0 dt$$ 
converges pointwise  $\mu$-almost everywhere
 and also in $L^p(\mu)$ to 
an $S(t)$-invariant function  $u_0^*$ when $R$ goes to infinity.
Moreover, $u_0^*$ is constant on the leaf $L_a$ for $\mu$-almost every $a$.
If $\mu$ is an ergodic harmonic measure, then $u$ is constant $\mu$-almost everywhere.
\end{theorem}
  
\proof
By Theorem  \ref{th_heat_real} $S(t)$ is a positive contraction in $L^p(\mu)$ for all $1\leq p\leq \infty.$
Consequently,  the pointwise $\mu$-almost everywhere convergence   is a consequence of the ergodic theorem as in Dunford-Schwartz
\cite[Th. VIII.7.5]{DunfordSchwartz}. We get a   function $u_0^*$  which is  $S(t)$-invariant. The $L^p$ convergence follows from the $L^p$ ergodic   theorem of Von Neumann.
For the rest of the proof, since $S(t)$ is a contraction 
in $L^p(\mu)$  for , it is enough to consider the case where 
$u_0$ is in $\Dc(\Fc)$.

 We first prove  that $u_0^*$
is 
in the domain of $A$ and 
\begin{equation}\label{e:Au_0=0}
Au_0^*=0.
\end{equation}
Define
$$u_R:=\frac{1}{R}\int_0^R S(t)u_0 dt.$$ 
This function belongs to $\Dom(A)$. 
Since $u_R$ converges to $u_0^*$ in $L^2(\mu)$ and the graph of $A$ is closed in $L^2(\mu)\times L^2(\mu)$, it is enough to show
that 
$Au_R\rightarrow 0$ in $L^2(\mu)$. We have
$$Au_R=\frac{1}{R}\int_0^RAu(t,\cdot)dt=-\frac{1}{R}\int_0^R{\partial\over \partial t}u(t,\cdot)dt
={1\over R}u_0-{1\over R}u(R,\cdot).$$
Since $S(t)$ is a contraction in $L^2(\mu)$, the last expression tends to 0 in $L^2(\mu)$. This  proves \eqref{e:Au_0=0}. 

By Lemmas \ref{lemma_e_tilde} and \ref{lemma_e_e_tilde}, we deduce from  equality \eqref{e:Au_0=0}that
$$\int |\nabla u_0|^2d\mu=-\int(\widetilde\Delta_P u_0)u_0d\mu =-\int (\Delta_P u_0)u_0d\mu=0.$$
It follows that $\nabla u_0=0$ almost everywhere with respect to $\mu$. 
Thus, $u_0$ is constant on the leaf $L_a$ for $\mu$-almost every $a$. 
When $\mu$ is extremal, this property implies that $u_0$ is constant $\mu$-almost everywhere, since every measurable set of leaves has zero or full $\mu$ measure.
\endproof

We will need the following lemmas.

\begin{lemma} \label{lemma_harm_pos}
 Let $\widehat \mu=\theta \mu$ be a (signed) quasi-harmonic measure, not necessarily positive, where $\theta$ is a function in $L^2(\mu)$  (see Definition \ref{D:harmonic_measure}). 
Let $\widehat \mu=\widehat \mu^+ - \widehat \mu^-$ be the minimal decomposition 
of $\widehat \mu$ as the difference of two positive measures. Then $\widehat \mu^\pm$ are harmonic.

\end{lemma}
\proof
We start with  assertion (1). Let $S(t)$ be the semi-group of contractions in $L^1(m)$ associated with $-\Delta_P$ as above. Define the action of $S(t)$ on measures by
$$\langle S(t)\widehat \mu,u_0\rangle:=\langle \widehat \mu, S(t)u_0\rangle\quad \mbox{for}\quad u_0\in L^2(\mu).$$ 
Consider a function $u_0\in \Dc(\Fc)$ and define $\eta(t):=\langle S(t)\widehat \mu,u_0\rangle$. By Theorem \ref{th_hille_yosida}, this is a $\Cc^1$ function on $\R^+$. We have since 
$\widehat \mu$ is quasi-harmonic and $\theta$ is in $L^2(\mu)$
$$\eta'(t)=\langle \widehat \mu, S'(t)u_0\rangle=\langle \widehat \mu, -\Delta_P (S(t)u_0)\rangle =0.$$
To see the last equality, we can use a partition of unity and the local description of $\widehat \mu$ on a flow box. 
So, $\eta$ is constant. It follows that $S(t)\widehat \mu=\widehat
\mu$. Since $S(t)$ is a positive  
contraction, we deduce that $S(t)\widehat \mu^\pm = \widehat \mu^\pm$. So, the functions $\eta^\pm(t):=\langle \widehat \mu^\pm,S(t)u_0\rangle$ are constant.
As above, we have
$$\langle \widehat \mu^\pm,\Delta_P u_0\rangle = -\langle \widehat \mu^\pm, S'(0)u_0\rangle=(\eta^\pm)'(0)=0.$$
Hence, $\widehat \mu^\pm$ are quasi-harmonic. Since they are finite positive,  they are also  harmonic by Theorem \ref{thm_harmonic_currents_vs_measures} (3).
\endproof
 
 Let $\mu_1$ and $\mu_2$ be  two  finite positive measures.
  Define $\mu:=\mu_1+\mu_2$ and $\theta_i$ a function in $L^1(\mu)$, $0\leq \theta_i\leq 1$, such that $\mu_i=\theta_i\mu$. 
Define also $\mu_1\vee \mu_2:=\max\{\theta_1,\theta_2\}\mu$ and $\mu_1\wedge \mu_2:=\min\{\theta_1,\theta_2\}\mu$.

\begin{lemma}\label{L:simplex}
 Let $\mu_1$ and $\mu_2$ be  two  harmonic measures.
 Then 
 $\mu_1\vee \mu_2$ and $\mu_1\wedge \mu_2$ are  also  harmonic.
\end{lemma}
\proof
We use the notation introduced just before  the lemma.
By Lemma \ref{lemma_harm_pos}, since the signed finite measure $\mu_1-\mu_2$ is quasi-harmonic, $\mu':=\max\{\theta_1-\theta_2,0\}\mu$ is harmonic.
It follows that $\mu_1\vee \mu_2=\mu'+\mu_2$ and $\mu_1\wedge \mu_2=\mu_1-\mu'$ are harmonic. 
\endproof

We also obtain the following result, see Candel-Conlon \cite{CandelConlon2}.

\begin{corollary} \label{cor_choquet}
Under the hypothesis of Theorem \ref{th_heat_real}, the family $\Hc$ of harmonic probability measures of $\Fc$  is a
non-empty compact and  for any $\mu\in\Hc$ there 
is a unique probability measure $\nu$ on the set of extremal elements $\Ec$ in $\Hc$ such that 
$\mu=\int_\Ec m d\nu(m)$. Moreover, two different extremal harmonic probability measures are mutually singular.
\end{corollary}
 \proof
 Under the hypothesis of Theorem \ref{th_heat_real}, the family $\Hc$ of harmonic probability measures of $\Fc$  is a
non-empty. Clearly, $\Hc$ is  compact.   
By Choquet's representation theorem \cite{Choquet}, we can decompose $\mu$ into extremal measures as in the corollary. 
 Using  
Lemma \ref{L:simplex},  the uniqueness of the decomposition is  a consequence of the Choquet-Meyer theorem \cite[p.163]{Choquet}.
 \endproof

 The following result  will be needed in Subsection \ref{SS:Geometric-ergodic-thms}.

\begin{proposition} \label{prop_dense_laplace}
Let $\mu=\int_{m\in\Ec} md\nu(m)$ be  as in Corollary \ref{cor_choquet}. 
Then, the closures of $\Delta_P(\Dc(\Fc))$ and of $\widetilde\Delta_P(\Dc(\Fc))$ in $L^p(\mu)$, 
$1\leq p\leq 2$, are the space of functions 
$u_0\in L^p(\mu)$ such that $\int u_0dm=0$ for $\nu$-almost every $m$.  In particular, if $\mu$ is an 
ergodic harmonic probability measure, then this space is the hyperplane of $L^p(\mu)$ defined by the equation 
$\int u_0d\mu=0$.
\end{proposition}
\proof
We only consider the case of $\Delta_P$; the case of $\widetilde\Delta_P$ is treated in the same way. 
It is clear that $\Delta_P(\Dc(\Fc))$ is a subset of the space of $u_0\in L^p(\mu)$ such that $\int u_0dm=0$ for $\nu$-almost every $m$ and the last space is closed in $L^p(\mu)$.
Consider a function $\theta\in L^q(\mu)$, with $1/p+1/q=1$, which is orthogonal to $\Delta_P(\Dc(\Fc))$. So, $\theta \mu$ is a quasi-harmonic signed measure.
Since $p\leq 2$, we have $\theta\in L^2(\mu)$. 
We have to show that $\theta$ is constant with respect to $\nu$-almost every $m$. 

Consider the disintegration of $\mu$ along the fibers of $\theta$. There are a probability measure $\nu'$ 
on $\R$ and a probability measures $\mu_c$ on $\{\theta=c\}$ such that $\mu=\int_{c\in\R} \mu_c d\nu'(c)$.  
By Lemma \ref{L:simplex}, for any $c\in\R$, the measure $\max\{\theta,c\}\mu$ is quasi-harmonic. Therefore, $\mu_c$ is harmonic for $\nu'$-almost every $c$. 
If $\nu_c$ is the probability measure on  $\Ec$ associated with $\mu_c$ as in Corollary \ref{cor_choquet}, we deduce from the uniqueness in this corollary that 
$$\nu=\int_{c\in\R}\nu_c d\nu'(c).$$
Now, since $\theta$ is constant $\mu_c$-almost everywhere, it is constant with respect to $\nu_c$-almost every $m.$
So
$$
\int_{c\in\R}\int_{m\in  \Ec}\int_X  \big |\theta - \int_X\theta dm\big| d\nu_c(m)  d\nu'(c)=0.
$$
We deduce from the last two equalities  that $\theta$ is equal to the  constant $ \int_X\theta dm$ for $\nu$-almost every $m$. This completes the proof.
\endproof

The following result gives  a version of the mixing property in our
context. 
The classical case is due to 
Kaimanovich \cite{Kaimanovich} who uses in particular the smoothness
of the Brownian motion, see also Candel \cite{Candel03} who relies on a
version of the zero-two law due to Ornstein and Sucheston \cite{Ornstein}.

\begin{theorem} \label{th_mixing} {\rm  (\cite[Theorem 5.12]{DinhNguyenSibony12})}
Under the hypothesis of Theorem \ref{th_heat_real}, assume moreover that $\mu$ is ergodic. If $S(t)$ is associated to $-\widetilde\Delta_P$,  then
$S(t)u_0$ converge to $\langle \mu,u_0\rangle$ in $L^p(\mu)$ when $t\rightarrow\infty$ 
for $u_0\in L^p(\mu)$ with $1\leq p<\infty$. In particular, $S(t)$ is mixing, i.e.
$$\lim_{t\rightarrow\infty} \langle S(t)u_0,v_0\rangle=\langle \mu,u_0\rangle \langle \mu,v_0\rangle\quad \mbox{for}\quad u_0,v_0\in L^2(\mu).$$
\end{theorem}

\begin{remark}\rm
 When  the lamination $\Fc$ is  compact nonsingular  hyperbolic (i.e. $E=\varnothing,$ $X$ is compact and $\Hyp(\Fc)=X$) several results in this  subsection for $\Delta_P$ can be deduced from
Candel-Conlon \cite{CandelConlon2} and Garnett \cite{Garnett}.    But
results on  Riemann surface laminations  with  singularities and on $\widetilde\Delta_P$ are new.
 \end{remark}
 
 \subsection{When  do  the abstract diffusions coincide with the leafwise diffusions?} \label{SS:coincidence}
 
 As mentioned in the Introduction, the results of this subsection are new. 
 As  observed  in Subsection  \ref{SS:heat_equation_CHRSL},  the  abstract heat diffusions  associated  to the operator $A:=-\Delta_P$  and the  harmonic measure $\mu$ enjoy many  important  ergodic properties.
 The only  drawback is that  these diffusions  are  not  so  concrete. On the other hand,  the leafwise  heat diffusions given in  \eqref{e:diffusions}   have  the  advantage of being  quite
 explicit because  of concrete  formulas \eqref{e:heat_kernel}--\eqref{e:heat_kernel3}.  So the natural  question arises  whether  these  two heat  diffusions  coincide.
 The  following result gives an effective  criterion for  this coincidence. 
 
 
 \begin{theorem}\label{T:diffusions_comparisons}
 Let $\Fc=(X,\Lc,E)$ be a   Riemann surface lamination with singularities   endowed
with the leafwise Poincar\'e metric $g_P.$ Let $g_X$ be  a Hermitian  metric on the  leaves  which is  transversally continuous. Let $\mu$ be a  harmonic measure.
  Suppose that
  \begin{enumerate}
  \item[(i)] {\rm (local upper-boundedness of $\eta$)}
the  function $\eta$ given in \eqref{e:eta_bis} is locally  bounded from   above by strictly positive constants on $X\setminus E;$  
 \item[(ii)] {\rm  (membership test)}  if $u$ is a measurable function on $X\setminus E$  such that $\|u\|_{L^\infty}<\infty,$   $\| |du|_P\|_{L^\infty}<\infty$ and $\|\Delta_Pu\|_{L^\infty}<\infty,$
 then  $u$ belongs necessarily to  $H^1(\mu)$ (for  the notation $|du|_P$  see \eqref{e:length-eucl-vs-Poincare}  and  for $H^1(\mu)$ see \eqref{e:norm_H1}).
  \end{enumerate}
  Then the  abstract heat  diffusions  associated to $\mu$ coincide with the leafwise heat diffusions. 
 
 \end{theorem}

   Let $S(t)$,
$t\in\R_+$, denote  the semi-group of contractions associated with the operator
   $-\Delta_P$ given by    Proposition \ref{prop_delta_max}.
   Then    the conclusion of Theorem  \ref{T:diffusions_comparisons}    says that
   \begin{equation*} S(t)u=D_tu\qquad \text{for}\qquad  u\in \Dom(-\Delta_P)\qquad\text{and}\qquad  t\in\R^+,
   \end{equation*}
 where $D_t$ is  the leafwise heat diffusion given in \eqref{e:diffusions}.
   
We  start the proof with  the following observation.
Since    $\Dc(\Fc)$ is  dense in $\Dom(-\Delta_P)=H_P(\mu)\subset L^2(\mu)$  by  Remark  \ref{R:Dom-Delta} and both diffusions  are positive contractions,  we only need to  prove that 
for   an arbitray  non-negative function $u_0\in\Dc(\Fc),$  
\begin{equation}\label{e:reduction_coincidence}
S(t)u_0=D_tu_0 \qquad\text{for}\qquad  t\in\R^+.
\end{equation}
Fix  such  a function $u_0.$
Since $\eta$  is locally  bounded from above on $X\setminus E,$  we infer that $\Par(\Fc)=\varnothing.$
  By Theorem  \ref{th_hille_yosida},
there is  a  unique  
 $\Cc^1$ function  $U$ from $\R^+$ to $L^2(m)$ with values in
$\Dom(-\Delta_P)$ which satisfies
\begin{equation}\label{e:heat-equation_bis}
{\partial U(t,\cdot)\over \partial t}-\Delta_P U (t,\cdot)=0\quad
\mbox{and}\quad U(0,\cdot)=u_0.
\end{equation}
Consider  the function $u:\ \R^+\times (X\setminus E)\to\R$  defined by
\begin{equation}\label{e:heat_diffusion_sol}
u(t, \cdot):=D_tu_0\end{equation} 
 To  complete the proof of \eqref{e:reduction_coincidence} it suffices to show   that
 \begin{equation}\label{e:U=u}
  U=u.
 \end{equation}
 
 We will prove \eqref{e:U=u} using the proof of  the  uniqueness in Theorem \ref{th_hille_yosida}. The  proof relies on  three facts.
 The first fact is  that 
 the function $u(t,\cdot):\  X\setminus E\to\R$ constructed in \eqref{e:heat-equation_bis}  is  measurable and it satisfies
\begin{equation}\label{e:Step1}
  \| |u(\cdot,\cdot)|_P\|_{L^\infty}<\infty,\quad\text{and for each $t>0,$}\quad  \| |du(t,\cdot)|_P\|_{L^\infty}<\infty,\quad \|\Delta_Pu(t,\cdot)\|_{L^\infty}<\infty.
 \end{equation}
The  second  fact is  that
\begin{equation} \label{e:Step2}
u\in\Cc^1(\R^+_*,L^2(\mu)) \quad \text{and}\quad {\partial u(t,\cdot)\over \partial t}-\Delta_Pu(t,\cdot)=0\quad\mathrm{for}\quad t\in\R^+_*
\end{equation}
The third fact is  the limit
\begin{equation}\label{e:Step3}
\lim_{t\to 0} u(t,\cdot)=u_0\qquad\mathrm{in}\qquad L^2(\mu).
\end{equation}
Taking for granted  these three facts,  we arrive at the

\proof[End of the  proof of Theorem  \ref{T:diffusions_comparisons}]
Since the   function $u:\ \R^+\times (X\setminus E)\to\R$ enjoys the properties stated in the first fact \eqref{e:Step1},
we infer from   the  membership test (ii) that  $u(t,\cdot)$ belongs  to  $H^1(\mu)$ for each $t\in\R^+_*.$
So  does $U(t,\cdot)-u(t,\cdot).$ Consequently, we  infer from  the  second fact  \eqref{e:Step2} 
that
\begin{equation*}
\langle {\partial \over \partial t}(U(t,\cdot)-u(t,\cdot) ) ,  U(t,\cdot)-u(t,\cdot)\rangle=\langle -\Delta_P(U(t,\cdot)-u(t,\cdot)) ,U(t,\cdot)-u(t,\cdot)\rangle\quad\text{for}\quad t\in\R^+_*.
\end{equation*}
Fix a $t\in\R^+_*$  and  consider the function    $\tilde u:=U(t,\cdot)-u(t,\cdot)\in H^1(\mu).$ By  \eqref{e:norm_H1}, for  every $\epsilon>0$  there is  $\tilde u_\epsilon\in\Dc(\Fc)$
such that $\|\tilde u-\tilde u_\epsilon\|_{H^1(\mu)} <\epsilon.$
On the one hand,  we have  $\tilde u_\epsilon\in  H_P(\mu)$  (see  \eqref{e:norm_P}) and $\tilde u\in H^1(\mu).$ On the other hand,   by \eqref{e:heat-equation_bis} and \eqref{e:Step1}--\eqref{e:Step2}  
we get
$\Delta_P\tilde u=\Delta_P U(t,\cdot)-\Delta_Pu(t,\cdot)\in L^2(\mu).$
Consequently,  by Lemma \ref{lemma_e_tilde} (see Remark \ref{R:Dom-Delta}) and  the inequality $\int i\partial \tilde u\wedge \dbar \tilde u\wedge T\geq 0,$ we infer that
$$
\tilde q(\tilde u,\tilde u_\epsilon)=\int i\partial \tilde u\wedge \dbar \tilde u_\epsilon\wedge T=\int i\partial \tilde u\wedge \dbar \tilde u\wedge T+O(\epsilon)\geq O(\epsilon).
$$
Letting $\epsilon$ tend to $0,$  the  above  inequality  together  with  the estimate $\|\tilde u-\tilde u_\epsilon\|_{H^1(\mu)} <\epsilon$ imply that 
$\tilde q(\tilde u,\tilde u)\geq 0.$  Hence, by  Lemma \ref{lemma_e_e_tilde} we get that  $  q(\tilde u,\tilde u)=\tilde q(\tilde u,\tilde u)\geq 0.$ So  for $ t\in\R^+_*,$
\begin{equation*}
\langle {\partial \over \partial t}(U(t,\cdot)-u(t,\cdot) ) ,  U(t,\cdot)-u(t,\cdot)\rangle=\langle -\Delta_P(U(t,\cdot)-u(t,\cdot)) ,U(t,\cdot)-u(t,\cdot)\rangle\leq 0.
\end{equation*}
Hence,
$$
{1\over 2} {\partial  \over \partial t} \|U(t,\cdot)-u(t,\cdot) \|^2_{L^2(\mu)}=\langle {\partial  \over \partial t} (U(t,\cdot)-u(t,\cdot)),  U(t,\cdot)-u(t,\cdot)\rangle\leq 0.
$$
So the  function $\R^+_*\ni t\mapsto  \|U(t,\cdot)-u(t,\cdot) \|^2_{L^2(\mu)}$ is  decreasing.
Using  the  third fact \eqref{e:Step3} we see that
 $U(0,\cdot)-u(0,\cdot)=u_0-u_0=0$ and  that $\lim_{t\to 0} U(t,\cdot)-u(t,\cdot)=0$ in $L^2(\mu).$
 Hence,  $U(t,\cdot)-u(t,\cdot)=0$ for all $t\in\R^+.$ 
The proof of \eqref{e:U=u} is thereby completed.
\endproof

  To  finish the  proof of Theorem  \ref{T:diffusions_comparisons} we need  to prove  the three facts. To this end   we do some simplifications.
  Using  a countable partition of unity of $X\setminus E$, we may  assume  without loss of generality that
the function $u_0$ is compactly  supported  in a given flow box $\U\simeq\B\times \T,$  where $\B$ is  simply the unit disc $\D.$
The  center of  a plaque $\B\times a$ with $a\in\T$ is  the point $0 \times a,$ where $0$ is  the center of $\D.$ 
Fix  a point $x_0\in  X\setminus E$ and  a  finite time $t_0>0.$
Let $(\Omega_j)_{j\in J}$  be all  connected components of $\phi_{x_0}^{-1}(\U)\subset \D.$
So the  (eventually  empty) index set $J$ (depending on $\U$ and $x_0$) is at most countable. 
Since  $\eta$ is  locally bounded from above and below by strictly positive constants outside $E,$    there is    a  positive constant $c_1>1$ such that
 $c_1^{-1} g_X\leq g_P\leq c_1g_X$ on plaques of  
 $\U.$
Observe   that  to travel from  the center of a plaque of $\U$ to the center of another plaque, we cover a distante $\geq c_2>0$  with respect to $g_X,$  hence  a distance $\geq c_1^{-1}c_2$
with respect to    $g_P.$
Consequently,  there are constants $c_3,c_4> 0$ (which depends only on $\U$ and which does not depend on $x_0$) such that 
\begin{equation}\label{e:property_Omega} \diam_P(\Omega_j)\leq c_3\quad\text{and}\quad  \dist_P(\Omega_p,\Omega_q)\geq c_4\quad\text{  for}\quad p\not=q,\quad p,q\in J.
 \end{equation}
Let $v:\ \R^+\times\D\to \R$ be the function defined on  by  
\begin{equation}\label{e:v} v(t,\zeta)=u(t,\phi_{x_0}(\zeta))\qquad\text{for}\qquad t\in\R^+,\ \zeta\in\D.
\end{equation}
For $j\in J$ let 
\begin{equation}\label{e:uj} u_j:=u_0\circ\phi_{x_0}\qquad\mathrm{on}\qquad \Omega_j.
\end{equation} Since  $u_0$ is compactly  supported  in the flow box $\U\simeq\B\times \T,$ it follows that $u_j\in\Dc(\Omega_j).$ 
Now $u_j$ can be regarded as  a nonnegative  smooth function compactly supported  on $\D$ by simply setting $u_j=0$ outside $\Omega_j.$
Set  \begin{equation}\label{e:vj} v_j(t,z)=(D_t u_j)(z)\qquad\text{ for}\qquad z\in \D.
\end{equation} So   we have
\begin{equation}
\label{e:v-vj}v=\sum_{j\in J} v_j.
 \end{equation}
Since  $u_j\in\Dc(\D),$  we deduce from  \eqref{e:heat-equation}  that on $\D$
\begin{equation}\label{e:heat-equation_bisbis}
{\partial v_j(t,\cdot)\over \partial t}-\Delta_P v_j(t,\cdot)=0\quad
\mbox{and}\quad v_j(0,\cdot)=u_j.
\end{equation}
\begin{lemma}\label{L:estimates_kernel}   For every $t_0>0,$
there is a constant $c_5$  which depends only on $\U$ and $t_0$ (so $c_5$ does not depend on $x_0\in X\setminus E$)  such that
\begin{equation*}
\sum_{j\in J} \sup\limits_{t\in[t_0/2,2t_0] }\big( \sup\limits_{ (x,y)\in\D^2: \ \dist_P(x,0)<1,\ y\in \Omega_j} \big|\Phi(x,y,t)\big|\big) <c_5
\end{equation*}
in the  following cases:
\begin{enumerate}
 \item $\Phi(x,y,t)=p_\D(x,y,t);$
 \item $\Phi(x,y,t)= {\partial p_\D(x,y,t)\over \partial t};$
 \item $\Phi(x,y,t)=(\Delta_P)_x p_\D(x,y,t);$
 \item $\Phi(x,y,t)=|d_x p_\D(x,y,t)|_P.$
\end{enumerate}
\end{lemma}
\proof
Since the proof of Case (1) is similar  and even  simpler that that of Case (2),  we  treat directly Case (2).  

Denote the Poincar{\'e} distance between $x$ and $y$ by $\rho:=\dist_P(x,y).$ 
We  deduce from formula 
\eqref{e:heat_kernel} and  identity \eqref{e:heat_kernel3} that
$$p_\D(x,y,t)={\sqrt{2}e^{-t/4} \over (2\pi t)^{3/2}}
\int_{\rho}^{\infty}{s  e^{-{s^2 \over 4t}}\over \sqrt{\cosh s-\cosh\rho}}ds\qquad\text{for}\qquad t\in\R^+.$$
It can be checked that $\sqrt{\cosh s-\cosh\rho}\gtrsim e^{s/2}$ when $s\geq \rho+1$ and 
$\sqrt{\cosh s-\cosh\rho}\gtrsim e^{s/2}\sqrt{s-\rho}$ for $\rho<s\leq \rho+1$.
Consequently, 
\begin{equation*}
  {\partial p_\D(x,y,t)\over \partial t}\lesssim  e^{-t/4} (  t^{-3/2}+t^{-5/2}+ t^{-7/2})
\big( \int_{\rho+1}^{\infty}s^3  e^{-{s^2 \over 4t}-{s\over 2}}ds   +   \int_\rho^{\rho+1}{ s^3 e^{-{s^2 \over 4t}-{s\over 2}}  \over \sqrt{ s-\rho}}ds\big).
\end{equation*}
We have the elementary inequality  for $t\in[t_0/2,2t_0]$ and $s>0,$
$$
 e^{-{s^2 \over 4t}} \leq  C(N,t_0)  e^{-Ns},
$$ 
where $N=N(t_0)$ can be made arbitrarily large and $c(N),t_0$ is a constant which depends on $N$ and $t_0.$
Since both integrands  on the right hand side of the  estimate for $ {\partial p_\D(x,y,t)\over \partial t}$ contain  the term  $e^{-{s^2\over 4t}},$  
we infer  from the  last line  that  for $t\in[t_0/2,2t_0],$
\begin{equation*}
 {\partial p_\D(x,y,t)\over \partial t}\leq  c(N,t_0) e^{-N\rho}.
\end{equation*}
So Case (2) will follow if one can show that
\begin{equation*}
\sum_{j\in J}  \sup\limits_{\dist_P(x,0)<1,\ y\in \Omega_j} e^{-N\dist_P(x,y)} <c_5.
\end{equation*}
Using  \eqref{e:property_Omega}  we  see that  when $N$ is chosen  large  enough with respect to $c_3$ and $c_4,$ the  above  inequality holds.
This completes the proof of Case (2).

Using  the  first identity of \eqref{e:heat-equation} and  the  symmetry $p_\D(x,y,t)=p_\D(y,x,t)$  (see \cite{Chavel}),  Case (3) is  equivalent to Case (2).

To  prove  Case (4) the following elementary result is needed
\begin{lemma}  \label{L:Poisson}
 Let $f\in \Cc^2(2\D).$ Then
 $$  \sup_{\D}\|df\|\leq c_3 ( \sup_{2\D}\|f\|+ \sup_{2\D}\|\ddc f\|).$$
 Here  we recall  that $\D_R$ denotes the Poincar\'e disc with center $0$ and radius $R.$
\end{lemma}
\proof[Proof of Lemma \ref{L:Poisson}]
Let $f$ be  a   function  on $2\D$ such that $\mu:=\ddc f$ is a Radon measure of total finite mass $\|\mu\|.$
 Let  $\Pc[\mu]$ be the  function  defined on $2\D$  by
 \begin{equation*}
 \Pc[\mu](z):={1\over\pi} \int_{\zeta\in 2\D} \log|z-\zeta| d\mu(\zeta),\qquad z\in 2\D.
 \end{equation*}
 Then we  deduce from this  integral  formula that
  $ \|  d \Pc[\mu]\|_{L^\infty(\D)}\leq  c\|\mu\|_{L^\infty}.$
  Moreover,   $u-\Pc[\mu]$ is a  harmonic  function on $2\D$ and   
  $$\|f-\Pc[\mu]\|_{\Cc^1(\D)}\leq c' \|f-\Pc[\mu]\|_{L^1((3/2)\D)}  \leq c \big(\|\mu\|+   \|f\|_{L^1(2\D)}\big).$$
   Here $c,$ $c'$ are constants  independent of  $f.$ 
  This completes the proof of the lemma.
\endproof
Using  this lemma,  Case (4) follows from combining  Case (1) and Case (3).
\endproof

\proof[Proof  of the first fact \eqref{e:Step1}] 
Since the norm of $D_t$ on $L^\infty(X)$ is  1,  we get that $\| u(\cdot,\cdot)\|_{L^\infty}\leq \|u_0\|_{L^\infty}.$
This proves  the  first  inequality in  \eqref{e:Step1}.

Fix  $t_0>0.$  We prove    the  following  stronger  inequality than  the  second  inequality in  \eqref{e:Step1}
$\sup_{t\in[t_0/2,2t_0]}\| |du(t,\cdot)|_P\|_{L^\infty}<\infty.$
In fact,  it is  enough to  show that 
$$\sup_{t\in[t_0/2,2t_0],\  x\in\phi_{x_0}(\D_1)}\| |d_xu(t,x)|_P\|_{L^\infty}<c,$$
where $c$ is a constant  independent of $x_0.$
Fix a  point $x_0\in X\setminus E.$
Using  \eqref{e:uj}, \eqref{e:vj}, \eqref{e:v-vj}, we get for  $t\in[t_0/2,2t_0]$ and $x\in\D_1$  that
\begin{equation}\label{e:sum-v}
 v(t,x)=\sum_{j\in J} v_j(t,x)= \sum_{j\in J}\int_{y\in\Omega_j} p_\D(x,y,t) u_0(\phi_{x_0}(y))dg_P(y).
\end{equation}
On the  other hand,  the first inequality of \ref{e:property_Omega} implies that there is a constant $c$ independent of $j\in J$ such that
$\int_{\Omega_j} dg_P(y)<c.$
Therefore, applying   Lemma \ref{L:estimates_kernel}     to 
  $\Phi(x,y,t)=|d_x p_\D(x,y,t)|_P$ and using Lebesgue dominated convergence in order   to  differentiate the sum in \eqref{e:sum-v}, we  get that
  $$
  |d_xv(t,x)|_P\leq c_5\|u_0\|_{L^\infty}\qquad\mathrm{for}\qquad t\in[t_0/2,2t_0]\qquad\mathrm{and}\qquad x\in\D_1.
  $$
 This,  combined  with \eqref{e:heat_diffusion_sol} and \eqref{e:v}, implies the desired  conclusion.
 
 For the last inequality  in  \eqref{e:Step1} it is  enough to  show that 
$$ \sup_{t\in[t_0/2,2t_0],\  x\in\phi_{x_0}(\D_1)}\| \Delta_Pu(t,x)\|_{L^\infty}<c,$$
where $c$ is a constant  independent of $x_0.$ We argue  as  above 
using   this time   Lemma \ref{L:estimates_kernel} applied   to 
    $\Phi(x,y,t)= (\Delta_P)_xp_\D(x,y,t).$  
  \endproof

\proof[Proof  of the second fact \eqref{e:Step2}]
We  use  \eqref{e:v-vj}, \eqref{e:vj}, \eqref{e:uj}.   We also
apply   Lebesgue's dominated convergence and Lemma \ref{L:estimates_kernel} for $\Phi(x,y,t)= {\partial p_\D(x,y,t)\over \partial t}.$ Consequently, we infer from equality
\eqref{e:heat-equation_bisbis} that
$$
{\partial v(t,x)\over \partial t}=\sum_{j\in J} {\partial v_j(t,x)\over \partial t}= \sum_{j\in J} \Delta_P v_j(t,x)=\Delta_P\sum_{j\in J}  v_j(t,x)=\Delta_Pv(t,x)
$$
 for  $t\in[t_0/2,2t_0]$ and $x\in\D_1.$
So  $$ 
{\partial u(t,x)\over \partial t} =\Delta_Pu(t,x).$$  
 This,  combined  with \eqref{e:heat_diffusion_sol} and \eqref{e:v}, implies the desired  conclusion.
 \endproof

To prove  the  third fact, we need  a counterpart of Lemma \ref{L:estimates_kernel} when  the time   $t$  is small
\begin{lemma}\label{L:estimates_kernel-bis} 
There is a constant $c_6$  which depends only on $\U$ and $t_0$ (so $c_6$ does not depend on $x_0\in X\setminus E$)  such that   for every $0<t_0<1,$ we have
\begin{equation*}
\sum_{j\in J} \sup\limits_{t\in[0, t_0]}  \big(  \sup\limits_{(x,y)\in\D^2: \ \dist_P(x,0)<1,\ \dist_P(y,0)>2,\ y\in \Omega_j} p_\D(x,y,t) \big)<c_6 t_0.
\end{equation*}
\end{lemma}

\proof 
Denote the Poincar{\'e} distance between $x$ and $y$ by $\rho:=\dist_P(x,y).$ We deduce from $ \dist_P(x,0)<1$ and  $\dist_P(y,0)>2$ that $\rho>1.$
As in the  proof of Lemma \ref{L:estimates_kernel}, using  formula 
\eqref{e:heat_kernel} and  identity \eqref{e:heat_kernel3}  as well as the estimates $\sqrt{\cosh s-\cosh\rho}\gtrsim e^{s/2}$ when $s\geq \rho+1$ and 
$\sqrt{\cosh s-\cosh\rho}\gtrsim e^{s/2}\sqrt{s-\rho}$ for $\rho<s\leq \rho+1,$ we get that
\begin{equation*}
  p_\D(x,y,t) \lesssim  e^{-t/4}   t^{-3/2}
\big( \int_{\rho+1}^{\infty}s  e^{-{s^2 \over 4t}-{s\over 2}}ds   +   \int_\rho^{\rho+1}{ s e^{-{s^2 \over 4t}-{s\over 2}}  \over \sqrt{ s-\rho}}ds\big).
\end{equation*}
We use  the following elementary  estimate: for $0<t<t_0\ll 1$ and  $s>1$
$$
  e^{-t/4}   t^{-3/2} e^{-{s^2 \over 4t}} \leq  C(N,t_0)  t^{4}e^{-Ns},
$$
 where $N=N(t_0)$ can be made arbitrarily large when $t_0>0$ is  sufficiently  small, and $c(N,t_0)$ is a constant which depends on $N$ and $ t_0.$
Since both integrands  on the right hand side of  the  estimate  for $p_\D(x,y,t)$  contain  the term  $e^{-{s^2\over 4t}},$ we infer from the last line that for $0<t<t_0\ll 1$ and  $\rho>1$
that
\begin{equation*}
 p_\D(x,y,t)\leq  c(N,t_0) t_0e^{-N\rho},
\end{equation*}
So the lemma  follows from the  inequality
\begin{equation*}
\sum_{j\in J}  \sup\limits_{(x,y)\in\D^2: \ \dist_P(x,0)<1,\ y\in \Omega_j} e^{-N\dist_P(x,y)} <c_5,
\end{equation*} 
which  has been estiablished  at the  end of the proof of Lemma \ref{L:estimates_kernel}.
\endproof

Now we complete the  proof of Theorem \ref{T:diffusions_comparisons}.

\proof[Proof  of the third fact \eqref{e:Step3}]
Let $j_0\in J$  be the (unique if exists)  index  such that $\Omega_{j_0}$ contains $0.$  So the  existence of  such $j_0$ is equivalent  to  the condition that $x_0\in\U.$  We will see that
if there is  no such an index $j_0$ then the rest of the proof is  trivially finished.
Using  Lemma  \ref{L:estimates_kernel-bis},  we  see that
$$
\sum_{j\in J,\ j\not=j_0}|v_j(t,\zeta)|\leq Ct_0\quad \mathrm{for}\quad t\in(0,t_0],\quad \zeta\in\D.
$$
On the other hand, applying  the  
second  equation in \eqref{e:heat-equation} to $u_{j_0}\in\Dc(\Omega_{j_0})$ 
yields  that 
$\lim_{t\to 0}v_{j_0}(t,\cdot)=u_{j_0}$ on $\D.$  
 Combining this with  the  previous  estimate, and  using  \eqref{e:v-vj} and  \eqref{e:uj},
 we  see that $\lim_{t\to 0}v(t,\cdot)=u_0\circ \phi_{x_0}$ on $\D.$ 
 So  by  \eqref{e:v}
 and \eqref{e:heat_diffusion_sol}, $\lim_{t\to 0}u(t,\cdot)=u_0$ everywhere on $X\setminus E.$
 By   the first fact \eqref{e:Step1},  $u(t,\cdot)$ is uniformly bounded. Therefore, by Lebesgue dominated convergence 
we get the conclusion \eqref{e:Step3} of the  third fact.
\endproof
 
 \begin{corollary}\label{C:diffusions_comparisons}
 Let $\Fc=(X,\Lc)$ be   a  compact  hyperbolic Riemann surface lamination (without singularities). Let $\mu$ be a harmonic measure.
 Then the  abstract heat  diffusions associated to $\mu$ coincide with the leafwise heat diffusions. 
 \end{corollary}
\proof
Observe that the function  $\eta$ is uniformly  bounded from above  and below by strictly positive constants. So  condition (i) is  satisfied.
To prove the  membership test (ii), pick  a  function $u$ as in (ii). Since  $X$ is compact, we  cover it by a finite number of flow boxes $\U\simeq \B\times \T,$
where $\B$ is  the unit-disc $\D.$
On each flow box $\U,$ we use   standard  convolution  on $\B$ in order to regularize $u$  on plaques  and glue the obtained approximants on these plaques  together  in  a uniform  way.
Using a  finite partition of unity associated to the finite cover $\U,$ we  obtain  a  sequence $u_n\in\Dc(\Fc)$ of approximants of $u$   such that  $\|u_n-u\|_{H^1(\mu)}\to 0$  as $n$
tends to infinity.  
The interested reader may also see the proof of Proposition \ref{P:diffusions_comparisons} below for more details. This proves (ii).  

Applying now  Theorem \ref{T:diffusions_comparisons} gives the result.
\endproof
 
 \begin{proposition}\label{P:diffusions_comparisons}
  Let $\Fc=(X,\Lc,E)$ be  
 a   singular holomorphic foliation (not necessarily compact)  such  that $E$ is a finite set. Assume that  the complex manifold $X$ is  endowed with a  Hermitian metric $g_X$   and    
the  function $\eta$ given in \eqref{e:eta_bis} is locally  bounded from   above by strictly positive constants on $X\setminus E$
and that there is a constant $c>0$ such that  $\eta(x)\geq c  \dist(x,E)$ for $x\in X.$  Here $\dist$ is  the distance associated to the   Hermitian metric on  $X.$
 Let  $T$ be a  directed positive   harmonic current   such that the Lelong number $\nu(T,x)=0$ for all $x\in E.$ Let $\mu$ be the positive quasi-harmonic measure given by  \eqref{e:m-T}.
 Suppose that $\mu$
  finite. Then the  abstract heat  diffusions associated to $\mu$ coincide with the leafwise heat diffusions.
 \end{proposition}
\proof
We only need to check the  membership test (ii) in Theorem \ref{T:diffusions_comparisons}.

Recall  that for $r>0,$  $\B_r$ denotes  the ball in $\C^k$ centered at $0$   with radius $r.$
We often denote the unit ball  by $\B,$  i.e.  $\B=\B_1.$
Fix  a nonnegative  smooth  function $\chi:\ \C^k\to  [0,1]$ such that  $\chi(z)=1$ for $|z|\leq 1/2$ and $\chi(z)=0$ for $|z|\geq 1.$
Let $\epsilon_0:= \int_{\C^n} \chi(z)d\vol(z),$  where $\vol(z)$ denotes the Lebesgue measure of $\C^k.$  For $0<\epsilon <1$ consider the function $\chi_\epsilon(z):= \epsilon^{-1}_0\epsilon^{-2k} \chi ( z/\epsilon),$
$z\in\C^k.$  So $ \int \chi_\epsilon(z) d\vol(z)=1.$

Consider  a  countable  or finite   cover $\Uc=(\U_p)_{p\in I}$ of $X$ (of dimension $k$) by  open sets  such that 
\begin{itemize}
 \item[$\bullet$] 
  for each $p\in I,$ $\U_p$ is  the unit ball $\B$ of $\C^k$ in a  suitable  local coordinate system;
\item[$\bullet$] for each $p\in I\setminus E,$ $\U_p$ is  a  flow box and  in a  foliated chart $\U_p \simeq \D\times \T_p,$
with $\D$  as  usual the unit disc and $\T_p$   an open  set in $\C^{k-1},$
and  $2\U_p\cap E=\varnothing,$  where  $2\U_p$ is $\B_2$  in  the above  local coordinate system;
\item[$\bullet$]  for each $a\in E,$   $\U_a\cap E=\{a\}.$ 
\end{itemize}

Let  $u$  be  a measurable   function on $X\setminus E$  such that 
\begin{equation}\label{e:assumption-u} \|u\|_{L^\infty}<\infty,\qquad\||du|_P \|_{L^\infty}<\infty,\qquad \|\Delta_Pu\|_{L^\infty}<\infty.
\end{equation}
 We need  to  prove  that $u$ belongs  to  $ H^1(\mu).$
 Since  $\eta$  is locally  bounded from   above by strictly positive constants on $X\setminus E.$ it is  uniformly bounded from   above  and below by strictly positive constants on
 $\bigcup_{p\in I\setminus E} 2\U_p.$  Consequently, 
using a  partition of unity subordinate to $\Uc$ and \eqref{e:assumption-u} we may assume  that $u$ is  compactly supported in a  single $\U_p.$ 
There are two cases  to consider.  

\noindent {\bf  Case}  $p\not\in E.$

So $u$ is  compactly supported in  a flow box $\U_p.$
For $n\geq 1$  consider the convolution $u_n:=u\star\chi_{1/n},$  where
$$
(u\star\chi_\epsilon)(z):=\int_{\C^k}u(z-w)\chi_\epsilon(w)d\vol(w)\qquad \text{for}\qquad z\in\C^k.
$$
Observe  that $u_n\in\Dc(\Fc).$ Since
 $\eta$    is  uniformly bounded from   above  and below by strictly positive constants on
 $ 2\U_p,$ we infer  from   \eqref{e:assumption-u} that   $\|u_n-u\|_{H^1(\mu)}\to 0$ as  $n$ tends to infinity.
 Hence,   $u$ belongs  to  $ H^1(\mu).$

\noindent {\bf  Case}  $p\in E.$

Let $a$ be  the  unique point $E\cap \U_p.$ 
Using the local coordinate system associated to $\U_p$ we may assume that $a=0$ and $\U_p=\B.$
So $u$ is compactly supported in $\B.$ For $0<\epsilon<1$ consider the  function $v_\epsilon:\ \C^k\to\R$  defined by
$$
v_\epsilon(z):=(1 -\chi(\epsilon^{-1}z)) u(z)\qquad\text{for}\qquad  z\in\C^k.
$$
Observe  that  $v_\epsilon=u$  outside $\B_\epsilon,$  $v_\epsilon=0$  on $\B_{\epsilon/2},$
and  $|u-v_\epsilon|\leq |u|$ everywhere.  Moreover, it follows from \eqref{e:assumption-u} that
$v_\epsilon$ also satisfies  \eqref{e:assumption-u} for every $0<\epsilon <1.$
Since the  support of $v_\epsilon$ does not meet $E,$  we can argue as in  Step 1  to prove  that $v_\epsilon\in H^1(\mu)$ for  every $0<\epsilon <1.$
So  the  proof of the membership  $u\in H^1(\mu)$  will  follow if one can show  that
\begin{equation}\label{e:reduction-u-v}
 \lim\limits_{\epsilon\to 0}\|  u-v_\epsilon\|_{H^1(\mu)}=0.
\end{equation}
On the one hand,  since $\int_X| u|^2 d\mu<\infty$  and $\mu(\{a\})=0,$ we  get 
$$
\int_X| u-v_\epsilon|^2d\mu=\int_{\B_\epsilon} | u-v_\epsilon|^2d\mu\leq \int_{\B_\epsilon}| u|^2d\mu\to 0\quad\text{as}\quad \epsilon\to 0.
$$
On the  other hand,  we have that
$$
\int_X | \nabla( u-v_\epsilon)|^2 d\mu=  \int_X i\partial(u-v_\epsilon)\wedge  \dbar (u-v_\epsilon)\wedge T= \int_{\B_\epsilon} i\partial\big(u(z)\chi(\epsilon^{-1}z)\big)\wedge
\dbar\big (u(z)\chi(\epsilon^{-1}z)\big)
\wedge T(z).$$
The  integral on the  right hand  side is  dominated by a constant times 
\begin{equation*}
 \int_{\B_\epsilon} i\partial u\wedge  \dbar u \wedge T
+\epsilon^{-2} \int_{\B_\epsilon\setminus \B_{\epsilon/2} } |u|^2 T\wedge \beta
+ \epsilon^{-1}\int_{\B_\epsilon\setminus \B_{\epsilon/2} } |d u|\cdot T\wedge\beta
=: I_1+I_2+I_3.
\end{equation*}
Here $\beta:=\ddc\|z\|^2$  (see the notation in Proposition  \ref{P:Skoda}).
By \eqref{e:assumption-u}, $ \|u\|_{L^\infty}<\infty$ and  $ \|\nabla u\|_{L^\infty}<\infty.$  Moreover, $\mu(\{a\})=0.$  So
$$
I_1=\int_{\B_\epsilon} |\nabla u|^2d\mu \lesssim  \int_{\B_\epsilon} d\mu \to  0\quad\text{as}\quad  \epsilon\to 0.
$$
On the other hand, by Proposition  \ref{P:Skoda} and Definition \ref{D:Lelong}, we  have
$$
I_2 \lesssim \epsilon^{-2}\int_{\B_\epsilon }  T\wedge\beta\to 0 \quad\text{as}\quad  \epsilon\to 0,
$$
because  by the hypothesis $\nu(T,a)=0.$
By \eqref{e:length-eucl-vs-Poincare}   and \eqref{e:nabla-vs-dP},  we get that
$$
 |du(z)|={|du(z)|_P\over \eta(z)}={|\nabla u(z)|\over \eta(z)}\quad\text{for}\quad z\in\B.
$$
By the  hypothesis, $\eta(z)\gtrsim  c \epsilon$ for  $z\in {\B_\epsilon\setminus \B_{\epsilon/2} }.$
Therefore,  we  argue  as in  estimating $I_2$ that  for $z\in {\B_\epsilon\setminus \B_{\epsilon/2} },$
$$
I_3 \lesssim \epsilon^{-2}\int_{\B_\epsilon }  T\wedge \beta\to 0\quad\text{as}\quad  \epsilon\to 0.
$$
The proof of  \eqref{e:reduction-u-v} is  thereby completed. This  finishes the proof of the proposition.



\endproof

 \begin{corollary}\label{C:diffusions_comparisons_bis}
 If 
 $\Fc=(X,\Lc,E)$ is a Brody hyperbolic compact  singular holomorphic foliation.  Suppose that  all the singularities    are    linearizable as well as weakly  hyperbolic. Then for every 
   positive quasi-harmonic  measure,   the  abstract heat  diffusions  coincide with the leafwise heat diffusions.
   In particular,   the  abstract heat  diffusions are unique, i.e., they are independent of the  quasi-harmonic measures.
 \end{corollary}
\proof  Since the foliation $\Fc$ is  compact    with only linearizable singularities,
By Proposition  \ref{P:Poincare_mass}, the mass of every positive  quasi-harmonic  measure $\mu$ is  finite. 
Therefore, by Theorem  \ref{thm_harmonic_currents_vs_measures} (3), every positive  quasi-harmonic  measure $\mu$ is harmonic.
Let $T$ be  the   directed  positive  harmonic  current  given by  \eqref{e:m-T}. So $T$  is  positive  $\ddc$-closed current of bidimension $(1,1)$ on $X.$
Since  all the singularities    are    linearizable as well as weakly  hyperbolic,  it follows from Theorem  \ref{T:Lelong} that  $\nu(T,a)=0$ for $a\in E.$
Moreover,  as all the singularities    are    linearizable, we know by  Proposition \ref{P:Poincare} that  $\eta(x)\approx  \dist(x,E)\lof \dist(x,E)\gtrsim  \dist(x,E).$
Hence, the  result follows from Proposition \ref{P:diffusions_comparisons}. 
\endproof

\begin{problem} \rm Let $\Fc=(X,\Lc,E)$ be  a
 a   singular holomorphic foliation, where $E$ is a  subvariety of $X$ with $\codim_X(E)\geq 2.$ Let  $T$ be a  directed positive  harmonic current giving no mass to $\Par(\Fc)\cup E$.  
 Find  sufficient conditions for $T$ and $E$  so that the measure $\Phi(T)$ given by \eqref{e:mu} is  finite  (that is,  by Theorem  \ref{thm_harmonic_currents_vs_measures} (3), $\Phi(T)$ is harmonic). 
 Moreover, when the measure $\Phi(T)$   is harmonic, find sufficient conditions for $T$ and $E$  so that the  abstract heat  diffusions  coincide with the leafwise heat diffusions.
\end{problem}

\begin{problem}
\rm In this  subsection we have studied   the problem of unique  heat diffusions for the  completion of the  space $\Dc(\Fc)$ with respect to the norm $\|\cdot\|_{H^1(\mu)}.$
It seems to be of interet to study the problem  for other  spaces of test functions  with respect to other norms.  
\end{problem}

\subsection{Heat  equation on holomorphically immersed Riemann surface laminations  }\label{SS:heat_equation_SHF}

Let  $\Fc=(X,\Lc,E)$   be  a Riemann surface lamination  with  singularities  which is  holomorphically immersed in a
complex manifold $M$ (see Definition  \ref{D:immersed-lamination}.   For simplicity, fix a Hermitian form $g_M$ on $M$. Let $T$ be a positive  $\ddc$-closed current
of bidimension $(1,1) $ 
in $M,$  $T$ is  not necessarily directed by $\Fc.$ 
So,
$T\wedge g_M$ is a positive measure  
in $M.$ 
We assume that   there is a $(1,0)$-form $\tau$
defined almost everywhere with respect to $T\wedge g_M$ such that
$\partial T=\tau\wedge T .$ 
In this context, the H{\"o}rmander $L^2$-estimates are proved in
\cite{BerndtssonSibony} for the $\dbar$-equation induced on $T$.

 We also assume that
  $T\wedge g_M$ is
absolutely continuous with respect to $T\wedge g_P$. In particular, this implies implicitly  that $T$ is  supported in $X$ and $T$ does not give mass to $\Par(\Fc)\cup E.$ 
This condition 
does not depend on the choice of 
$g_M$ and allows us to define 
the operators
 $\nabla^\partial_P$,
$\nabla^{\dbar}_P$ and $\nabla_P$ on $u\in\Dc(M)$ by
$$(\nabla^\partial_P u) T\wedge g_P:=i\partial u\wedge 
\overline \tau\wedge T=i\partial(u\dbar T),\quad
(\nabla^{\dbar}_P u) T\wedge g_P:=-i\dbar u\wedge \tau\wedge T=-i\dbar(u\partial T)$$
and 
$$\nabla_P:=\nabla^\partial_P+\nabla^{\dbar}_P.$$
Define also the
operators $\Delta_P$ and $\widetilde\Delta_P$ on
$u\in\Dc(M)$ by
\begin{equation}\label{e:Laplacian}(\Delta_P u) T\wedge g_P:=i\ddbar u\wedge T 
\quad \mbox{and}\quad \widetilde\Delta_P:=\Delta_P+{1\over 2}\nabla_P. 
 \end{equation}
 Here, for the definition of $\Delta_P$ in  the  first  equation of \eqref{e:Laplacian} we have used the assumption that    
  $T\wedge g_M$  is
absolutely continuous with respect to $T\wedge g_P$.
We  will extend  the  definition of $\Delta_P$  to larger  spaces, suitable  for developing   $L^2$-techniques. To this end, 
consider the measure $\mu$ on $X$ defined by 
\begin{equation}\label{e:m} \mu:= T\wedge g_P\quad\text{on}\quad \Hyp(\Fc)\quad\text{and}\quad \mu:=0\quad\text{on}\quad \Par(\Fc)\cup E.
\end{equation}
Finally, we   assume  that the measure $\mu$ is  finite, or equivalently,  the Poincar\'e mass of $T$ is finite (see Definition \ref{D:Poincare_mass}). 
Consider  the  Hilbert space $L=L^2(T):=L^2(\mu).$ 
We also  introduce  the Hilbert space  $H=H^1(T)\subset L^2(T)$ 
associated with $T$   as the  completion of $\Dc(M)$ with  respect  
to the Dirichlet norm
$$\|u\|^2_{H^1(T)}:=\int |u|^2  T\wedge g_P +i\int \partial u
\wedge \dbar u\wedge T.$$
Since  $\Dc(M\setminus E)\subset  \Dc(M),$ we infer form  the definition of $H^1(T)$ and \eqref{e:m} and  \eqref{e:norm_H1} that if $T$ is directed by $\Fc$  then  $H^1(\mu)\subset H^1(T).$

Observe 
the operators $\nabla^\partial_P$, $\nabla^{\dbar}_P$ and $\nabla_P$ are defined on $H^1(T)$ with values in $L^1(T)$.

Define also for $u,v\in\Dc(M)$ (for simplicity, we only consider real-valued functions)
$$q(u,v):=-\int (\Delta_P u)vT\wedge g_P, \quad e(u,v):=q(u,v)+\int uv T\wedge g_P$$
and
$$\widetilde q(u,v):=-\int (\widetilde \Delta_P u)vT\wedge g_P, \quad 
\widetilde e(u,v):=\widetilde q(u,v)+\int uv T\wedge g_P.$$

We will define later the
domain of $\Delta_P$ and $\widetilde\Delta_P$ which allows us to extend
these identities to more general $u$ and $v$. 
The following lemma also holds for $u,v$ in 
the domain of $\Delta_P$ and $\widetilde\Delta_P$.

\begin{lemma} \label{lemma_e_tilde_c}
We have 
for $u,v\in\Dc(M),$
$$\widetilde q(u,v)=\Re\int  i\partial u\wedge \dbar v \wedge T 
\quad \mbox{and}\quad \int (\widetilde \Delta_P u)v T\wedge g_P = \int
u(\widetilde \Delta_P v) T\wedge g_P.$$
In particular, $\widetilde q(u,v)$ and $\widetilde e(u,v)$ are symmetric in $u,v$ and 
$$\int (\widetilde\Delta_P u) T\wedge g_P=
\int (\Delta_P u) T\wedge g_P=\int (\nabla_P u) T\wedge g_P 
=0 \quad \mbox{for}\quad u\in\Dc(M).$$
\end{lemma}
\proof
Since $T$ is $\ddc$-closed, the integral of $i\ddbar( u^2)\wedge T$
vanishes. We deduce using Stoke's formula that
\begin{eqnarray*}
\widetilde q(u,v) & = & -\int (\Delta_P u+ {1\over 2}\nabla_P u)  v T\wedge g_P\\
& = &  -\int  i\ddbar u\wedge v T
-\Re\int i \partial u\wedge \overline \tau \wedge vT\\
& = &  -\Re\int  i\ddbar u\wedge v T
-\Re\int i \partial u\wedge \overline \tau \wedge vT\\
& = &  \Re\int i\partial u\wedge \big[\dbar (v T)-v\dbar T\big]=
\Re\int i\partial u\wedge \dbar v\wedge T.
\end{eqnarray*}
This gives the first identity in the lemma.

We also have since $T$ is $\ddc$-closed
$$\int(\nabla_P u) T\wedge g_P=2\Re\int i\partial u\wedge \dbar T=2\Re\int -iu \ddbar T=0.$$
The other assertions are obtained  as in   Lemma \ref{lemma_e_tilde}.
\endproof

\begin{definition}\label{D:Poincare-regular} \rm
Let $T$ be  a positive $\ddc$-closed  bidimension $(1,1)$-current on $M$  such that  there is  a $(1,0)$-form $\tau$  defined  almost everywhere  with respect  to $T\wedge g_M$
such that  $\partial T=\tau\wedge T.$  Then $T$ is called  {\it Poincar\'e-regular} if  there is a   constant $c>0$ such that 
  $i\tau\wedge \overline{\tau}\wedge T\leq c \cdot \, T\wedge g_P.$
  \end{definition}
 The following result  gives a typical example of Poincar\'e-regularity.

\begin{proposition}\label{P:g_P-regular}
Let $T$ be a  positive  $\ddc$-closed  current directed by $\Fc.$ 
Assume  that $T$  does not give mass to $\Par(\Fc).$ Then 

\begin{enumerate}
 \item  $T\wedge g_M$ is
absolutely continuous with respect to $\mu$ defined by \eqref{e:m}.
\item   $T$ is  Poincar\'e-regular. 
 \item Let $\Gc$ be  a  singular holomorphic  foliation on $M$ such that the restriction of $\Gc$ on $X\setminus E$  induces $\Fc.$ If
all points of $E$ are linearizable  singularities of $\Gc,$ then     the Poincar\'e mass of $T$   is  finite.  
\end{enumerate}
 
\end{proposition}
 \proof
In  the  decomposition  \eqref{e:local_presentation} of $T$ in a flow box
$\U\simeq \B\times\T$, we can restrict $\nu$
in order to assume that $h_a\not=0$ for $\nu$-almost every $a$.  
Assertion (1) follows easily from the this decomposition.

We turn to assertion (2). Using again  the  decomposition  \eqref{e:local_presentation} of $T$ in a flow box
$\U\simeq \B\times\T$,  we  see that  $\tau= h_a^{-1}\partial h_a$ on the  plaque passing  through $a\in\T$
for $\nu$-almost every $a$. Then by  the proof of  \cite[Proposition 3]{FornaessSibony10}, we get
\begin{equation}\label{e:Poincare-reg} i\tau\wedge \overline{\tau}\wedge T\leq   T\wedge  g_P.
\end{equation}
Hence,  assertion  (2) follows  (with $c=1$ in Definition \ref{D:Poincare-regular}).

 For  the reader's convenience, we give here a  direct  proof of \eqref{e:Poincare-reg}.
 The  following elementary  result is needed.
 \begin{lemma}\label{L:curvature}
  Let $h$ be  a positive harmonic  function on $\D.$
  Then $$i\partial h(0)\wedge \overline{\partial h(0)}\leq h(0)^2 \cdot\, id\zeta\wedge d\bar\zeta.$$
 \end{lemma}
\proof[Proof of Lemma \ref{L:curvature}]
We may assume  without loss of generality that $h(0)=1.$
Since $h$ is  positive harmonic, we see easily that $h$ belongs to the Hardy space $ H^1(\D).$
So there is a finite positive Borel measure $\lambda$ on $\partial \D$ such that $u$ is  the  Poisson integral of $\lambda,$ that is, 
$$
h(z)=\int_{\partial \D} {1-|z|^2\over  |z-\zeta|^2}  d\lambda(\zeta)\quad \text{for}\quad z\in\D.
$$
Acting the derivative $\partial$ and then  evaluating both sides at $z=0,$   we  see easily that
$$
|\partial h(0)|\leq  \int_{\partial \D}  d\lambda(\zeta)=h(0).
$$
This  implies the desired estimate.
\endproof

Now we come back the proof of  \eqref{e:Poincare-reg}. We may assume  without loss of generality that $a\in\Hyp(\Fc).$ Fix  a  point $x$  in the plaque $V_a:=\B\times\{a\},$  we  only need  to show  that 
 $$\big({i\partial h_a\wedge \overline{\partial h_a}\over h_a^2}\big)(x)\leq  g_P(x).$$
For  every $R>0,$   covering  $\phi_x(\D_R)$  by a finite  number of  flow boxes, we can show that there is  a positive harmonic  function $h$ on $\D_R$
such that $h_a\circ \phi_x=h$ on $\D_R\cap \phi_x^{-1} (V_a)  ,$ where $\phi_x$ is  given in \eqref{e:covering_map}.
  Let $0<r<1$ be such that $r\D=\D_R$ by \eqref{e:radii_conversion}.
 Applying Lemma \ref{L:curvature} to  $h$ defined on $\D_R$ and pushing forward, via $\phi_x,$  the  result on $\phi_x(\D_R),$ we get  that
 $$
 {i\partial h_a(x)\wedge \overline{\partial h_a(x)}\over h_a^2(x)}\leq  {1\over r^2} g_P(x).
 $$
 Letting $R$ tend to infinity, we get $r\to 1,$ and hence   the last inequality  implies  the  desired  estimate.
 
 Assertion (3)  is a consequence of Proposition \ref{P:Poincare_mass}.  
  \endproof

We have the following lemma.

\begin{lemma} \label{lemma_e_e_tilde_c}
Assume that $T$ is  of finite Poincar\'e mass.
Then the  bilinear forms $\widetilde q$ and $\widetilde e$ extend
  continuously to $H^1(T)\times H^1(T)$. 
If moreover $T$ is  Poincar\'e-regular, then the same property holds for $q$ and $e$.
Moreover, we have
$q(u,u)=\widetilde q(u,u)$ and $e(u,u)=\widetilde e (u,u)$ for $u\in H^1(T)$. 
\end{lemma}
\proof
The first assertion is deduced from Lemma
\ref{lemma_e_tilde_c}. Assume that   $T$ is Poincar\'e-regular.
Using Cauchy-Schwarz's inequality, we have for $u,v\in\Dc(M)$
\begin{eqnarray*}
|q(u,v)-\widetilde q(u,v)|^2 & \leq & \Big|\int \partial u \wedge
v\overline\tau \wedge T \Big|^2 \\
& \leq &  \Big(\int i\partial u\wedge \dbar u
\wedge T \Big) \Big(\int iv^2\tau\wedge\overline\tau \wedge T\Big)\\  
& \leq & c \Big(\int i\partial u\wedge \dbar u
\wedge T \Big) \Big(\int v^2 g_P \wedge T\Big)\\
&\leq  & c\|u\|^2_{H^1(T)}\|v\|^2_{L^2(T)},
\end{eqnarray*}
where, in the third line we use the Poincar\'e regularity of $T,$ and in the last line we use the finiteness of the Poincar\'e mass  $\int g_P\wedge T<\infty.$ 
This implies the second assertion. We also have for $u\in\Dc(M)$
\begin{eqnarray*}
q(u,u)-\widetilde q(u,u) & = & \Re\int(\nabla_P^\partial
u)uT\wedge g_P=\Re\int i\partial u\wedge u \overline\tau\wedge T \\
& = & {1\over 2}\Re\int i\partial u^2 \wedge \dbar T=-{1\over 2}\Re\int
iu^2 \ddbar T= 0. 
\end{eqnarray*}
The  identity  $e(u,u)=\widetilde e (u,u)$ follows readily from the last  equality.
\endproof
\begin{remark}\label{R:Dom}\rm
Define the domain 
$\Dom(\pm\widetilde\Delta_P)$ of $\pm\widetilde\Delta_P$
(resp. $\Dom(\pm\Delta_P)$ of $\pm\Delta_P$ when $T$ is Poincar\'e-regular) as the space of $u\in
H^1(T)$ such that $\widetilde q(u,\cdot)$ (resp. $q(u,\cdot)$) extends to a linear
continuous form on $L^2(T)$. When $T$ is Poincar\'e-regular, we
have seen in the proof of Lemma \ref{lemma_e_e_tilde_c} that
$q(u,v)-\widetilde q(u,v)$ is continuous on $H^1(T)\times
L^2(T)$. Therefore, we deduce  from  this discussion  and  Definition \ref{D:Dom-Delta}
and Remark  \ref{R:Dom-Delta} that $$\Dom(\pm\Delta_P)=\Dom(\pm\widetilde\Delta_P)=H_P(T),$$
where  where  $H_P(T)$ is the completion of $\Dc(M)$ for the norm 
\begin{equation*}
\|u\|_{H_P(T)}:=\sqrt{\|u\|^2_{L^2(T)}+\|\Delta_P u\|_{L^2(T)}^2}. 
\end{equation*}
If moreover $T$ is directed by $\Fc,$  then we deduce from  the inclusion   $\Dc(M\setminus E)\subset  \Dc(M)$  and \eqref{e:norm_P} that  $H_P(\mu)\subset H_P(T).$
\end{remark}
The following result can be proved  in the same way as    Proposition   \ref{prop_delta_max}.

\begin{proposition} \label{prop_delta_max_c}
Let $T$ be a positive $\ddc$-closed current of bidimension 
$(1,1)$   in  a complex manifold $M$ such that   $T\wedge g_M$ is
absolutely continuous with respect to $T\wedge g_P$  and  the Poincar\'e mass of $T$ is  finite. Then the associated operator
$-\widetilde\Delta_P$ (resp. $-\Delta_P$ when $T$ is
Poincar\'e-regular) is maximal monotone on $L^2(T).$ 
In particular,
it is the infinitesimal generator of semi-groups of contractions on $L^2(T)$ and its graph is closed.
\end{proposition}

The  following result is   an ergodic theorem associated to the abstract heat
diffusions.
\begin{theorem}\label{T:diffusions} {\rm (Dinh-Nguyen-Sibony \cite{DinhNguyenSibony12}). }
 We  keep the hypothesis of Proposition \ref{prop_delta_max_c}. 
 Let $S(t)$,
$t\in\R^+$, denote  the semi-group of contractions associated with the operator
  $-\widetilde\Delta_P$ (or $-\Delta_P$  if $T$ is Poincar\'e-regular) which is given by Theorem \ref{th_hille_yosida}.
  Then 
\begin{enumerate}
 \item 
the  
measure $\mu$ is  $S(t)$-invariant (that is,  $\langle S(t)u,\mu\rangle=\langle u,\mu\rangle $ for every $u\in L^p(T)$), and 
$S(t)$ is a positive contraction in $L^p(T)$ for all $1\leq p\leq\infty$ (that is, $\|S(t)u\|_{L^p(T)}\leq  \|u\|_{L^p(T)}$ for every $u\in L^p(T)$); 
\item
for all $u_0\in  L^p(T),$  $1\leq p<\infty$, the average 
$$\frac{1}{R}\int_0^R S(t)u_0 dt$$ 
converges pointwise  $\mu$-almost everywhere
 and also in $L^p(T)$ to 
an $S(t)$-invariant function  $u_0^*$ when $R$ goes to infinity.
Moreover, $u_0^*$ is constant on the leaf $L_a$ for $\mu$-almost every $a$.
If $T$ is an extremal directed current, then $u$ is constant $\mu$-almost everywhere.
\end{enumerate}
\end{theorem}

The following result gives us the mixing for the operator $-\widetilde\Delta_P$. 

\begin{theorem} 
Under the hypothesis of Theorem  \ref{T:diffusions}, if $S(t)$ is associated 
with $-\widetilde\Delta_P$ and if $T$ is extremal,  
$S(t)u_0\to \langle \mu,u_0\rangle$ in $L^p(T)$ when $t\to \infty$
for every $u_0\in L^p(T)$ with $1\leq p<\infty$.
\end{theorem}

The  following result  is  similar to  Proposition \ref{prop_dense_laplace}.
\begin{proposition} \label{prop_dense_laplace_c}
Let  $T$ be  an 
extremal directed positive $\ddc$-closed  current of finite Poincar{\'e} mass.  Then, the closures of $\Delta_P(\Dc(M))$ and of $\widetilde\Delta_P(\Dc(M))$ in $L^p(T)$, 
$1\leq p\leq 2$, are   the 
hyperplane of $L^p(T)$ defined by the equation $\int v T\wedge g_P=0$.
\end{proposition}

 

\subsection{Geometric ergodic theorems}\label{SS:Geometric-ergodic-thms}

In this subsection, we will give an analogue  of  
Birkhoff's ergodic theorem in the  context of a   Riemann surface  lamination $\Fc=(X,\Lc,E)$
with  singularities. 
 Our  ergodic theorem is of  geometric nature  and   it is close to Birkhoff's averaging on orbits of a
dynamical system.   Here  the averaging is  on  hyperbolic  leaves and the  time is the  hyperbolic time.

 Recall  from the Main Notation that
for $0<r<1,$ $r\D$ denotes the disc of center $0$ and of radius $r.$ In the
Poincar{\'e} disc $(\D,g_P),$ 
$r\D$ is also the disc of center $0$ and of radius 
\begin{equation}\label{e:radii_conversion}
R:=\log{1+r\over 1-r}\cdot
\end{equation}
which is  also denoted by $\D_R.$  So $r\D=\D_R.$    

Let $\Fc=(X,\Lc,E)$ be a  Riemann surface  lamination with singularities.    
Let $T$ be a  positive harmonic 
 current directed by $\Fc.$ 
Assume that  $T$ has no mass on $\Par(\Fc).$  Consider the  positive  measure $\mu=\Phi(T)$ on $X$  defined by \eqref{e:mu}.  
 Assume in addition  that  
$\mu$ is a  probability measure.

\begin{figure}[h]%
\begin{center}
\def\svgwidth{0.8\columnwidth}
\resizebox{0.7\textwidth}{!}{\input{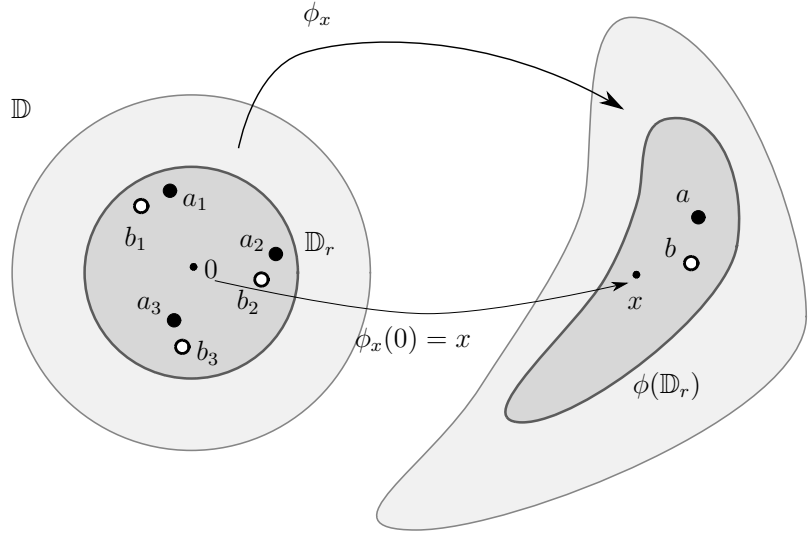}}%
\caption{Consider the restriction of the universal covering map $\phi_x:\  \D_R\to  \phi_x(\D_R)\subset  L_x.$   The set of preimages of a point  $a$  (resp. $b$) $\in \phi_x(\D_R)$  
is  $\{a_1,a_2,a_3\}$  (resp.  $\{b_1,b_2,b_3\}$).
}
\label{fig:Poincare-disc}
\end{center}
\end{figure}

For any point $x\in \Hyp(\Fc),$ 
let $\phi_x:\D\rightarrow L_x$  be given by  \eqref{e:covering_map}.  
  Following  Forn{\ae}ss-Sibony  \cite{FornaessSibony05,FornaessSibony08}, consider, for all $0<R<\infty$, the measure $m_{x,R}$ and  the current $\tau_{x,R}$ given by 
\begin{equation}\label{e:m_x,R}\begin{split}  m_{x,R}&:=\frac{1}{M_R}(\phi_x)_* \big(\log^+ \frac{r}{|\zeta|}g_P\big),\\
 \tau_{x,R}&:=\frac{1}{M_R}(\phi_x)_* \big(\log^+ \frac{r}{|\zeta|}\big).
 \end{split}
 \end{equation}
Here, $\log^+:=\max\{\log,0\}$, $g_P$ denotes as  usual  the Poincar{\'e} metric on $\D$ and 
\begin{equation*}
M_R:= \int \log^+ \frac{r}{|\zeta|} g_P=\int \log^+ \frac{r}{|\zeta|} \frac{2}{(1-|\zeta|^2)^2}
id\zeta\wedge d\overline\zeta.
\end{equation*}
So, $m_{x,R}$ (resp. $\tau_{x,R}$) is a probability measure  (resp.  a directed  positive  current of bidimension $(1,1)$) which depends on $x,R;$ but 
does not depend on the choice of $\phi_x$. The current 
$\tau_{x,R}$ is  called the {\it Nevanlinna current} of index $r.$

\begin{theorem} \label{T:geometric_ergodic}{\rm (Dinh-Nguyen-Sibony \cite{DinhNguyenSibony12}). }
 We  keep the above hypothesis and notation.
Assume in addition that the current  $T$ is extremal.    
Then for $\mu$-almost every point $x\in X,$ the measure $m_{x,R}$ defined above converges to $\mu$
when $R\to\infty$.  
\end{theorem}

 To  prove  the theorem,  our main ingredient  is a  delicate estimate on the heat kernel of the Poincar\'e disc
 (inequality \eqref{e:heat_kernel_main_estimate} below, see also   \cite[p. 370, line 8-]{DinhNguyenSibony12}).
 This estimate allows us to deduce  the the  desired result  from 
  the ergodic theorem associated to the  abstract heat
diffusions (Theorem \ref{th_heat_real}). The remainder of the subsection is  devoted  to an outline of our proof.

For  every  $0<R<\infty$, we introduce  the operator $B_R$ on $L^1(\mu)$ by 
$$B_Ru(a):=\frac{1}{M_R}\int_{|\zeta|<1} \log^+
\frac{r}{|\zeta|}(\phi_a)^* (ug_P)=\langle
m_{a,R},u\rangle.$$
Note that for $u\in L^1(\mu)$, the function $B_Ru$ is defined
$\mu$-almost everywhere. So, the convergence in 
Theorem \ref{T:geometric_ergodic} is equivalent to the convergence
$B_Ru(a)\rightarrow \langle \mu,u\rangle$ for $u$ continuous and for $\mu$-almost every $a$.

\begin{proposition}\label{prop_B_L1} 
Under the hypothesis of Theorem \ref{T:geometric_ergodic}, 
for  every $u\in L^1(\mu)$, we have 
$$\int (B_R u) d\mu =\int u d\mu.$$
In particular, $B_R$ is positive and of norm $1$ in $L^p(\mu)$ for all
$1\leq p\leq \infty$.  
\end{proposition}

 \proof
Fix an $R>0$.
The positivity of $B_R$ is clear. Since $B_R$ preserves constant functions, 
its norm in $L^p(\mu)$ is at least equal to 1.
It is also clear that $B_R$ is an operator of norm 1 on $L^\infty(\mu)$. 
If $B_R$ is of norm 1 on $L^1(\mu)$, by interpolation \cite{Triebel}, its norm on $L^p(\mu)$ is also equal to 1. 
So, the second assertion is a consequence of the first one.
 Observe  that  as  in the  proof of  Theorem  \ref{thm_harmonic_currents_vs_measures} (3),  the operator
$B_R$ can be obtained as an average of the operators   $A_t$ on $t\leq R,$ where $A_t$ is  given by \eqref{e:A_R}.
So the  first  assertion follows from   Lemma \ref{L:mu_Dt_invariant}.
\endproof

 \begin{remark}
  \rm In the  above proposition which    relies on  Lemma \ref{L:mu_Dt_invariant}, we make  an essential  use of the assumption that the current $T$ is directed. 
 \end{remark}

We have the following ergodic theorem.

\begin{theorem} \label{th_ergodic_Lp}  
Under the hypothesis of Theorem \ref{T:geometric_ergodic}, 
if $u$ is a function in $L^p(\mu)$, with $1\leq p <\infty$, then
$B_Ru$ converge in $L^p(\mu)$ towards a constant function
$u^*$  when $R\to \infty$. 
\end{theorem}
\begin{proof}
We show that it is enough to consider the case where $p=1$.
By Proposition \ref{prop_B_L1}, it is enough to consider $u$ in a
dense subset of $L^p(\mu)$, e.g. $L^\infty(\mu)$. 
For $u$ bounded, we have $\|B_Ru\|_\infty\leq \|u\|_\infty$. Therefore, if $B_Ru\to u^*$ in $L^1(\mu)$ we have $B_Ru\to u^*$ in $L^p(\mu)$, $1\leq p<\infty$. 

Now, assume that $p=1$. Since $B_R$ preserves constant functions,
by Proposition \ref{prop_dense_laplace} applied to $p=1$, it is enough to consider
$u=\Delta_P v$ with $v\in\Dc(\Fc)$. We have to show that $B_Ru$
converges to 0. 
Note that since $v$
is in $\Dc(\Fc)$, the function $\Delta_P v$ is defined at every point
on $\Hyp(\Fc) $  by the formula $(\Delta_P v)
g_P:=i\ddbar v$ on the leaves (see formulas \eqref{e:Laplacian_disc} and  \eqref{e:Delta_commutation}). 
 
We deduce from \eqref{e:m_x,R} that  $m_{a,R}=\tau_{a,R}\wedge g_P$ and
$$B_Ru(a)= B_R(\Delta_Pv)(a) = \langle  \tau_{a,R},(\Delta_P v) g_P \rangle
= \langle  \tau_{a,R}, i\ddbar v \rangle=\langle i\ddbar \tau_{a,R},v\rangle.$$
Now  we show that the mass of
$i\ddbar \tau_{a,R}$ tends to 0 uniformly on $a$.
Indeed, by Jensen's formula,  we have that
$$M_R\cdot\, i\ddbar \tau_{a,R}=i\ddbar (\phi_a)_* \big(\log^+ \frac{r}{|\zeta|}\big)=(\phi_a)_*(\nu_r) -\delta_a,$$
where $\nu_r$ denotes the Lebesgue measure on the circle $\partial (r\D)$ which is  the circle  with center $0$ and  radius $r,$ and $\delta_a$ is the Dirac mass at $a.$
Since  a direct computation shows that  $M_R-2\pi R=O(1),$  we  get $M_R\rightarrow\infty.$ Hence, it is easy to see 
that the  mass of
$i\ddbar \tau_{a,R}$ 
tends to 0 uniformly on $a$. The result follows.\end{proof}

\begin{lemma} \label{lemma_heat_delta}
The leafwise heat diffusions $D_t$  (see \eqref{e:diffusions})  extends continuously to an operator of norm $1$ on  $L^p(\mu)$ for 
$1\leq p\leq\infty$. Moreover, there is a constant $c>0$ such that for 
all $\epsilon>0$ and $u\in L^1(\mu)$, 
we have 
$$\mu\big\{\widetilde Du>\epsilon\big\}\leq c\epsilon^{-1}\|u\|_{L^1(\mu)},$$
where the operator $\widetilde D$ is defined by 
$$\widetilde Du(a):=\limsup_{R\to\infty}\Big|\frac{1}{R}\int_0^R D_tu(a)dt\Big|.$$
\end{lemma}
\proof
Recall from \eqref{e:diffusions}  and \eqref{e:semi_group} that $D_t$ is positive and preserves constant
functions. Its norm on $L^\infty(\mu)$ is equal to 1.
On the  other hand,  by  Theorem  \ref{thm_harmonic_currents_vs_measures} (3), the norm of  $D_t$ on $L^1(\mu)$ is equal to 1.
By interpolation \cite{Triebel}, its norm on $L^p(\mu)$ is also equal to 1. The first assertion follows.

 The second one is a direct consequence of  Lemma VIII.7.11 in Dunford-Schwartz
\cite{DunfordSchwartz}. This lemma says that if $D_t$ is a semi-group
acting on $L^1(\mu)$ for some probability measure $\mu$ such that
$\|D_t\|_{L^1(\mu)}\leq 1$,  $\|D_t\|_{L^\infty(\mu)}\leq 1$ and $t\mapsto D_tu$
is measurable with respect to the Lebesgue measure on $t$, then 
$$\mu\big\{\widetilde Du>\epsilon\big\}\leq
c\epsilon^{-1}\|u\|_{L^1(\mu)},$$
where $\widetilde D$ is defined as above.
\endproof

Consider also  the  operator $\widetilde B$ given by
$$\widetilde B u(a):=\limsup_{R\to\infty}| B_Ru(a)|.$$
We have the following lemma.

\begin{lemma}\label{lemma_limsup}
There is a constant $c>0$ such that for all $\epsilon>0$ and $u\in L^1(\mu)$ we have 
$$\mu\big\{\widetilde Bu>\epsilon \big\} \leq   c\epsilon^{-1} \|u\|_{L^1(\mu)}.$$ 
\end{lemma}
\proof
Since we can write $u=u^+-u^-$ with $\|u\|_{L^1(\mu)}=\|u^+\|_{L^1(\mu)}+\|u^-\|_{L^1(\mu)}$, it is enough to consider $u$ positive with $\|u\|_{L^1(\mu)}\leq 1$. 
Write $u=\sum_{i\geq 0} u_i$ with $u_i$ positive and bounded such that $\|u_i\|_{L^1(\mu)}\leq 4^{-i}$. 
We will show that  $\widetilde Du_i=\widetilde Bu_i$. This, together with 
Lemma \ref{lemma_heat_delta} applied to $u_i$ and to $2^{-i-1}\epsilon$ give the result.

So, in what follows, assume that $0\leq u\leq 1$. We show that $\widetilde Du=\widetilde Bu$. 
This assertion will be  an immediate consequence  of  the  following  estimate
\begin{equation}\label{e:heat_kernel_main_estimate}\Big|B_Ru(a)-\frac{2\pi}{M_R}\int_0^{{M_R\over 2\pi}} D_tu(a)dt\Big|\leq  c R^{-1/2}\sqrt{\log R}
\end{equation}
for $\mu$-almost every $a,$ where  $c$ is a constant independent of $u$ and $a,$ $R.$
Observe that the integrals in the left hand side of  \eqref{e:heat_kernel_main_estimate} can be computed on $\D$ in terms of
$\widehat u:=u\circ \phi_a$ and the Poincar{\'e} metric $g_P$ on $\D$. So, in order to simplify the notation, we will work on $\D$. We have to show that
$$\Big|B_R\widehat u(0)-\frac{2\pi}{M_R}\int_0^{M_R\over 2\pi} D_t\widehat u(0)dt\Big|\leq  cR^{-1/2}\sqrt{\log R}$$
where
$$B_R\widehat u(0):=\frac{1}{M_R}\int_{\D_R} \log{\frac{r}{|\zeta|}}\widehat u g_P
\quad\mbox{and}\quad D_t\widehat u(0):= \int_\D p_\D(0,\cdot,t)\widehat u g_P.$$
   In fact, the delicate  inequality  \eqref{e:heat_kernel_main_estimate}
 follows  from  hard  estimates based on  the following identities 
 for the heat kernel $p_\D(x,y,t)$   on the Poincar\'e disc $(\D,g_P)$  (see \eqref{e:heat_kernel}):  this is a positive function on $\D^2\times \R^+_*,$ 
smooth when $(x,y)$ is outside the diagonal of $\D^2,$ and it satisfies
\begin{equation}\label{e:heat-kernel-disc}\int_\D p_\D(x,y,t) g_P (y)=1 \quad \mbox{and} \quad 
{1\over 2\pi}\log{\frac{1}{|y|}}  =\int_0^\infty p_\D(0,y,t)dt.
\end{equation}
Moreover, the function $p_\D(0,\cdot,t)$ is radial, see e.g. Chavel
\cite[p.246]{Chavel}. 
 \endproof

\noindent{\it Proof of Theorem \ref{T:geometric_ergodic}.}  Let $u$ be a function in $L^1(\mu)$. It is enough to 
show that $B_Ru(a)\to \langle \mu,u\rangle$ for $\mu$-almost every $a$.
Since this is true when $u$ is constant, we can assume   without loss of generality that  $\langle \mu,u\rangle=0$.
Fix a constant $\epsilon>0$ and define 
$E_\epsilon(u):=\big\{\widetilde Bu\geq \epsilon \big\}$.
To prove  the theorem  it suffices  to show  that  $\mu(E_\epsilon(u))=0$.

By Proposition  \ref{prop_dense_laplace},  $\Delta_P(\Dc(\Fc))$ is  dense
in the hyperplane of functions with mean 0 in $L^1(\mu)$.   
Consequently,  for  every $\delta>0$  we  can choose
a smooth function $v$ such that $\|\Delta_P v-u\|_{L^1(\mu)}<\delta$.   
We  have  
$$E_{\epsilon}(u)\subset E_{\epsilon/2}(u-\Delta_Pv)\cup E_{\epsilon/2}(\Delta_Pv).$$
Therefore,
$$\mu\big(E_{\epsilon}(u)\big)\leq \mu\big(E_{\epsilon/2}(u-\Delta_Pv)\big) +
\mu\big(E_{\epsilon/2}(\Delta_Pv)\big).$$
We have
$$B_R(\Delta_Pv)(a)= \langle \tau_{a,R}, i \ddbar v\rangle = \langle i\ddbar \tau_{a,R},  v\rangle.$$
The last integral tends to 0 uniformly on $a$ since the mass of $i\ddbar \tau_{a,R}$ satisfies this property.
Hence, $\mu\big(E_{\epsilon/2}(\Delta_Pv)\big)=0$ for $R$ large enough.

On the other hand, by Lemma  \ref{lemma_limsup},  we  have 
\begin{eqnarray*}
\mu\big(E_{\epsilon/2}(u-\Delta_Pv)\big) &=&
\mu\big(\widetilde B(u-\Delta_P v)>\epsilon/2\big)\\
&\leq& 2c\epsilon^{-1}\|u-\Delta_Pv\|_{L^1(\mu)} \leq 2c\epsilon^{-1}\delta. 
\end{eqnarray*}
Since $\delta$  is  arbitrary,  we  deduce  from the last  estimate  that
$\mu\big(E_{\epsilon}(u)\big)=0$.
This completes the proof of the theorem.
\hfill  $\square$
\section{Unique ergodicity theorems}
\label{S:Unique_ergodicity}

\subsection{Case of compact nonsingular  laminations}
\label{SS:Unique_ergodicity_conformal}

A general principle for the unique ergodicity of a lamination $\Fc$ is that  there exists only one  directed positive harmonic  current  (up to a multiplicative constant).   

Garnett establishes the unique ergodicity of the weak stable foliation of the geodesic flow of a
compact surface of constant curvature $-1,$ see \cite[Proposition 5]{Garnett}).

Consider a  suspension  $\Fc_h$ constructed in Example
\ref{Ex:Suspensions}. C. Bonatti and X. Gomez-Mont \cite{BonattiGomez-Mont} prove
 that either the group $h(\pi_1 (S))$ has an invariant probability measure or the
foliation $\Fc_h$ is uniquely ergodic and this is the generic case. They construct a proba-
bility measure on $M_h$ by an appropriate averaging process on the leaves and unique
ergodicity means that the averaging process applied to an arbitrary leaf gives always
the same limit.  

In \cite{DeroinKleptsyn} Deroin and  Kleptsyn   investigate the case of minimal sets of  a singular holomorphic  foliation $\Fc=(X,\Lc,E)$.
Recall  that a set $M\subset  X\setminus E $ is  said to be {\it minimal} if it is leafwise saturated closed subset of $X$ which contains no proper subset  with this property.
Recall   also  the  the  following {\bf  Minimal Set Problem}.
\begin{problem}
 \rm Does  there  exist a $\Fc\in\Fc_d(\P^2)$ with $d>1$ which has  a nontrivial minimal set, i.e. a minimal set which is not an algebraic curve?
\end{problem}
This question seems to be asked by C. Camacho in the  mid 1980's. It still remains open, see \cite[Subsection 5.6]{FornaessSibony08} for a recent discussion. 

 Deroin and  Kleptsyn in \cite{DeroinKleptsyn}  prove  the following  result.
\begin{theorem} Let $\Fc=(X,\Lc,E)$ be a singular holomorphic   foliation  in a compact complex surface $X,$ and $M$ be a minimal set  whose  leaves are all hyperbolic.
Assume that there is  no  positive  closed current directed by  $\Fc$ which is   supported on 
$ M $. Then there exists a unique directed positive harmonic current on $M$ (up to a multiplicative constant).
\end{theorem}

The first main step of  their proof  is the
existence of at least one directed positive harmonic current whose associated Lyapunov exponent is strictly  negative. The second step  exploits  this and  the similarities between
Brownian motions on different leaves  in order to infer the unique ergodicity.

It is  worthy  noting here  that all  the laminations  considered in this subsection are  compact nonsingular hyperbolic.
We wish to address  the  unique  ergodicity for holomorphic  foliations. However, a general holomorphic foliation  is often singular.
In the presence of singularities, the machinary developed by the previous  authors  do not work. Let us look at  the  simple case where the ambient compact complex manifold
is simply $\P^2.$
\subsection{Case of  $\P^2$}
\label{SS:Unique_ergodicity_P2}

In  \cite{FornaessSibony05}   Forn{\ae}ss  and Sibony develop an energy theory for  positive $\ddc$-closed   currents of   bidegree $(1,1)$  on every compact K\"ahler manifold $(X,\omega)$
of arbitrary dimension $k\geq 2.$
This  allows them to define  $\int_X T\wedge T\wedge \omega^{k-2}$ for  every positive $\ddc$-closed current $T$ of bidegree $(1,1)$ on $X.$
This theory applies to directed  positive $\ddc$-closed  currents on singular holomorphic  foliations on compact K\"ahler surfaces. 
A short  digression  will be presented in Section \ref{S:Densities_Unique_Ergodicity}.

In  \cite{FornaessSibony05, FornaessSibony10}   Forn{\ae}ss  and Sibony also  develop a geometric intersection theory 
for directed  positive harmonic  currents on singular holomorphic  foliations on $\P^2.$

Combining these two theories, they obtain  the following remarkable unique   ergodicity result for   singular holomorphic foliations
without  invariant algebraic  curves.

\begin{theorem} \label{T:FornaessSibony}{\rm (Forn{\ae}ss-Sibony \cite{FornaessSibony10}). } Let $\Fc$ be a singular  holomorphic foliation  in $\P^2$ 
whose singularities are all hyperbolic. Assume that $\Fc$ has no invariant
algebraic curve. Then $\Fc$  has a unique directed positive  $\ddc$-closed  
current of mass $1.$ Moreover, this unique current $T$ is not closed.
In particular,  for every point $x$ outside the singularity set of $\Fc,$
the current $\tau_{x,R}$ defined in \eqref{e:m_x,R} converges to $T$
when $R\to\infty$.  
\end{theorem}
 By \cite{Brunella06} 
 (see also  Theorem \ref{T:generic} (3)),
   if  all the singularities of $\Fc\in\Fc_d(\P^2)$ are hyperbolic
and $\Fc$ does not possess any invariant algebraic curve, then $\Fc$ admits no
directed positive closed current. So  the conclusion of Theorem \ref{T:FornaessSibony} is a typical property of the family $\Fc_d(\P^2).$
The proof in  \cite{FornaessSibony10} is  based on two   ingredients.  The first one is the energy theory for positive $\ddc$-closed currents
which we mentioned  previously.
The second one is  a geometric  intersection  calculus for  these  currents. For the  second ingredient,
 the  transitivity of the automorphism group of $\P^2$ is heavily  used. Indeed, they define  the geometric intersection  of two directed positive $\ddc$-closed currents  $T,$ $S$  in $\P^2$ as
 the  positive measure
 $$
 T\wedge S:=\lim_{\epsilon\to 0}  (T\wedge \Phi^*_\epsilon S)\, ,
 $$
 where $\Phi_\epsilon$ is a continuous  family of automorphisms $\Aut(\P^2)$  with $\Phi_0$  the identity. 
 Moreover, the proof is quite technical. The computations needed to estimate the geometric intersections are quite involved.
Using these techniques, P\'erez-Garrand\'es \cite{Perez-Garrandes}  has studied the case where $X$ is a homogeneous compact K\"ahler surface.

The  case  where $\Fc$  possesses  invariant algebraic  curves has recently  been solved  as follows.

\begin{theorem} \label{T:Dinh-Sibony} {\rm (Dinh-Sibony \cite{DinhSibony18}). }
Let $\Fc$ be a singular  holomorphic foliation  in $\P^2$ 
whose singularities are all hyperbolic. Assume that
 $\Fc$
admits a finite number of invariant algebraic curves. Then   any directed   positive $\ddc$-closed  current is a linear combination of the currents of integration on
these curves.  In particular, all directed   positive $\ddc$-closed  currents
are closed.
\end{theorem}

 Theorem \ref{T:Dinh-Sibony} is  suprising even in the special  case where $\Fc$  admits the line at infinity $L_\infty$  as an invariant curve.
 Let $\Fc$ be a generic  foliation of a given degree $d>1$ with this property.
 By Khudai-Veronov 
\cite{IlyashenkoYakovenko}, all  leaves  (except $L_\infty$) of $\Fc$ are dense. So by intuition from  Theorem \ref{T:FornaessSibony} one could  expect   that there  should be 
a directed $\ddc$-closed  current with the full support $\P^2.$ However, Theorem \ref{T:Dinh-Sibony}  says that this  intuition is   false.

To prove Theorem   \ref{T:Dinh-Sibony} we need to show that if $T$ is a positive $\ddc$-closed
current directed by $\Fc$ having no mass on any leaf, then $T$ is zero.
For this purpose,   Dinh and Sibony \cite{DinhSibony18} develop a theory of densities of positive $\ddc$-closed $(1,1)$-currents in a compact K\"ahler surface. The  pioneering
theory  that laid down the foundation was previously introduced by these authors in \cite{DinhSibony18b}  in the context of positive closed currents  defined on  compact K\"ahler manifolds. 
Applications of these  theories in   complex dynamics of higher dimension could be found in \cite{DinhNguyenTruong,DinhSibony16, DinhSibony17b} etc.

Theorems \ref{T:FornaessSibony} and \ref{T:Dinh-Sibony} gives the complete  dichotomy  of the unique ergodicity for singular holomorphic  foliations  in  $\P^2$
with hyperbolic  singularities.
 
  \begin{problem}\label{Pr:Pk}\rm
    Are there  any versions of     Theorem \ref{T:FornaessSibony} and    Theorem   \ref{T:Dinh-Sibony} in $\P^k$  with $k>2$  when  we assume that the singularities are all hyperbolic linearizable ? 
  \end{problem}

  \begin{problem}\label{Pr:Pk-bis}\rm
    Find  analogous  versions of     Theorem \ref{T:FornaessSibony} and    Theorem   \ref{T:Dinh-Sibony}  when  the singularities are only linearizable in the  case of $\P^2,$
 and then the general case $\P^k$  with $k>2.$   
  \end{problem}

\subsection{Case of compact K\"ahler surfaces}
\label{SS:Unique_ergodicity_Kaehler_surfaces}

Our  recent  work in collaboration with  Dinh and Sibony \cite{DinhNguyenSibony18} gives a complete answer to the unique  ergodicity for singular holomorphic
foliations on general  compact K\"ahler surfaces. 
Our results also hold for bi-Lipschitz laminations (by Riemann surfaces) (without singularities) in $X$. 
Recall that such a  lamination is a compact subset of $X$ which is locally a union of disjoint graphs of holomorphic functions depending in a bi-Lipschitz way on parameters,
see Subsection \ref{SS:positive_ddc_closed_currents}   for a precise local description. 

Let $H^{1,1}(X)$ denote the Dolbeault cohomology
group of real smooth $(1,1)$-forms on $X.$ For  a real smooth closed $(1,1)$-form $\alpha$ on $X,$ let $\{\alpha\}$ be  its  class
in $H^{1,1}(X).$
The cup-product $\smile$ on $H^{1,1} (X) \times H^{1,1} (X)$ is defined by
$$
(\{\alpha\},\{\beta\}) \mapsto \{\alpha\} \smile \{\beta \} :=\int_X\alpha\wedge\beta,
$$
where $\alpha$ and $\beta$ are real smooth closed forms. The last integral depends only on
the classes of $\alpha$  and $\beta.$ The bilinear form $\smile$ is non-degenerate and induces
a canonical  isomorphism  between  $H^{1,1} (X)$ and its  dual  $H^{1,1} (X)^*$  (Poincar\'e duality). In
the definition of $\smile$ one can take $\beta$ smooth and $\alpha$ a current in the sense of
de Rham. So, $H^{1,1} (X)$ can be defined as the quotient of the space of real
closed $(1, 1)$-currents by the subspace of $d$-exact currents. Recall that an $(1, 1)$-current $\alpha$ is real  (resp. $\ddc$-closed)  if $\alpha = \bar \alpha$ (resp. $\ddc\alpha=0$).
Assume that $\alpha$ is a real
$\ddc$-closed $(1, 1)$-current (this  is  the case when, for example,  $\alpha$ is  a directed positive harmonic  current, see 
Theorem  \ref{thm_harmonic_currents_vs_measures} (1)). Then 
by the $\ddc$-lemma, the integral $\int_X \alpha \wedge \beta$ is also independent of the
choice of $\beta$ smooth and closed in a fixed cohomology class. So, using the above isomorphism,
one can associate to  such $\alpha$ a class $\{\alpha\}$ in $H^{1,1} (X).$

 We need to recall some  terminology in K\"ahler geometry. A K\"ahler form on $X$ is a  strictly positive closed smooth  $(1,1)$-form. 
The  {\it  K\"ahler cone} of $X$ is  the  set of  the cohomology classes of  K\"ahler  forms on $X.$ This is a cone in $H^{1,1}(X).$  
 We say  that  a cohomology class in $ H^{1,1}(X)$ is  {\it nef} if  it belongs to the closure of the K\"ahler cone of $X.$
  We say  that  a cohomology class in $ H^{1,1}(X)$ is  {\it big} if it can be represented by a strictly positive closed (not necessarily smooth) $(1,1)$-current.

  Now  we are in the position to state  the first  result of this subsection.
  
\begin{theorem} \label{T:Unique_ergodicity_1}  {\rm (\cite[Theorem 1.1]{DinhNguyenSibony18})}
 Let $\Fc=(X,\Lc,E)$  be a singular holomorphic foliation  with   only hyperbolic singularities or
 a bi-Lipschitz lamination in a compact K\"ahler surface $(X,\omega)$. 
Assume that $\Fc$ admits no directed positive closed current.
Then there exists a  unique  positive  $\ddc$-closed  current $T$ of mass  $1$  directed  by $\Fc.$  
In particular, for an arbitrary point $x\in X\setminus E,$
 $\tau_{x,R} \to T$ in the sense of currents, as $R \to \infty$,  where $\tau_{x,R}$ is   the  Nevanlinna current  defined by \eqref{e:m_x,R}. Moreover, the cohomology class $\{T\}$ of $T$ is nef and big.
 \end{theorem}

The  new  idea in the proof of  Theorem  \ref{T:Unique_ergodicity_1} is  to introduce 
a more flexible  tool 
which is   a  density theory for tensor products of  positive $\ddc$-closed currents. It is  worthy noting that
such tensor  products  are  in general  not $\ddc$-closed. So  these currents go beyond the scope of previous  theories of densitis \cite{DinhSibony18b,DinhSibony18}.
The method allows us to bypass   the assumption of   homogeneity of $X,$  which was frequently used in  \cite{FornaessSibony10}.    
The proof is  more conceptual and also  far less technical. The strategy is as follows. Given a positive $\ddc$-closed current $T$ on a surface $X$, we consider 
the positive current $T\otimes T$ near the diagonal  $\Delta$ of $X\times X$ which, in general, is not $\ddc$-closed.
We study the tangent currents to $T\otimes T$ along the diagonal $\Delta$. 
As one can expect this is related to the self-intersection properties of the current $T$. It turns out that the geometry of the tangent currents is quite simple. 
They are positive closed currents and are the pull-back of positive measures $\vartheta$ on $\Delta$ to the normal bundle of $\Delta$ in $X\times X$.
We relate the mass of $\vartheta$ to a cohomology class of the current $T$ and its energy.

The foliation or lamination enters in the picture to prove that $\vartheta$ is zero when $T$ is directed by a foliation or lamination as above. This is done using the local properties of the foliation or lamination,
the local description of $T$ and in particular, that the singularities are hyperbolic.
The vanishing of $\vartheta$ gives easily the uniqueness using a kind of Hodge-Riemann relations.

Note that in  Theorem  \ref{T:Unique_ergodicity_1}, the current $T$ is necessarily extremal in the cone of all positive $\ddc$-closed currents on $X$. 
Indeed, if $T'$ is such a current and $T'\leq T$, then $T'$ is necessarily directed by the foliation and according to the theorem, 
$T'$ is proportional to $T$. Note also that the nef property of $\{T\}$ is a consequence of a general result of independent interest, see Corollary \ref{c:nef} below. That corollary is a byproduct of our theory of densities of currents.

The following trichotomy  gives us a more complete picture of the strong ergodicity obtained in the  study of  arbitrary compact K\"ahler surfaces. 
For the notion of invariant (closed) analytic curves, see Definition \ref{D:invariant-curve}.

\begin{theorem}\label{T:Unique_ergodicity_1bis}  {\rm (\cite[Theorem 1.2]{DinhNguyenSibony18})} Let $\Fc=(X,\Lc,E)$    be  
a   singular  holomorphic foliation   with   only hyperbolic singularities or a bi-Lipschitz lamination   in a 
   compact K\"ahler surface $(X,\omega)$. Then one and only one of the following three possibilities occurs.
\begin{itemize}
\item[(a)] $\Fc$ admits invariant closed analytic curves and all positive directed $\ddc$-closed $(1,1)$-currents are linear combinations, with non-negative coefficients, of the currents of integration on those curves. In particular, these currents are all closed.
\item[(b)] $\Fc$ admits a directed positive closed $(1,1)$-current $T$ of mass $1$ having no mass on invariant closed analytic curves (this property holds when there is no such a curve).
Every directed positive $\ddc$-closed $(1,1)$-current is closed, and if it has no mass on invariant closed analytic curves, then it has no mass on each single leaf and its cohomology class 
is proportional to $\{T\}$. Moreover, $\{T\}$ is nef  and $\{T\}^2=0$.
\item[(c)] $\Fc$ admits a unique 
directed positive $\ddc$-closed and non-closed $(1,1)$-current $T$ of mass $1$ having no mass on each single leaf. 
Every directed positive $\ddc$-closed $(1,1)$-current is a combination, with non-negative coefficients, of $T$ and the currents of integration on invariant closed analytic curves (if there are any). 
Moreover, $\{T\}$ is nef and big.
\end{itemize}   
\end{theorem}

Note  that a general theorem by Jouanolou \cite{Jouanolou78} says that either   there are only finitely many invariant closed curves, or $\Fc$ admits a  first meromorphic integral. When $\Fc$ admits a 
first meromorphic integral,
all leaves are  invariant closed analytic curves, and hence  all directed positive   $\ddc$-closed currents are  closed.

A polynomial vector field in $\C^2$ induces a holomorphic foliation in $\P^2$.
When we fix the maximum of the degrees of its coefficients, if the vector field is generic,
the line at infinity $L_\infty:=\P^2\setminus\C^2$ is an invariant curve, see Ilyashenko-Yakovenko \cite{IlyashenkoYakovenko}. 
The current $[L_\infty]$ is the only directed positive $\ddc$-closed $(1,1)$-current of mass 1 by  Theorem \ref{T:Dinh-Sibony}. 
So Property (a) holds in that case, see \cite{DinhSibony18} for details and also Rebelo \cite{Rebelo} 
for a related result. 

If $\Fc$ is a smooth fibration on $X$, then the directed positive $\ddc$-closed currents are all closed and are generated by the fibers of $\Fc$. They belong to the same cohomology class which is
nef with zero self-intersection. So Property (b) holds in that case. Using a suspension one can also construct examples satisfying Property (b) which are not fibrations, see \cite[Ex.\,1]{FornaessSibonyWold} and replace the circle there by $\P^1$. In such examples, there are two invariant closed curves and infinitely many directed positive closed $(1,1)$-currents of mass 1 having no mass on those curves.

Property (c) implies in particular that  for $x\in X\setminus E,$ a  cluster point of the limit $\tau_{x,R}$  as $R$ tends to infinity  is a current   of the form $\lambda T+\sum\lambda_j[V_j],$
where  $\lambda,\lambda_j\in\R^+$ and $V_j$ are  invariant closed analytic curves.
Property (c) holds for foliations which are, in some sense, generic. There are many examples of such foliations in $\P^2$ without invariant closed analytic curves. The cohomology class of the unique directed $\ddc$-closed $(1,1)$-current here is K\"ahler because $H^2(\P^2,\R)$ is of dimension 1. 
If we blow up the singularities of the foliation, we get examples satisfying the same property and having invariant closed analytic curves. Then, the cohomology class of
the unique directed $\ddc$-closed $(1,1)$-current is no more K\"ahler but it is big. In fact, we have the following general result which is a direct consequence of Theorem \ref{T:Unique_ergodicity_1bis}.

\begin{corollary}
Let $\Fc$ be a singular  holomorphic foliation  with   only hyperbolic singularities or a bi-Lipschitz lamination  in a 
   compact K\"ahler surface $X$. Let $T$ be a positive $\ddc$-closed current directed by $\Fc$ having no mass on invariant closed analytic curves. Then the following properties are equivalent : 
   
\ \     (1) \ $T$ is not closed; \qquad (2) \ $\{T\}$ is big;  \qquad (3) \ $\{T\}^2>0$; \quad and  \quad (4) \ $\{T\}^2\not=0$.
\end{corollary}
 
Note that the hyperbolicity of the singularities is necessary in this result. The foliation on $\P^2$, given on an affine chart by the holomorphic 1-form 
$x_2dx_1-ax_1dx_2$ with $a\in\R$,  admits a non-hyperbolic singularity at 0 as well as diffuse invariant positive closed $(1,1)$-currents whose cohomology classes are K\"ahler.
See also Corollary \ref{c:nef} and Theorem \ref{T:auxiliary} below which apply for foliations with arbitrary singularities.

 \begin{remark}
  \rm  The results in this  subsection  give a complete answer to  Problem 5.8 in  our previous survey \cite{NguyenVietAnh18d} (see also  \cite{Deroin13}).
 \end{remark}

  \begin{problem}\rm  (Generalization  of Problem \ref{Pr:Pk}).
    Are there  any versions of     Theorems \ref{T:Unique_ergodicity_1} and \ref{T:Unique_ergodicity_1bis} for  singular holomorphic  foliations  with  hyperbolic linearizable  singularities
    on 
 compact K\"ahler manifold $X$ of dimension  $k>2$  ? 
  \end{problem}

  \begin{problem}\rm  (Generalization  of Problem \ref{Pr:Pk-bis}).
    Are there  any versions of     Theorems  \ref{T:Unique_ergodicity_1} and \ref{T:Unique_ergodicity_1bis} for singular holomorphic  foliations  with   linearizable  singularities
 on compact K\"ahler manifold $X$ of dimension  $k>2$  ? 
  \end{problem}

\section{Theory of energy, theory of densities and strategy for the  proof of the unique ergodicity}  \label{S:Densities_Unique_Ergodicity}

In this section we outline  the  proof of  Theorem \ref{T:Unique_ergodicity_1bis}. Before  doing so, we   present the main tools need for the proof: Fornaess-Sibony theory of energy and 
our theory of densities for a class of non $\ddc$-closed currents.
We refer the reader to \cite{DinhSibony18b, DinhSibony18} for the case of $\ddc$-closed currents.

\subsection{Energy of positive $\ddc$-closed currents}
The following result was obtained by  Forn{\ae}ss-Sibony in  \cite{FornaessSibony05}. 

\begin{theorem} \label{T:pot}
Let $T$ be a positive $\ddc$-closed current on $X$. Then
it can be represented as 
\begin{equation} \label{e:decompo_posi_har}
T=\Omega+\partial S+\overline{\partial S} + i\ddbar u
\end{equation}
with $\Omega$ a smooth real closed $(1,1)$-form, $S$ a current of bi-degree $(0,1)$ and $u$ a real function in $L^p$ for $p<2.$ 
Moreover, for every such a representation, the currents $\dbar S$ and $\partial\overline S$ do not depend on the choice of $\Omega,S,u$ and they are
forms of class $L^2$, uniquely determined by $T$. 
\end{theorem}

Note that the representation \eqref{e:decompo_posi_har} is not unique but the uniqueness of $\dbar S$ and $\partial \overline S$ and their  membership in the space $L^2$
allow Forn{\ae}ss-Sibony in  \cite{FornaessSibony05} 
to define the {\it energy} $\Ec(T)$ of $T$ as
\begin{equation} \label{e:energy}
\Ec (T):=\int_X\dbar S\wedge \partial\overline S.
\end{equation}
This is a non-negative number which is independent of the choice of $\Omega, S$ and $u$.
It is not difficult to see that $\Ec(T)=0$ if and only if $\dbar S=0,$ and   if and only if $T$ is closed, see \cite{FornaessSibony05} for details.
The following  result improves the regularity of the potential $u$ and its gradients.

\begin{theorem} {\rm  (see \cite[Proposition B.4 ]{DinhNguyenSibony18})           }\label{T:T-rep}
There is a representation as in \eqref{e:decompo_posi_har} such that all currents 
$S,\overline S, \partial S, \partial\overline S, \dbar S$, $\dbar\overline S$ are forms of class $L^2$ and $u$ is a   function of class $L^2$  and 
$ \partial u, \dbar u$ are   forms  of class $L^p$ for every $1\leq p<2$. 
\end{theorem}

The  energy  theory  of  Forn{\ae}ss-Sibony has  recently been developed   in some  new directions, see   \cite{DinhNguyen20}.

\subsection{Tangent  currents of  tensor products of positive $\ddc$-closed currents}   \label{SS:positive_ddc_closed_currents}

  Let $(X,\omega)$ be a compact K\"ahler surface.   Let $\Delta:=\{(x,x):\ x\in X\}$ be  the diagonal of $X\times X.$  We identify a chart of $X$ with the unit ball $\B$ in $\C^2.$
On the chart $\B\times \B$ of $X\times X,$  we use two local coordinate systems: the first system is the standard one $(x,y)$ and the second system is $(z,w):=(x-y,y)$ on which 
$\Delta$ is given by the equation $z=0$. The tangent bundles of $X\times X$ and $\Delta$ are denoted by $\Tan(X\times X)$ and $\Tan(\Delta)$. 
The normal vector bundle of $\Delta$ in $X\times X$ is denoted by $\E:=\Tan(X\times X)|_\Delta/\Tan(\Delta)$, where $\Delta$ is also identified to the zero 
section of $\E$. Denote by $\pi:\E\to\Delta$ the canonical projection. The fiberwise multiplication by $\lambda\in\C^*$ on $\E$ is denoted by $A_\lambda$. 
Over $\Delta\cap (\B\times \B)$, with the coordinates $(z,w)$, $\E$ is identified to $\C^2\times \B$, $\pi$ is the projection $(z,w)\mapsto w$ and $A_\lambda$ 
is equal to the map $a_\lambda(z,w):=(\lambda z,w)$.

Consider two positive $\ddc$-closed $(1,1)$-currents $T_1$ and $T_2$  on $X$.
We will study the density of $T_1\otimes T_2$ near the diagonal $\Delta$ of $X\times X$ via a notion of ``tangent cone'' to $T_1\otimes T_2$ along $\Delta$ that we introduce now.

\begin{definition}    \label{D:admissible-maps} 
\rm A {\it smooth admissible map} is a smooth bijective map $\tau$ from a
neighbourhood of $\Delta$ in $X\times X$ to a neighbourhood of $\Delta$ in $\E$ such that
\begin{enumerate}
\item The restriction of $\tau$ to $\Delta$ is the identity map on $\Delta$; in particular,
the restriction of the differential $d\tau$ to $\Delta$ induces a map from
$\Tan(X\times X ) |_\Delta$ to $\Tan(\E)|_\Delta$; since $\Delta$ is pointwise fixed by $\tau$, the differential $d\tau$ also induces two endomorphisms of $\Tan(\Delta )$ and $\E$  respectively;
\item The differential $d\tau (x,x),$ at each point $(x,x) \in \Delta,$ is a $\C$-linear map from the tangent space to $X\times X$ at $(x,x)$ to 
the tangent space to $\E$ at $(x,x)$;
\item  The endomorphism of $\E$, induced by $d\tau$ (restricted to $\Delta$), is the identity map.
\end{enumerate}
\end{definition}

Note that the dependence of $d\tau (x,x)$ in $(x,x) \in  \Delta$ is  in general not holomorphic.
Consider the exponential map from $\E$ to $X\times X$ with respect to any Hermitian metric on $X\times X$. It defines
a smooth bijective map from a neighbourhood of $\Delta$ in $\E$ to  a neighbourhood of $\Delta$ in $X\times X$. The inverse map is smooth and admissible, see also \cite[Lemma 4.2]{DinhSibony18b}.

\begin{figure}[h]%
\begin{center}
\def\svgwidth{0.9\columnwidth}
\resizebox{0.8\textwidth}{!}{\input{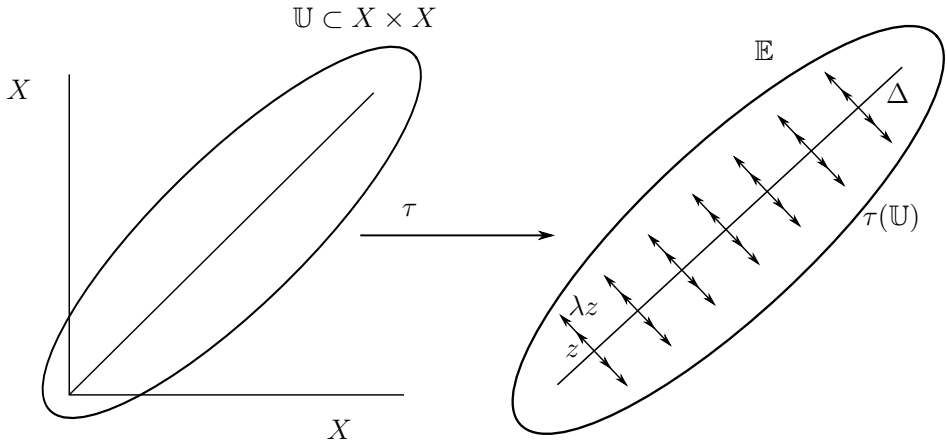}}%
\caption{Admissible map  $\tau$   sends   a
neighbourhood $\U$ of $\Delta$ in $X\times X$ onto a neighbourhood of the  zero section $\Delta$ in $\E.$
The restriction of $\tau$ to $\Delta$  is the identity map on $\Delta; $  the differential $d\tau (x,x)$ at each point $(x,x) \in \Delta$ is a $\C$-linear map from  
$\Tan(X\times X)$ at $(x,x)$ to 
 $\Tan(\E)$ at $(x,x)$;
and the endomorphism of $\E$, induced by $d\tau$ (restricted to $\Delta$), is the identity map.
For $\lambda\in\C^*,$ the  dilatation $A_\lambda$ is the fiberwise multiplication by $\lambda$ on $\E.$  Here  is the illustration
with $\lambda=2.$}  
\label{fig:tangent_currents}
\end{center}
\end{figure}

Let $\tau$ be any smooth admissible map as
above. Define
\begin{equation}\label{e:T_lambda}
(T_1\otimes T_2)_\lambda  := (A_\lambda )_* \tau_*  (T_1\otimes T_2 ).
 \end{equation}
This is a current of degree $4.$ Its domain of definition is some open subset of
$\E$ containing $\Delta$ which increases to $\E$ when $|\lambda|$ increases to infinity. 

Observe that $(T_1\otimes T_2)_\lambda$ is not a $(2, 2)$-current and we cannot speak of its positivity. Moreover,
it is not $\ddc$-closed in general and we cannot speak of its cohomology class.
The present situation is more involved than the case where $T_1$ and $T_2$ are closed because in this case the current
$(T_1\otimes T_2)_\lambda$ is also closed.

By \eqref{e:decompo_posi_har}   we can write for $j\in\{1,2\},$
\begin{equation} \label{e:decompo_posi_har_T12}
T_j=\Omega_j+\partial S_j+\overline{\partial S_j} + i\ddbar u_j,
\end{equation}
where $\Omega_j$ is a closed real smooth $(1,1)$-form, $S_j$ is a current of bi-degree $(0,1)$ and $u_j$ is a real current of bi-degree $(0,0)$. Note that $\partial \overline S_j$ and $\dbar S_j$ are forms of class $L^2$ which are independent of the choice of $\Omega_j,S_j,u_j$.
It turns out that a crucial argument in the proof of Theorem \ref{T:main_2} below is a result on the regularity of the potentials $u_j$ and their gradients, see Theorem \ref{T:T-rep}.

One of the main ingredient of the proof is 
the following theorem.  

\begin{theorem}\label{T:main_2} {\rm (\cite[Theorem 2.2]{DinhNguyenSibony18})}
Let $T_1$ and $T_2$ be  two positive $\ddc$-closed $(1,1)$-currents on a compact K\"ahler surface $(X,\omega)$.   Assume that $T_1$ has no mass on the set $\{\nu(T_2,\cdot)>0\}$ and $T_2$ has no mass on the set $\{\nu(T_1,\cdot)>0\}$. Then, with the above notations, we have the following properties.
\begin{enumerate}
\item The mass of $(T_1\otimes T_2)_\lambda$ on any given compact subset of $\E$ is bounded uniformly on $\lambda$ for $|\lambda|$ large enough. If $\T$ is a cluster value of $(T_1\otimes T_2)_\lambda $ when $\lambda\to\infty ,$ then it
is a positive closed $(2, 2)$-current on $\E$ given by $\T = \pi^* (\vartheta)$ for some
positive measure $\vartheta$ on $\Delta.$ Moreover, if $(\lambda_n )$ is a
sequence tending to infinity such that $(T_1\otimes T_2)_{\lambda_n}  \to \T,$ then $\T$ may depend on $(\lambda_n )$
but it does not depend on the choice of the map $\tau .$
\item
The mass of $\vartheta$ does not depend on the choice of $\T$ and it is given by
\begin{equation*}
\|\vartheta\|= \int_X\Omega_1\wedge \Omega_2-  \int_X \dbar S_1\wedge \partial \overline{S}_2 - \int_X \dbar S_2\wedge \partial\overline{S}_1  .
\end{equation*}
In particular, if $T_1=T_2=T$ with $T=\Omega+\partial S+\dbar\overline S+i\ddbar u$ as in \eqref{e:decompo_posi_har}, then  
\begin{equation}\label{e:mass_bis}
\|\vartheta\|= \int_X\Omega^2-  2\int_X \dbar S\wedge \partial \overline{S} =\int_X\Omega^2-2\Ec (T)  .
\end{equation}
\end{enumerate}
\end{theorem}

Note that in general $\T$ is not unique as this is already the case for positive
closed currents, see \cite{DinhSibony18b} for details. However, the mass formula shows that if one of such currents is zero then all of them are zero.
We can now introduce the following notion.

\begin{definition} \rm  
Any current $\T$ obtained as in Theorem \ref{T:main_2} is called a {\it tangent
current} to $T_1\otimes T_2$ along the diagonal $\Delta .$
\end{definition}

Combining Theorem  \ref{T:main_2}  with  the  works of Berndtsson--Sibony \cite{BerndtssonSibony}  Demailly-P\u{a}un \cite{DP} and  Siu \cite{Siu}, we obtain  the following result and refer to McQuillan \cite{McQuillan} and  Burns--Sibony \cite{BurnsSibony} for some related results in the foliation setting.

\begin{corollary} \label{c:nef} {\rm (\cite[Corollary 2.4]{DinhNguyenSibony18})}
Let $T$ be a positive $\ddc$-closed $(1,1)$-current of a compact K\"ahler surface $X$. Assume that the set $\{\nu(T,\cdot)>0\}$ is of Hausdorff $2$-dimensional measure $0$.
Then the cohomology class $\{T\}$ of $T$ is nef, and when $T$ is not closed, $\{T\}$ is also big. In particular, if $T$ is a positive closed $(1,1)$-current having no mass on proper analytic subsets of $X$, then $\{T\}$ is nef.
\end{corollary}

Recall that if
a closed subset $Y$ of a complex surface $X$ is laminated by Riemann
surfaces, then it admits an open covering ${\U_j}$ and on each $\U_j$ there is a homeomorphism $\varphi_j = (h_j , \lambda_ j ) :\ \U_j \cap Y \to  \D \times \T_j$,
where $\T_j$ is a locally compact  metric space and the maps $\varphi_j^{ -1}(z,t)$, with $(z,t)\in\D\times \T_j$, are holomorphic in $z$.
Moreover, on their domains of definition, the transition maps have the form
$$\varphi_k \circ \varphi^{-1}_j (z, t) = \big(h_{jk} (z, t), \lambda_{jk} (t)\big),$$
where $ h_{jk} (z, t)$ is holomorphic with respect to $z$ and $\lambda_{jk}(t)$ do not depend on $z$.
We can choose $\T_j$ as the intersection of a holomorphic disc with $Y$ and $\varphi_j$ such that its restriction to $\T_j$ is the canonical map from $\T_j$ to $\{0\}\times\T_j$. With this choice,
when all $\varphi_j(z,t)$ are bi-Lipschitz maps, we say that the lamination is {\it bi-Lipschitz}.

The following result gives us  the  vanishing  of the  tangent  currents in the setting of foliations and laminations. 
   
\begin{theorem} \label{T:main_3} {\rm (\cite[Theorem 2.5]{DinhNguyenSibony18})}
Let $\Fc$    be  either a singular  holomorphic foliation   with   only hyperbolic singularities,
or a bi-Lipschitz lamination, in a compact K\"ahler surface $X$. 
Then for every  positive $\ddc$-closed  current $T$  directed  by $\Fc$ which  does not give mass to any  invariant closed analytic curve,  zero is  the unique  tangent  current  to $T\otimes T$  along  the diagonal $\Delta.$
\end{theorem}

The last theorem expresses that the current $T\otimes T$ is not too singular along the diagonal of $X\times X$ as its density along the diagonal is  zero.
Now  we outline  the main steps of the proof of Theorem \ref{T:main_3}.

Consider 
a positive $\ddc$-closed $(1, 1)$-current $T$ directed by $\Fc$. 
We can show that if $T$ has positive mass on a leaf,  then this leaf is an invariant closed analytic curve of $\Fc,$  see Theorem  \ref{T:auxiliary} below.
So for Theorem  \ref{T:main_3}, we can assume that  $T$ has no
mass on  each  single leaf of $\Fc$. Using  Proposition   \ref{P:decomposition} and Definition \ref{D:Lelong}, a straightforward computation  shows that 
$\nu(T,x)=0$ for all $x$ outside the singularities of $\Fc$.
Since positive $\ddc$-closed $(1,1)$-currents have no mass on finite sets, we can apply 
Theorem \ref{T:main_2} to the tensor product $T\otimes T$. 

Consider a tangent current $\T$ to  $T\otimes T$ along $\Delta$. With the notation as in the above sections, 
there is a sequence $\lambda_n$ converging to infinity and a positive measure $\vartheta$ on $\Delta\simeq X$ such that 
$$\T=\lim_{n\to\infty} (T\otimes T)_{\lambda_n} = \pi^*(\vartheta).$$ 
We can identify $\vartheta$ with a positive measure on $X$. Recall that by Theorem \ref{T:main_2} the mass $m$ of
$\vartheta$ does not depend on the choice of $\T$.  Using  Definition \ref{D:admissible-maps} and  \eqref{e:T_lambda} outside   the singularities $E,$ we can show the 
following result, see  \cite{Kaufmann} for a related  situation. 

\begin{proposition}\label{P:Kaufmann}  {\rm (\cite[Proposition 4.1]{DinhNguyenSibony18})}
For every choice of the tangent current $\T$, the measure  $\vartheta$  is  supported on the singularities of $\Fc$. 
\end{proposition}
 
For any function or more generally a current $f(s)$, depending on the parameter $s>0$, we denote {\it the expectation} of $f(s)$ on the interval 
$(0,s]$ by $\Ebf(f(s))$. This is the mean value of $f$ on the interval $(0,s]$ which is given by the formula
\begin{equation*}
\Ebf(f(s)):=s^{-1}\int_0^s f( t)dt.
\end{equation*}
 The  difficult part in the proof of  Theorem \ref{T:main_3} is  the  following
 
\begin{proposition}\label{P:singularities} {\rm (\cite[Proposition 4.2]{DinhNguyenSibony18})}
We have 
$$\lim_{s\to\infty} \Ebf((T\otimes T)_{e^s}) =  0$$
in a neighbourhood of each point $(a,a)\in\Delta$, where $a$ is any singular point of $\Fc$.
\end{proposition}

To  prove Proposition \ref{P:singularities}, we  carry out a  very delicate analysis of the current $(T\otimes T)_\lambda$ with $\lambda=e^s$  near  the singular point $a.$
This  leads us  to  a geometric study  of the leaves of $\Fc$ near  the singular point $a.$
Here we make  an essential use of the  hyperbolic  nature of the point $a.$
 
\proof[End of the proof of Theorem \ref{T:main_3}]
Let $\Mc$ be the  set of  $2$-current on $\E$ with  measure coefficients. 
We introduce  natural semi-distances $\dist_m$ ($m\in\N$) on  $\Mc$  as follows.
Let $\E_m$ be  an increasing sequence of domains in $\E$ such that $\E_0$ contains 
the image of the zero section of  the vector bundle $\E$   and  $\E_m\Subset \E_{m+1}$ and $\bigcup_{m\in\N} \E_m=\E.$
Recall from Subsection \ref{ss:forms_measures} that $\Dc^2_1(\E)$  (resp.  $\Dc^2_1(\E_m)$)
is  the space of $2$-forms with compact support of class  $\Cc^1$  on $\E$  (resp. on $\E_m$). We fix a  finite    atlas of the vector bundle $\E.$
For  $\alpha\in\Dc^2_1(\E_m),$ let $\|\alpha\|_{\Cc^1}$  be  the
sum of $\Cc^1$-norms of the coefficients of $\alpha$ with respect to the  restriction of the  above atlas   to $\E_m.$
If $\T$ and $\S$  are two  $2$-currents  in $\Mc,$ define for $m\in\N,$
$$
\dist_m(\T,\S):=\sup_{\|\alpha\|_{\Cc^1}\leq 1}\left|\langle \T-\S,\alpha\rangle\right|,
$$
where $\alpha$ is a form in $\Dc^2_1(\E_m).$ 

A current $\T\in\Mc$  may be  regarded as a  $2$-form with measure  coefficients. 
For $m\in\N,$  let  $\|\T\|_m$  denote the  sum of the mass  of the variation of the coefficient  measures
of $\T\in\Mc$ with respect to the  restriction of the  above atlas   to $\E_m.$

We make the following observation. Let $m\in\N$ be a fixed number and  $(\T_n)$   a sequence in $\Mc$ such that $\| T_n\|_m$ 
is uniformly bounded in $n. $ 
If $\T_n|_{\E_m}$  converge weakly to $\T|_{\E_m}$   as $n$ tends to infinity (i.e., $\langle \T_n ,\alpha\rangle\to \langle \T,\alpha\rangle$ for all $\alpha\in\Dc^2(\E_m)$), then
by  Arzel\`a-Ascoli theorem  we get that  $\dist_m(\T_n,\T)\to 0.$ Conversely, if   $\dist_m(\T_n,\T)\to 0,$ then $\T_n|_{\E_m}$  converge  to  $\T|_{\E_m}.$

For  $\Ac,\Bc\subset \Mc,$ let
$$
\dist_m(\Ac,\Bc):=\inf_{\T\in \Ac,\  \S\in \Bc} \dist_m(\T,\S).
$$
We  say that a family $\Bc\subset \Mc$ is  bounded if  for every $m,$  the set $\{\|  \T\|_m:\  \T\in\Bc\}$  is  bounded.
It  is  clear that if  $\Bc$ is  bounded then  its closure  $\overline\Bc$ in $\Mc$ with respect  to the weak topology is also  bounded.

Now consider the family  $\Bc:=\{ (T_1\otimes T_2)_\lambda:\  \lambda\in\C,\ |\lambda|\geq 1\}\subset\Mc.$
By  Theorem \ref{T:main_2} (1), $\Bc$ is  a bounded family. So $\overline\Bc$ is also  bounded.
Let $\Kc$ be the set of all tangent currents $\T.$ Also by  Theorem \ref{T:main_2} (1), $\Kc$ is a closed subset of $\overline\Bc$ and all elements of $\Kc$  are  positive currents on $\E.$
So $\Kc$ is  a compact set in $\Mc.$   Moreover,   applying Theorem \ref{T:main_2} (1) 
yields that
$$
\lim_{\lambda\to\infty} \dist_m( (T_1\otimes T_2)_\lambda, \Kc)=  0\quad\text{for each}\quad m\in\N.
$$
Otherwise, there would be $m\in\N,$  $\epsilon>0$ and a sequence $(\lambda_n)\to\infty$ such that 
$\dist_m( (T_1\otimes T_2)_{\lambda_n}, \Kc)\geq \epsilon$  and $(T_1\otimes T_2)_{\lambda_n}\to \T\in\Kc$  weakly.
But by the above  observation,  the last limit  would imply that   $\lim_{n\to\infty}\dist_m( (T_1\otimes T_2)_{\lambda_n}, \T)=0,$ which shows, in turn, that $\lim_{n\to\infty}
\dist_m( (T_1\otimes T_2)_{\lambda_n}, \Kc)=0,$
this is  a contradiction.

For  $\lambda=e^s$ with $s>0,$ write 
$$
 \Ebf((T\otimes T)_{e^s}) ={1\over s }\int_0^s  (T\otimes T)_{e^t}dt.
$$
Since  
$\Ebf((T\otimes T)_{e^s}) $ is a convex combination  of $(T\otimes T)_{e^t}$
for $t\in[0,s],$
it follows that
$$
\lim_{s\to\infty} \dist_m( \Ebf((T\otimes T)_{e^s}),\widehat \Kc)=0,
$$
where  $\widehat\Kc$ is the convex hull of $\Kc$ in $\Mc.$  As $\Kc$ is closed,  so is $\widehat\Kc.$  

Let $\T'$ be a limit current of $\Ebf((T\otimes T)_\lambda)$ when $\lambda>0$ tends to infinity. 
We infer  from  the last limit that $\dist_m(\T', \widehat \Kc)=0$ for every $m.$ Hence by the  above observation,
for every $m\in\N,$ there exists $\T_m\in\widehat\Kc$ such that $\T'|_{\E_m}=\T_m|_{\E_m}.$
It is  then easy to see that the Ces\`aro means  ${1\over m}\sum_{j=0}^{m-1} \T_j$  tend weakly to $\T'$ as $m$ tends to infinity.
Since these means belong  to $\widehat\Kc,$
the current $\T'$ belongs also to $\widehat\Kc.$ So by Theorem \ref{T:main_2} (1), we have $\T'=\pi^*(\vartheta')$ for some positive measure $\vartheta'$ of mass $m$ on $\Delta\simeq X$.
By  Propositions \ref{P:Kaufmann} and \ref{P:singularities},  we have $\vartheta'=0$. Therefore, we get $m=0$ and hence, by the mass formula in Theorem \ref{T:main_2}, 
we have $\T=0$ for any choice of $\T$. This proves the vanishing theorem.
\endproof


\subsection{Sketchy  proof of the unique  ergodicity for singular holomorphic foliations} \label{SS:Unique_ergodicity_foliations}

The following result holds in a more general setting but we only state it in the case we use, see also \cite{DinhSibony18, FornaessSibony05}. Here, we do not
need to assume that the singularities of the foliation are hyperbolic.

 \begin{theorem} \label{T:auxiliary} {\rm (\cite[Theorem 2.6]{DinhNguyenSibony18})}
 Let $T$ be a positive $\ddc$-closed $(1,1)$-current, on a  compact K\"ahler surface $X$, which is directed by 
a singular  holomorphic foliation or by a bi-Lipschitz lamination. 
\begin{enumerate}
\item[(a)] If $T$ has a positive mass on a leaf $L$, then $\overline L$ is a closed analytic curve and $\overline L\setminus L$ is contained in the set of singularities of the foliation. Moreover, we can write 
$T=T'+T_{\rm an}$, where $T'$ is a directed positive $\ddc$-closed $(1,1)$-current which is diffuse, i.e. having no mass on each single leaf, and 
$T_{\rm an}$ is a finite or countable combination,  with non-negative coefficients, of currents of integration on invariant closed analytic curves. 
\item[(b)] Assume that $T$ gives no mass to any invariant closed analytic  curve. Then $T$ is diffuse and its cohomology class $\{T\}$ is nef. Moreover, $\{T\}$ is also big when $T$ is not closed.
\end{enumerate}
 \end{theorem}

 The first step of our proof  consists in proving the following lemma.
 
\begin{lemma}\label{L:closed}  {\rm (\cite[Lemma 2.7]{DinhNguyenSibony18})}
Let $\Fc$ be either a singular holomorphic foliation  with only  hyperbolic  singularities, or a bi-Lipschitz
lamination    in  a compact K\"ahler surface $(X,\omega).$ 
Let $T_1$ and $T_2$ be two positive $\ddc$-closed currents of mass $1$ directed by $\Fc$  such that
neither of them gives mass to any  invariant closed analytic curve.
Then $T_1-T_2$ is   a closed current. If both $T_1$ and $T_2$ are closed, then we have $\{T_1\}^2=\{T_2\}^2=\{T_1\}\smallsmile\{T_2\}=0$.
\end{lemma}
\proof
Since   both $T_1 $ and  $ T_2$ do not give mass to any  invariant closed analytic curve, it follows from Theorem \ref{T:auxiliary} that  
$\nu(T_1,x)=\nu(T_2,x)=0$ for all $x$ outside the singularities of $\Fc.$ 
Since $T_1$ and $T_2$ do not  give mass to this finite set, we  see that $T_1$ and $T_2$  satisfy the assumption of Theorem \ref{T:main_2}.

By  \eqref{e:decompo_posi_har_T12} and Stokes' theorem, we have (the second integral is the mass of $T_j$ which is assumed to be 1)
\begin{equation}\label{e:Omega_omega}
 \int_X \Omega_j\wedge \omega= \int_X T_j\wedge \omega=1\quad\text{for}\quad j=1,2.
\end{equation}
Applying  Theorems \ref{T:main_2} and \ref{T:main_3} to  each one of the  three directed  positive $\ddc$-closed currents $T_1,$ $T_2$ and $T_1+T_2$,  we obtain that all
 $T_1\otimes T_1,$  $T_2\otimes T_2$ and   $(T_1+T_2)\otimes( T_1+T_2)$     admit  zero as the unique  tangent current along the  diagonal $\Delta.$
This, combined with \eqref{e:decompo_posi_har_T12} and \eqref{e:mass_bis}, implies that
\begin{equation}\label{e:Omega_square}
\begin{split}
 \int_X  \Omega_1^2=2\int_X\dbar S_1\wedge \partial\overline S_1, \quad&    \int_X  \Omega_2^2=2\int_X\dbar S_2\wedge \partial\overline S_2  \\
\text{and} \quad  \int_X  (\Omega_1+\Omega_2)^2&=2\int_X\dbar (S_1+S_2)\wedge \partial (\overline S_1+\overline S_2).
 \end{split}
\end{equation}
If both $T_1$ and $T_2$ are closed, we deduce from the discussion after \eqref{e:energy} 
that $\dbar S_1=\dbar S_2=0$ and hence all integrals in \eqref{e:Omega_square}
vanish. This implies $\{T_1\}^2=\{T_2\}^2=\{T_1\}\smallsmile\{T_2\}=0$ as stated in the second assertion of the lemma.

Let $T:=T_1-T_2,$ $\Omega:=\Omega_1-\Omega_2,$   $S:=S_1-S_2$ and $u:=u_1-u_2.$ We infer  from \eqref{e:decompo_posi_har_T12} and \eqref{e:Omega_omega} that 
\begin{equation} \label{e:decompo_posi_har_T}
T=\Omega+\partial S+\overline{\partial S}+i\ddbar u\quad\text{and}\quad\int_X \Omega\wedge \omega=0.
\end{equation}
Moreover, it follows from  \eqref{e:Omega_square} that
\begin{equation} \label{e:integral}
\int_X  \Omega^2 =    \int_X  (\Omega_1 -\Omega_2)^2 =   
2\int_X  \Omega_1^2+2\int_X  \Omega_2^2- \int_X  (\Omega_1+\Omega_2)^2\\
 = 2 \int_X\dbar S\wedge \partial\overline S.
 \end{equation}

On one hand, since $\dbar S$ is an $L^2$  $(0,2)$-form,  the  current $\dbar S\wedge \partial\overline S= \dbar S\wedge \overline{\dbar S}  $ is a positive measure. So the last integral in \eqref{e:integral} is non-negative and  it vanishes if only if $\dbar S=0$  almost everywhere.
On the other hand, since  we know by \eqref{e:decompo_posi_har_T}  that $\int_X \Omega\wedge \omega=0,$ the cohomology class  of $\Omega$ is  a primitive class of $H^{1,1}(X,\R).$ Therefore, it follows
from  the  classical Hodge--Riemann  theorem that  the first integral in \eqref{e:integral} is non-positive, see e.g. \cite{Voisin}.
We conclude that  $\dbar S=0$  almost everywhere.
This and \eqref{e:decompo_posi_har_T} imply that $dT=0.$  The proof of the lemma is thereby completed.
\endproof

\proof[End of the proof of Theorem \ref{T:Unique_ergodicity_1bis} (see also  \cite{FornaessSibony05})]
We only consider the case of a foliation because the case of a lamination can be obtained in the same way.
It is clear that not more than one property in the theorem holds.
By Theorem \ref{T:existence_harmonic_currents}, there exists a positive $\ddc$-closed current $T_1$ directed by $\Fc.$ We can assume that Property (a) in the theorem does not hold. So we can find a current $T_1$ of mass 1 which has no mass on each single leaf of $\Fc$, see Theorem \ref{T:auxiliary}. We show that either Property (b) or (c) holds.

\medskip\noindent
{\bf Case 1.} Assume that there is such a current $T_1$ which is not closed. We show that the foliation satisfies Property (c) in the theorem.
By Theorem \ref{T:auxiliary}, the class $\{T_1\}$ is nef and big. It remains to prove the 
uniqueness of $T_1$. Assume by contradiction that there is another positive $\ddc$-closed current $T_2$ of mass 1 directed by $\Fc$. 
If there is such a current which is closed, then we assume that $T_2$ is closed.
So we have 
$$\int_X T_1\wedge \omega=\int_XT_2\wedge \omega=1.$$
We need to find a contradiction.  

Consider a flow box $\U$ away from the  set of singularities $E$ that we identify with $\D\times \Sigma,$  $\D$ being as usual the unit disc and $\Sigma$ being a transversal of $\U.$  By  Proposition \ref{P:decomposition}, we have
$$T_j = \int_\Sigma h^\alpha_j [V_\alpha ]d\mu_j (\alpha),\quad j = 1, 2,$$
where  for $\alpha\in\Sigma,$  $[V_\alpha]$  is the current of integration  on the plaque $V_\alpha\simeq  \D\times \{\alpha \}.$
Let $\mu = \mu_1 + \mu_2$ and write $\mu_j = r_j \mu$ with a non-negative bounded function $r_j\in L^\infty(\mu)$.  Then we have
$$T_1 - T_2 = \int_\Sigma  \big(h^\alpha_1 r_1 (\alpha) - h^\alpha_2 r_2 (\alpha) \big)[V_\alpha]d\mu(\alpha).$$
Since we know by Lemma  \ref{L:closed} that $T_1 - T_2$ is a  closed  current, 
$h^\alpha_1 r_1 (\alpha) - h^\alpha_2 r_2 (\alpha)$ is constant, for $\mu$-almost every $\alpha$,  that we will denote by $c(\alpha)$. 

We decompose $c(\alpha)\mu(\alpha)$ on the space of plaques $\Sigma$ and obtain that
$ c(\alpha)\mu(\alpha) = \nu_1 -\nu_2$ for
 mutually singular positive measures $\nu_1$ and $\nu_2$. Then
$$T_1 - T_2 = [V_\alpha ]\nu_1 (\alpha) -[V_\alpha ]\nu_2 (\alpha) = T^+ - T^-$$
for positive closed currents $T^\pm.$ These currents fit together to a global positive closed
currents on $X \setminus E.$ Observe  that the mass of $T^\pm$ is  bounded  by   the mass of $T_1+T_2$.  So  the mass of $T^\pm$ is bounded near $E$. Since $E$ is a finite set,
$T^\pm$ extend as positive closed currents through $E,$  see \cite{Sibony,Skoda}. 
Recall that positive $\ddc$-closed currents of bidimension $(1,1)$  have no mass on finite sets. Therefore,
since we assumed above that $T_1\not=T_2$, we have either $T^+\not=0$ or $T^-\not=0$. It follows from our choice of $T_2$ that $T_2$ is closed and hence $T_1$ is closed as well. 
This is a contradiction which shows that such a current $T_2$ as above does not exist.

\medskip\noindent
{\bf Case 2.} Assume now that all directed positive $\ddc$-closed $(1,1)$-currents are closed. Consider arbitrary directed positive closed $(1,1)$-currents $T_1$ and  $T_2$ of mass 1 which are diffuse. 
So by Theorem \ref{T:auxiliary} applied to $T_1,T_2$, the classes $\{T_1\}$ and $\{T_2\}$ are nef. By Lemma \ref{L:closed}, we have
$\{T_1\}^2=\{T_2\}^2=\{T_1\}\smallsmile\{T_2\}=0$. 
We show that Property (b) in the theorem holds. It is enough to show that $\{T_1\}=\{T_2\}$. 

Since $T_1$ and $T_2$ are of mass 1, we have $(\{T_1\}-\{T_2\})\smallsmile \{\omega\}=0$. So $\{T_1\}-\{T_2\}$ is a primitive class in the Hodge cohomology group $H^{1,1}(X,\R)$ of $X$. By the classical Hodge-Riemann theorem, we have $(\{T_1\}-\{T_2\})^2<0$ unless 
$\{T_1\}-\{T_2\}=0$, see e.g. \cite{Voisin}. Using that $\{T_1\}^2=\{T_2\}^2=\{T_1\}\smallsmile\{T_2\}=0$, we deduce that $\{T_1\}=\{T_2\}$. This ends the proof of the theorem.
\endproof

\proof[End of the proof of Theorem \ref{T:Unique_ergodicity_1}]
We only consider the case of a foliation because the case of a lamination can be proved in the same way.
By hypothesis, the foliation has no invariant closed analytic curve. Moreover, by
Theorem \ref{T:Unique_ergodicity_1bis}, Property (c) in that theorem holds. It follows that 
the foliation admits a unique directed positive $\ddc$-closed current $T$ of mass 1.
This current is not closed and $\{T\}$ is nef and big.
Since  every cluster point of $\tau_{x,R}$  as $R$ tends to infinity is  a positive  $\ddc$-closed current of mass $1,$
 $\tau_{x,R}$ converges necessarily to $T$  as $R$ tends to infinity.   
\endproof

\section{Lyapunov--Oseledec theory for  Riemann surface laminations} \label{S:Lyapunov}

The purpose of this section is to present some  recent results obtained in our   works \cite{NguyenVietAnh17a,NguyenVietAnh17b}.
 
\subsection{Cocycles}
The notion of (multiplicative)  cocycles   have been   introduced 
in  \cite{NguyenVietAnh17a}  for ($N$-dimensional real or complex)  laminations. For  the sake of simplicity  we  only formulate this notion  for Riemann surface laminations
with singularities in this article.  In the rest of the section   
 we make  the following convention: $\K$ denotes either the  field $\R$ or $\C.$
Moreover,  given any  integer $d\geq 1,$  $\GL(d,\K)$   denotes  the general linear group of degree $d$ over $\K$
and $\P^d(\K)$ denotes the  $\K$-projective space of dimension $d.$
\begin{definition}\label{D:cocycle} {\rm (Nguyen \cite[Definition 3.2]{NguyenVietAnh17a}). }\rm
Let $\Fc=(X,\Lc,E)$ be a Riemann surface lamination with  singularities and $\Omega:=\Omega(\Fc)$ be its sample-path space.
 A {\it $\K$-valued   cocycle} (of rank $d$)   is
  a   map  
$\mathcal{A}:\ \Omega\times \R^+ \to  \GL(d,\K)      $
such that\\  
(1)  ({\it identity law})  
$\mathcal{A}(\omega,0)=\id$  for all $\omega\in\Omega ;$\\
(2) ({\it homotopy law}) if  $\omega_1,\omega_2\in \Omega_x$ and $t_1,t_2\in \R^+$ such that 
     $\omega_1(t_1)=\omega_2(t_2)$
and $\omega_1|_{[0,t_1]}$ is  homotopic  to  $\omega_2|_{[0,t_2]}$ (that is, the path $\omega_1|_{[0,t_1]}$ can be  deformed  continuously on  $L_x$ to the path  $\omega_2|_{[0,t_2]},$ 
the two endpoints of $\omega_1|_{[0,t_1]}$  being kept fixed  during the deformation), then 
$$
\mathcal{A}(\omega_1,t_1)=\mathcal{A}(\omega_2,t_2);
$$
(3) ({\it multiplicative law})    $\mathcal{A}(\omega,s+t)=\mathcal{A}(\sigma_t(\omega),s)\mathcal{A}(\omega,t)$  for all  $s,t\in \R^+$ and $\omega\in \Omega$  (see \eqref{e:shift} for $\sigma_t$);\\
(4) ({\it measurable law})  the {\it local expression} of $\mathcal{A}$ on each  laminated  chart is   Borel measurable.
Here,   the {\it local expression} of  $\mathcal A$ on the laminated chart 
$\Phi:  \U\to \D\times \T,$  is the map
$A:\  \D\times \D\times \T\to\GL(d,\K)$  defined  by
$$
A(y,z,t):=\mathcal A(\omega,1),
$$
where  $\omega$ is  any leafwise path  such that $\omega(0)=\Phi^{-1}(y,t),$ $\omega(1)=\Phi^{-1}(z,t)$ 
and  $\omega[0,1]$ is  contained in the simply connected  plaque   $\Phi^{-1}(\cdot,t).$   
 \end{definition}
 
 A cocycle  $\mathcal A$ on a  smooth Riemann surface lamination with singlarities  $\Fc$  is called    {\it   smooth}  if,
for each laminated chart $\Phi$ as  above,  the local expression   $A$ of $\mathcal A$  is   smooth with respect to  $(y,z)$
 and its  partial  derivatives of any  total order 
 with respect to $(y,z)$ are jointly continuous   in $(y,z,t).$
  
 The  cocycles of rank $1$  have been  investigated  by several  authors (see, for example, Candel \cite{Candel03}, Deroin \cite{Deroin05}, etc.).   
 The  holonomy cocycle (or equivalently the normal  derivative cocycle)  of the  regular part $(X\setminus E,\Lc)$ of a singular holomorphic  foliation   $\Fc=(X,\Lc,E)$ with $\dim_\C X=n$ is
 a typical   example of  $\C$-valued cocycles of rank $n-1.$   These   cocycles   capture  the topological aspect of the considered  foliations. Moreover, 
we can produce  new  cocycles from the old ones by performing some basic operations such as the wedge  product and the tensor product
(see \cite[Section 3.1]{NguyenVietAnh17a}). 

\subsection{Oseledec multiplicative ergodic theorem}

Now  we are in the position to state the Oseledec multiplicative  ergodic theorem for  Riemann surface laminations with singularities.
 \begin{theorem} \label{T:VA_general}  {\rm (Nguyen \cite[Theorem 3.11]{NguyenVietAnh17a}).}
 Let $\Fc=(X,\Lc,E)$ be  a    Riemann surface lamination with singlarities. 
Let  $\mu$ be a  harmonic measure which is  also  ergodic.
Consider  a    cocycle
$\mathcal{A}:\ \Omega\times \R^+ \to  \GL(d,\K)    .  $   Assume that the following integrability condition  is  satisfied for some real number $t_0>0:$  
\begin{equation}\label{e:integrability}
\int_{x\in X}\big(\int_{\Omega_x} \sup_{t\in[0,t_0]}\log^+\|\mathcal A(\omega,t)\|dW_x(\omega)\big)d\mu(x)<\infty,
\end{equation}
where $\log^+:=\max(0,\log).$
Then  there exist  a leafwise saturated  Borel  set $Y\subset X$ of  total $\mu$-measure  and a number $m\in\N$  together with $m$ integers  $d_1,\ldots,d_m\in \N$  such that
the following properties hold:
\begin{itemize}
\item[(i)] For   each $x\in Y$  
 there   exists a  decomposition of $\K^d$  as  a direct sum of $\K$-linear subspaces 
$$\K^d=\oplus_{i=1}^m H_i(x),
$$
 such that $\dim H_i(x)=d_i$ and  $\mathcal{A}(\omega, t) H_i(x)= H_i(\omega(t))$ for all $\omega\in  \Omega_x$ and $t\in \R^+.$   
Moreover,  $x\mapsto  H_i(x)$ is   a  measurable map from $  Y $ into the Grassmannian of $\K^d.$
For each $1\leq i\leq m$ and each $x\in Y,$ let $V_i(x):=\oplus_{j=i}^m H_j(x).$  Set $V_{m+1}(x)\equiv \{0\}.$
\item[
(ii)]  There   are real numbers 
$$\chi_m<\chi_{m-1}<\cdots
<\chi_2<\chi_1,$$
  and   for  each $x\in Y,$ there is a set $F_x\subset \Omega_x$ of total $W_x$-measure such that for every $1\leq i\leq m$ and  every  $v\in V_i(x)\setminus V_{i+1}(x)$
 and every  $\omega\in F_x,$
\begin{equation}
\label{e:Lyapunov}
\lim\limits_{t\to \infty, t\in \R^+} {1\over  t}  \log {\| \mathcal{A}(\omega,t)v   \|\over  \| v\|}  =\chi_i.    
\end{equation}
Moreover, 
\begin{equation}
\label{e:Lyapunov_max}
\lim\limits_{t\to \infty, t\in \R^+} {1\over  t}  \log {\| \mathcal{A}(\omega,t)  \|}  =\chi_1    
\end{equation}
 for  each $x\in Y$  and for   every  $\omega\in F_x.$ 
\end{itemize}
Here    $\|\cdot\|$  denotes the standard   Euclidean norm of $\K^d.$  
 \end{theorem}  
The  above  result  is  the counterpart, in the context of  Riemann surface laminations with singularities,
of the classical  Oseledec   multiplicative ergodic theorem for maps   (see \cite{KatokHasselblatt,Oseledec}).
In fact, Theorem 3.11 in  \cite{NguyenVietAnh17a} is  much more general than  Theorem \ref{T:VA_general}.
Indeed, the former is formulated for $l$-dimensional  laminations endowed with    leafwise Riemannian metrics, which satisfy some
reasonable geometric   conditions.
  
 Assertion  (i) above tells us that the  Oseledec  decomposition exists for all points $x$ in  a leafwise saturated Borel  set
of  total $\mu$-measure and that this  decomposition is holonomy invariant.
Observe that  the Oseledec  decomposition  in  (i)  depends only on $x\in Y,$ in particular, it 
does not depend on paths $\omega\in\Omega_x.$

The decreasing  sequence  of  subspaces  of $\K^d$ given by assertion (i):
$$
\{0\}\equiv V_{m+1}(x)\subset V_m(x)\subset \cdots \subset V_1(x)=\K^d
$$
is  called the {\it Lyapunov filtration} associated to $\mathcal A$  at a given point $x\in Y.$ 
 
 The   numbers  $\chi_m<\chi_{m-1}<\cdots
<\chi_2<\chi_1$ given by  assertion (ii) above are called  the {\it  Lyapunov exponents} of the cocycle $\mathcal{A}$ with respect to the harmonic measure $\mu.$
  It follows   from formulas (\ref{e:Lyapunov}) and (\ref{e:Lyapunov_max}) above  that these characteristic numbers measure heuristically the  expansion rate of
  $\mathcal A$ along different vector-directions $v$ and  along leafwise Brownian trajectories.
In other words,  the stochastic formulas (\ref{e:Lyapunov})-(\ref{e:Lyapunov_max})  highlight the  dynamical  character  of the Lyapunov exponents.

\subsection{Estimates of the Lyapunov exponents of compact smooth hyperbolic  laminations}\label{SS:estimates_of_Lyapunov_exponents}

Let $\mathcal{A}:\ \Omega\times \R^+ \to  \GL(d,\K)$ be a smooth cocycle  defined on a  compact smooth hyperbolic  Riemann surface lamination (without singularities) $\Fc=(X,\Lc).$ 
Observe that  the map $\mathcal{A}^{*-1}:\ \Omega\times \R^+ \to  \GL(d,\K),$ defined by    
  $\mathcal{A}^{*-1}(\omega,t):= \big (\mathcal A(\omega,t)\big)^{*-1},$ is  also a cocycle,
where  $A^*$ (resp. $A^{-1}$) denotes as  usual the transpose (resp. the inverse)  of a square  matrix $A.$ 

We define  two functions  $\bar\delta(\mathcal A),\ \underline\delta(\mathcal A):\ X\to\R$ as well  as four quantities  $\bar\chi_{\max}( \mathcal A ),$
$ \underline\chi_{\max}( \mathcal A ),$    $ \bar\chi_{\min}( \mathcal A ),$
$\underline\chi_{\min}( \mathcal A )$
   as  follows.
  Fix  a point $x\in  X,$ an element $u\in\K^d\setminus \{0\}$  and     a  simply connected plaque $K$   of $\Fc$ passing through
$x.$
Consider   the  function  $f_{u,x}:\  K\to \R$ defined by
\begin{equation}\label{eq_function_f}
f_{u,x}(y):= \log {\| \mathcal A(\omega,1)u \|\over  \| u\|} ,\qquad  y\in K,\ u\in\K^d\setminus\{0\},
\end{equation}
where  $\omega\in \Omega$ is any path  such that $\omega(0)=x,$ $\omega(1)=y$ 
and that $\omega[0,1]$ is  contained in $K.$  Then define
\begin{equation}\label{eq_formulas_delta}
\bar \delta(\mathcal A)(x):=\sup_{u\in \K^d:\ \|u\|=1} (\Delta_P f_{u,x})(x)\ \ \text{and}\  \
 \underline \delta(\mathcal A)(x):=\inf_{u\in \K^d:\ \|u\|=1} (\Delta_P f_{u,x})(x),\\
\end{equation}
where $\Delta_P$ 
 is, as  usual,   the  Laplacian   
on the leaf $L_x$ induced  by the leafwise Poincar\'e metric  $g_P$    (see formulas \eqref{e:Laplacian_disc} and \eqref{e:Delta_commutation}).
We also define
\begin{equation}\label{eq_formulas_chi}
\begin{split}
\bar\chi_{\max}=\bar\chi_{\max} (\mathcal A)&:=\int_X \bar\delta(\mathcal A)  (x)   d\mu(x),\\
\underline\chi_{\max}=\underline\chi_{\max} (\mathcal A)&:=\int_X \underline\delta(\mathcal A)  (x)   d\mu(x);\\
\underline\chi_{\min}=\underline\chi_{\min}( \mathcal A )&:=-    \bar\chi_{\max}(\mathcal A^{*-1}) ,\\
\bar\chi_{\min}=\bar\chi_{\min}( \mathcal A )&:=-   \underline\chi_{\max}(\mathcal A^{*-1}) .
\end{split}
\end{equation}
Note that our functions  $\bar\delta,$ $\underline\delta$ are  the multi-dimensional generalizations  of the operator $\delta$ introduced by Candel \cite{Candel03}.

We are in the position to state  effective  integral estimates on the Lyapunov  exponents. 

 \begin{theorem}  \label{T:VA_smooth}  {\rm (Nguyen \cite[Theorem 3.12]{NguyenVietAnh17a}).}
 Let $(X,\Lc)$ be   a      compact  smooth  lamination by hyperbolic Riemann surfaces. 
 Let $\mu$ be a   harmonic  measure which is  ergodic.
Let
$\mathcal{A}:\ \Omega\times \R^+ \to  \GL(d,\K)      $ be  a smooth cocycle.  
 Let    
$$\chi_m<\chi_{m-1}<\cdots
<\chi_2<\chi_1$$
be the Lyapunov exponents of the cocycle $\mathcal A$ with respect  to $\mu,$ given by Theorem \ref{T:VA_general}.
Then  the following  inequalities hold $$
\underline\chi_{\max}\leq \chi_1\leq \bar\chi_{\max} \quad\text{and}\quad
\underline\chi_{\min} \leq \chi_m\leq \bar\chi_{\min} .$$ 
\end{theorem} 
 
 This theorem generalizes some  results of Candel \cite{Candel} and Deroin  \cite{Deroin05}  who  treat the case $d=1.$
 Under the assumption of   Theorem  \ref{T:VA_smooth}, the integrability  condition  \eqref{e:integrability}
   follows from some  well-known  estimates of the heat kernels of the Poincar\'e disc and  the fact that  the lamination is  
 compact and is  without singularities. In fact, we improve  the method of Candel in \cite{Candel03}.

 
 \section{Applications}
 \label{S:Applications}

 
 \subsection{Lyapunov exponent of a singular holomorphic foliation on a compact projective surface}
 \label{SS:Lyapunov}

We recall   the holonomy cocycle  of a singular holomorphic  foliation $\Fc=(X,\Lc,E)$ on a Hermitian  complex surface $(X,g).$
For each point $x\in X\setminus E,$ let $\Tan_x(X)$ (resp.  $\Tan_x(L_x)\subset \Tan_x(X)$) be    the tangent space  of $X$ (resp. $L_x$) at $x.$
 For every transversal    $\S$ at a point $x$ (that is, $\S$ is  complex submanifold of a flow box and $\S$ is transverse to every leaf of that flow box and $x\in \S$),
 let  $\Tan_x(\S)$ denote the tangent space of $\S$ at $x.$  

Now fix a point $x\in X\setminus E$ and a path $\omega\in\Omega_{x}$  and a time $t\in \R^+,$ and  let $y:=\omega(t).$
Fix  a transversal $\S_x$ at $x$ (resp. $\S_y$  at $y$) such that the complex line
$\Tan_{x}(\S_x)$ is the  orthogonal complement  of the complex line $\Tan_{x}(L_{x})$ in the Hermitian space $(\Tan_x(X),g(x))$
(resp.        $\Tan_{y}(\S_y)$ is the  orthogonal complement  of $\Tan_{y}(L_{y})$ in $(\Tan_{y}(X),g(y))$). 
Let  $\hol_{\omega,t}$  be the  holonomy map along the path  $\omega|_{[0,t]}$ from   an open  neighborhood of $x$ in $\S_x$   onto
 an open  neighborhood of $y$ in $\S_y,$  that is, let
 $$
 \hol_{\omega,t}:=\hol_\gamma,
 $$
 where  $\gamma:\ [0,1]\to  L_x$ is the  path  given by $\gamma(s):=\omega(ts)$ for $s\in[0,1]$  (see Definition \ref{D:holonomy_map}).
 
 The  derivative  $D \hol_{\omega,t}:\
\Tan_{x}(\S_x)\to \Tan_{y}(\S_y) $ induces 
the so-called {\it holonomy cocycle} $\mathcal H:\ \Omega\times\R^+\to\R^+$ given by  
\begin{equation}\label{e:hol_cocycle_dist}\mathcal H(\omega,t):=
  \|D \hol_{\omega,t}(x)\|.
  \end{equation} 
The last  map depends  only on the path  $\omega|_{[0,t]},$ in fact, it
depends only on the  homotopy class of  this path.
 In particular, it is  independent of   the choice  of transversals  $\S_x$ and $\S_y.$ 
 We see easily that
 $$
\mathcal H(\omega,t) =\lim_{z\to x,\ z\in \S_x} \dist(\hol_{\omega,t}(z), y)/ \dist(z,x).
  $$
 On the other hand, we note the following multiplicative property   which  is  an immediate consequence  of the  definition of $
\mathcal H(\omega,t),$ 
\begin{equation} \label{eq_multiplicative_cocycles}
 \mathcal{H}(\omega,t+s)   
=  \mathcal{H}(\omega,t)   \mathcal{H}(\sigma_t(\omega),s)  ,\qquad t,s\in\R^+, \omega\in\Omega,
\end{equation}
  where  $\sigma_t:\  \Omega\to\Omega$ is the   shift-transformation
    given by 
\eqref{e:shift}.

 The  holonomy cocycle (or equivalently, the normal derivative cocycle)  of a foliation  
is  a very important object  which  encodes
 dynamical  as well as  geometric and analytic informations    of  the  foliation.
 Exploring  this  object  allows  us  to understand  more  about the  foliation itself.
 
  The following fundamental  question  arises  naturally: 

\smallskip

\noindent {\bf Question. }{\it Can one  define  the  Lyapunov exponents of 
an ergodic harmonic  measure $\mu$ on a  compact singular holomorphic hyperbolic  foliation $\Fc=(X,\Lc,E)$ ?}

\smallskip

By  Theorem \ref{thm_harmonic_currents_vs_measures},  this  question can be rephrased  for  directed harmonic  currents on the foliation.
We have recently obtained  the  following  affirmative  answer  to this  question for generic  foliations in dimension two.

\begin{theorem}\label{T:VA}  {\rm  (\cite[Theorem 1.1]{NguyenVietAnh18b}).}
Let $\Fc=(X,\Lc,E)$ be  a  holomorphic  foliation  by Riemann surfaces defined on a    Hermitian compact complex projective surface $X$ satisfying the following  two conditions:

$\bullet$  its singularities  $E$ are all hyperbolic;

 $\bullet$ $\Fc$ is  Brody hyperbolic (see Definition \ref{D:uniform_hyperbolic_laminations}).
 
 Let $\mu$ be  a harmonic  measure  which does not give mass to any  invariant  analytic curve.
Assume, in addition, that $\mu$ is  ergodic. 

   Then    \begin{enumerate}
   \item $\mu$ admits
the (unique) {\bf Lyapunov exponent} $\chi(\mu)$  given by the  formula
\begin{equation}\label{e:Lyapunov_exp}
\chi(\mu):= \int_X\big(\int_{\Omega} \log \|\mathcal{H}(\omega,1)\| dW_x(\omega)\big)d\mu(x).  
\end{equation}

\item
  For  $\mu$-almost  every $x\in X\setminus E,$  we have
 $$
 \lim\limits_{t\to \infty} {1\over  t} \log  \| \mathcal H(\omega,t) \|=\chi(\mu)  
 $$
  for    almost every path  $\omega\in\Omega$ with respect to $W_x.$
  \end{enumerate}
\end{theorem}
 In fact, assertion (1) is a  consequence of  the so-called  integrability of the  holonomy cocycle of singular holomorphic foliations.
Assertion (2) says that  
the  characteristic  number $\chi(\mu)$    measures  heuristically   the exponential  rate of  convergence  of leaves toward each other  along  leafwise Brownian trajectories
 (see Candel  \cite{Candel03},  Deroin \cite{Deroin05} for the nonsingular case).   
 Therefore, Theorem \ref{T:VA} gives a  dynamical  characterization of $\chi(\mu).$
 
 
 \subsection{Sketchy proof of the  existence of Lyapunov exponent} \label{SS:Existence-Lyapunov}
 

 Prior to the  sketchy  proof of Theorem \ref{T:VA}, we discuss how to deal with the singularities.
 
   To study $\Fc$ near  a hyperbolic  singularity $a\in E,$ we use the following {\bf local model} introduced in \cite{DinhNguyenSibony14a}.
In this model, a neighborhood of $a$ is identified with the  bidisc $\D^2,$ and the restriction of $\Fc$  to $\D^2,$  i.e., the leaves of $(\D^2,\Lc,\{0\})$  coincide with 
the restriction to $\D^2$ of the   
 integral curves of a  vector field 
 \begin{equation}\label{e:Z}Z(z,w) = z {\partial\over \partial z}
+ \lambda w{\partial\over \partial w}\quad\text{with some}\quad \lambda\in\C\setminus \R. 
\end{equation}
 Notice that
if we flip $z$ and $w,$ we replace $\lambda$ by $\lambda^{-1}.$ Since $(\Im\lambda)( \Im\lambda^{-1})<0,$ we
may assume   below that the axes are chosen so that $\Im \lambda > 0.$

For $x=(z,w)\in \D^2\setminus\{0\}$, define the holomorphic map  $\psi_x:\C\rightarrow\C^2\setminus\{0\}$ 
\begin{equation}\label{e:leaf_equation}
\psi_x(\zeta):=\Big( ze^{i\zeta},we^{i\lambda\zeta}\Big)\quad \mbox{for}\quad \zeta\in\C.
\end{equation}
It is easy to see that $\psi_x(\C)$ is the integral curve of $Z$ which contains
$\psi_x(0)=x$.
Write $\zeta=u+iv$ with $u,v\in\R$. The domain $\Pi_x:=\psi_x^{-1}(\D^2)$ in $\C$ is defined by the inequalities
$$(\Im \lambda)u+(\Re\lambda)v >\log|w|\qquad \text{and}\qquad v > \log|z| .$$
So, $\Pi_x$ defines a sector $\S_x$ in $\C.$ It contains $0$ since $\psi_x(0)=x$.  
The leaf of $\Fc$ through $x$ contains  the  Riemann surface 
\begin{equation}\label{e:Riemann_surface_L_x}
\widehat L_x:=\psi_x(\Pi_x)\subset L_x.
\end{equation}
In particular, the leaves in a singular flow box are parametrized   using holomorphic maps $\psi_x:\Pi_x\to L_x.$

Now  we  fix  a  finite cover $\Uc$  of flow boxes on $X.$ 
We only consider flow boxes which are biholomorphic to $\D^2$. A {\it regular flow box} $\U$ is a flow box with  foliated chart $\Phi:\ \U\to \B\times\Sigma$ outside the singularities, where 
$\B$ and $\Sigma$ are open sets in $\C.$ 
 For each  $\alpha\in\Sigma,$ the  Riemann surface $V_\alpha:=\Phi^{-1}( \B\times\{\alpha\})$ is called a plaque  of $\U.$ 
{\it Singular flow boxes} are identified to their models $(\D^2,\Lc,\{0\})$ as 
described above. 
 For  $\U:=\D^2$ and  $s>0,$  let
 $s\U:= (s\D)^2.$
For each singular point $a\in E$, we fix a singular flow box $\U_a$ such that $2\U_a\cap 2\U_{a'}=\varnothing$ if $a,a'\in E$ with $a\not=a'$.
We also cover $X\setminus \cup_{a\in E} \U_a$ by a finite 
number of regular flow boxes $(\U_q)$ such that  each $\overline \U_q$ is contained in a larger regular flow box $\U'_q$ with $\U'_q\cap \cup_{a\in E} (1/2)\U_a =\varnothing.$ 
Thus  we obtain a  finite  cover $\Uc:=(\U_p)_{p\in I} $ of $X$  consisting  of  regular  flow boxes $(\U_p)_{p\in I\setminus E}$ and  singular ones $(\U_a)_{a\in E}.$  

Let $g_X$ be  a Hermitian  metric on $X$
 and let $\dist$ denote the  distance  on $X$ induced by $g_X.$
We often suppose without loss of generality that  the ambient metric $g_X$ coincides with the standard Euclidean metric on each singular flow box $2\U_a\simeq 2\D^2,$ $a\in E.$

 Now we  discuss  
the proof of Theorem \ref{T:VA}.   It  consists of two steps.  
In the  first step  we show that    Theorem \ref{T:VA}  follows  from
the new integrability condition  \eqref{e:necessary_integrability}.
\begin{equation}\label{e:necessary_integrability}
 \text{(new integrability  condition):}\qquad \int_X| \log \dist(x,E)| \cdot d\mu(x)<\infty.
 \end{equation}
 This new condition has  the advantage over  the old one \eqref{e:integrability}, since  the former does not involve the somewhat complicating  Wiener measures, and hence  it is
 easier to handle than the latter.

For  this purpose  we  study  the behavior of the holonomy cocycle
near the  singularities  with respect to  the leafwise Poincar\'e metric $g_P.$  
 The  following result gives an explicit  expression for  $\mathcal H$ near a singular point $a$ 
 using the above local model  $(\D^2,\Lc,\{0\}).$    
\begin{proposition}\label{P:holonomy} {\rm (\cite[Proposition 3.1]{NguyenVietAnh18b})} Let $\D^2$ be endowed with the  Euclidean metric.
For each  $x=(z,w)\in\D^2,$ consider the function  $\Phi_x:\  \Pi_x\to\R^+$   as follows.
For $\zeta\in \Pi_x,$ consider a path $\omega\in \Omega$ (it always exists since $\Pi_x$ is convex and $0\in\Pi_x$ as $x\in\D^2$) such  that
$$
\omega(t)= \psi_x(t\zeta)= (z e^{i \zeta t},we^{i\lambda\zeta t})\subset \D^2
$$
for all $t\in [0,1]$ (see \eqref{e:leaf_equation} above). 
Define $\Phi_x(\zeta):=  
\mathcal H(\omega,1).$
Then 
$$
\Phi_x(\zeta)=|e^{i \zeta}||e^{i \lambda\zeta}|{\sqrt{| z|^2+|\lambda w|^2}\over \sqrt{ | ze^{i\zeta}|^2+|\lambda we^{i \lambda\zeta}|^2           }}. 
$$
\end{proposition}
 
 Roughly speaking, Step 1
 quantifies  the  expansion speed of the hololomy  cocycle in terms
 of the ambient metric $g_X$ when  one travels  along   unit-speed geodesic rays.
 The first   main ingredient
 is Proposition \ref{P:holonomy}.
 The  second main ingredient of the  first  step 
  is a  detailed  analysis of  the  behaviour  of the leafwise Poincar\'e metric near hyperbolic  singularities
 which  had previously been  carried  out  in  \cite{DinhNguyenSibony12,DinhNguyenSibony14a,DinhNguyenSibony14b} and which culminates in Proposition \ref{P:Poincare}.

  The  second main step is  then devoted to  the proof of  inequality \eqref{e:necessary_integrability}.
  The main difficulty is that  known  estimates (see, for example,  \cite{DinhNguyenSibony12})
 on the behavior of $T$
near  linearizable  singularities,    only  give a  weaker  inequality
  \begin{equation} \label{e:known_estimate}
\int_X | \log \dist(x,E)|^{1-\delta} d\mu(x)<\infty, \qquad\forall \delta>0.
  \end{equation}
  So \eqref{e:necessary_integrability} is  the limiting case of  \eqref{e:known_estimate}.
  
  To prove \eqref{e:necessary_integrability}, by Theorem  \ref{thm_harmonic_currents_vs_measures} (1) write $\mu=\Phi(T)$
  for  a directed  positive  harmonic   
  current $T$ giving no mass to $\Par(\Fc).$   Moreover, such  a current  $T$ may be identified with a  directed  positive $\ddc$-closed
  current giving no mass to $\Par(\Fc)$  by   Proposition \ref{P:extension} and Definition \ref{D:Directed_hamonic_currents_with_sing}.
  
  The proof of \eqref{e:known_estimate}  relies on  the finiteness of the Lelong number of $T$  at
  every   point which has been   established in Proposition \ref{P:Skoda}. 
 Recall that Theorem  \ref{T:Lelong}   sharpens the last estimate
by showing that the Lelong number  of $T$ vanishes at every   hyperbolic singular point $x\in E.$  
Nevertheless, even this better estimate does not suffice to prove \eqref{e:necessary_integrability}. 
So  the main difficulties in the proof are (1) the information of $T$ near the singularities is
poor and (2) the Poincar\'e metric on leaves needs to be controlled near the singularities.

The situation near the singularities is not homogeneous.
  The new idea  in \cite{NguyenVietAnh18b} is  that we use   a cohomological argument
  which  exploits  fully  the assumption that $X$ is  projective.
  This  assumption  imposes a stronger mass-clustering  condition on harmonic  currents  than \eqref{e:known_estimate}.
  We need to divide
the space $X$ into small cells where one uses different techniques to get the desired estimates.
Some auxiliary quantities depending on $T$ are introduced and used to get a control, good
enough, near the singularities. The final result is obtained using both a careful analysis
on singular boxes and a global argument (roughly since the current sits in a projective
surface, it cannot have a transcendental behavior). Delicate  approximations are
needed to deduce the global estimates from the local ones.

Now  we  explain in more details  our proof of  the integrability condition  \eqref{e:necessary_integrability}.
Our approach is  based on a cohomological  invariance which  says roughly that
if two algebraic  curves $\Cf $ and $\Df$ on $X$ are cohomologous, that  is, they  are in the same cohomology class in $H^{1,1}(X)$  (for example,  if they have the same algebraic  degree  when $X=\P^2$), then  under  suitable  assumptions, we can define the wedge-product
$T\wedge[\Cf],$ $T\wedge [\Df]$   which are finite positive Borel measures and their masses are equal, i.e,
\begin{equation}\label{e:coho_inv_intro}
\int_X T\wedge[\Cf] =\int_X T\wedge [\Df].
\end{equation}
Before  going further,  let us   explain     why  equality \eqref{e:coho_inv_intro} could be  true.
Since $\Cf $ and $\Df$ on $X$ are cohomologous on $X$, the  $\partial\overline\partial$-lemma for compact K\"ahler manifolds
provides us an integrable  function  $u$ on $X$  such that
$$
[\Cf]-[\Df]=i \partial\overline\partial u\qquad \text{in the sense of currents.} 
$$
So  we can write
$$
\int_X T\wedge[\Cf] -\int_X T\wedge [\Df]=\int_X T\wedge i \partial\overline\partial u.
$$
The function $u$  is, in general, not smooth near $\Cf$ and $\Df.$ However, if we  could consider it like a smooth function,
Stokes' theorem would turn the  right hand side of the last line into the following integral
$$
\int_X u (i \partial\overline\partial T ) =0,
$$
where the  last equality holds since  the $\ddc$-closedness  of $T$ implies that $i \partial\overline\partial T=0.$
Therefore, we may expect equality \eqref{e:coho_inv_intro} to hold.

 Resuming  the sketchy proof of  the integrability condition  \eqref{e:necessary_integrability}, let $x_0\in E$ and  fix a  coordinate system $(z,w)$ around $x_0$ such that
the two separatrices of the hyperbolic singular point $x_0$  are $\{z=0\}$ and $\{w=0\}.$
Then we can show that the vanishing of the Lelong number of $T$ at  $0$   is  equivalent to 
the following convergence 
\begin{equation}\label{e:reduction_intro}
\int_{\B(0,r)}  T\wedge[z=r] \to  0 \qquad\text{as}\qquad r\to 0,
\end{equation}
where   $\B(0,r)$ is the ball in $X$ with center $x_0=0$ and radius $r.$
More importantly, 
  the integrability  condition \eqref{e:necessary_integrability}  is  somehow  equivalent to the statement that
the  convergence \eqref{e:reduction_intro}
has,  in a certain very  weak sense, a speed   of order $|\log{ r}|^{-\delta}$  as  $r\to 0$ for some $\delta>0.$
 Note, however, that this speed   does  not at all mean that $\int_{\B(0,r)}  T\wedge[z=r]=O(|\log{ r}|^{-\delta}).$  

Now suppose for  the sake of simplicity that $X=\P^2$ and $N\in\N$ is  large  enough. We  choose an algebraic curve $\Cf$ of degree $N$ which looks like the analytic curve $\{z=w^N\}$ near $0.$
 We also   choose an algebraic curve $\Df$ of degree $N$ which looks like the analytic curve $\{r=z-w^N\}$ near $0.$  The following  seven
  observations play  a key role in our approach, where  $0<\delta<1$ is   an exponent independent of $r$
and $N,$ $0<r<r_0$ with $r_0>0$ a fixed small number.

$(i)$   Outside  a  small ball  $\B(0,r_0),$
  the analytic curve $\{z=w^N\}$  (and hence  the algebraic curve $\Cf$) falls into a tubular neighborhood
with size $O(r^\rho)$  of
  the analytic curve $\{r=z-w^N\}$ (and hence the algebraic curve $\Df$), where $\rho$ is  a real number depending on $N$
 with $0<\rho\leq 1.$ So we may expect
  $$
  \int_{X\setminus\B(0,r_0)}  T\wedge[\Cf] =\int_{X\setminus\B(0,r_0)} T\wedge [\Df]+O(r^\rho).
  $$
 
$(ii)$ Outside 
  the ball   $\B(0,r^{1/N}|\log{r}|^{3/N})$ and inside the small ball  $\B(0,r_0),$  
 the   analytic  curve $\{r=z-w^N\}$ (and hence the  algebraic curve $\Df$)  behaves like  the analytic curve $\{z=w^N\}$  (and hence the  algebraic curve $\Cf$) while intersecting the two curves
 with a general leaf.  Indeed, when $|w|\geq r^{1/N}|\log{r}|^{3/N},$  
 we have $r\ll |w|^N.$ So we may expect
  $$
  \int_{ \B(0,r_0)\setminus \B(0,r^{1/N}|\log{r}|^{3/N})} T\wedge [\Cf]=
\int_{ \B(0,r_0)\setminus \B(0,r^{1/N}|\log{r}|^{3/N})} T\wedge [\Df]
+O(|\log r|^{-\delta}).
  $$

 $(iii)$ The corona
     $\A_{r,N}:=\B(0,r^{1/N}|\log{r}|^{3/N})\setminus \B(0,r^{1/N}|\log{r}|^{-3/N}) $
     is,  in some  sense,   small and it may be considered as  negligible.
 So we may expect
$$
  \int_{ \A_{r,N}} T\wedge [\Cf]=O((\log r)^{-\delta})\quad\text{and}\quad 
\int_{  \A_{r,N} } T\wedge [\Df]
=O(|\log r|^{-\delta}).
  $$
  
 $(iv)$ Our next  observation is   the following  partition of $X$  for $0<r\ll 1:$
\begin{equation*}
X=\big(X\setminus\B(0,r_0)\big)\coprod \big(\B(0,r_0)\setminus   \B(0,r^{1/N}|\log r|^{3/N} )\big)\coprod \A_{r,N}\coprod  
\B(0,r^{1/N}|\log r|^{-3/N}).
\end{equation*}
This  allows us to decompose both integrals of \eqref{e:coho_inv_intro}
 into corresponding  pieces. 
 
Consequently, when the degree $N$ is  sufficiently high, by taking into account the   observations $(i)$-$(ii)$-$(iii)$-$(iv)$,
 and using  \eqref{e:coho_inv_intro},  we  see that
 $$
  \int_{\B(0,r^{1/N}|\log{r}|^{-3/N}) } T\wedge [\Cf]-
\int_{ \B(0,r^{1/N}|\log{r}|^{-3/N})  } T\wedge [\Df]
=O(|\log r|^{-\delta}).
  $$

$(v)$ Inside 
  the ball   $\B(0,r^{1/N}|\log{r}|^{-3/N}),$ the analytic  curve  $\{z=w^N\}$ (and hence the algebraic curve  $\Cf$)    clusters around $0,$ in a  certain sense,  much more often 
than the analytic  curve  $\{z=r\}$ (and hence the algebraic curve  $\Df$). Indeed,     we see   in the equation $z=w^N$ that both $z$ and $w$ can tend to $0,$  whereas in the  equation $z=r,$
only $w$ could tend to $0.$ So we may expect that  in a certain sense,  
$$
     \int_{\B(0,r^{1/N}|\log{r}|^{-3/N}) } T\wedge [\Df] \ll
\int_{ \B(0,r^{1/N}|\log{r}|^{-3/N})  } T\wedge [\Cf]
 .
  $$
 This, combined  with the  estimate  obtained just at the end of  $(iv)$,  implies that both integrals
 $$
 \int_{\B(0,r^{1/N}|\log{r}|^{-3/N}) } T\wedge [\Cf] \quad\text{and}\quad
\int_{ \B(0,r^{1/N}|\log{r}|^{-3/N})  } T\wedge [\Df]
 $$
 admit, in a certain sense, a speed   of order $|\log{ r}|^{-\delta}.$
   
$(vi)$ Inside 
  the ball   $\B(0,r^{1/N}|\log{r}|^{-3/N}),$ 
 the   analytic  curve $\{r=z-w^N\}$ (and hence the  algebraic curve $\Df$)  behaves like  the analytic curve $\{z=r\}$ while intersecting the two curves
 with a general leaf. Indeed, when $|w|\leq r^{1/N}|\log{r}|^{-3/N},$  
 we have $|w|^N\ll r.$
 So we may expect
$$
  \int_{\B(0,r^{1/N}|\log{r}|^{-3/N}) } T\wedge [\Df]-
\int_{ \B(0,r^{1/N}|\log{r}|^{-3/N})  } T\wedge [z=r]
=O(|\log r|^{-\delta}).
  $$
 This, together with the  estimate just obtained at the end of  $(v),$ 
yields  that
 $$
  \int_{ \B(0,r^{1/N}|\log{r}|^{-3/N})  } T\wedge [z=r]
  $$
has, in a certain sense, a speed   of order $|\log{ r}|^{-\delta}.$ 
 
$(vii)$ Our last observation is that one can show  that there is  a constant $c_N>1$ independent of $r$  such that
$$ c^{-1}_N \int_{\B(0,r^{1/N})}  T\wedge[z=r]\leq \int_{\B(0,r)}  T\wedge[z=r]\leq  c_N \int_{\B(0,r^{1/N})}  T\wedge[z=r].$$
 This, together with the  estimate just obtained at the end of  $(vi),$ 
implies  that
 $$
  \int_{ \B(0,r )  } T\wedge [z=r]
  $$
 admits, in a certain sense, a speed   of order $|\log{ r}|^{-\delta}.$
 Hence, we get the convergence with speed   \eqref{e:reduction_intro}. This is  what  we  are looking for.

In fact, the factor  $|\log{r}|^{3/N}$ appearing in the above  observations  comes  from the degeneration of the 
Poincar\'e metric $g_P$ relative to the ambient metric $g_X$  (see formula \eqref{e:eta} and Proposition \ref{P:Poincare}).
Moreover, the larger  the degree  $N$ is,
the more  evident  the  mass-clustering phenomenon in the previous observation  becomes.  

Our approach underlines several  tasks.  On the one hand, we need to define a geometric intersection
of a directed positive harmonic  current with a  singular analytic curve defined on a neighborhood of a singular point of the foliation.
On the other hand, we need to approximate  some (local) analytic curves by global algebraic ones.   
The assumption of projectivity  of $X$ is needed in order to ensure a good supply of algebraic  curves.
  
\begin{remark}\rm The  condition  of Brody hyperbolicity  seems to be indispensable  for the   integrability of the  holonomy cocycle. Indeed,
 Hussenot  \cite[Theorem A]{Hussenot} finds out    the  following  remarkable property
 for  a class of  Ricatti foliations $\Fc$ on $\P^2.$
For every $x\in\P^2$ outside invariant curves of every foliation in this class, it holds that
   $$
 \limsup\limits_{t\to \infty} {1\over  t} \log  \| \mathcal A(\omega,t) \|= \infty 
 $$ for almost every path $\omega\in\Omega_x$ with respect to the Wiener measure at $x$ which lives
 on the leaf passing through $x.$ By     Theorem  \ref{T:generic}, these foliations  are hyperbolic since all their  singular points    have nondegenerate linear part.
 Nevertheless, neither of them is Brody hyperbolic because  they all contain integral curves which are some images of $\P^1$  (see  Remark  \ref{R:Brody_sufficiency}). 
   \end{remark}
   
   \begin{remark}\label{R:projective}
    \rm If  one can prove  the   new integrability condition  \eqref{e:necessary_integrability}  without  using that $X$  is  projective,
    then  Theorem \ref{T:VA}   will hold  without this assumption.
   \end{remark}

   \begin{remark}\rm There is  some growing interest in the study  of 
    Lyapunov exponents for surface group representations  (see \cite{DeroinDujardin}  and the references therein). 
   \end{remark}

  \begin{problem}\rm 
Does  Theorem \ref{T:VA}   still hold if 
the ambient compact projective manifold $X$ is of dimension $>2$ ? 
  \end{problem}

  \begin{problem}\rm 
Does   Theorem \ref{T:VA}   still hold if 
the ambient compact  surface $X$ is only  K\"ahler ?  If this is true,  then  it is  a  good question to investigate  the general case  of ambient compact  K\"ahler manifolds of  dimension $>2.$
  \end{problem}

  \begin{problem}\rm 
Does   Theorem \ref{T:VA}   still hold if the  singularities of $\Fc$ are  merely linearizable ?  
  \end{problem}
 \subsection{Negativity and  cohomological formulas of Lyapunov exponent}
 \label{SS:Lyapunov_bis}

 Suppose now  that $\Fc=(X,\Lc,E)$    is  
a   singular  holomorphic foliation   with   only hyperbolic singularities  in a    compact projective surface $X$ 
such  that $\Fc$ admits no directed positive closed current.
So the assumptions of both Theorem  \ref{T:VA} and \ref{T:Unique_ergodicity_1}  are fulfilled.
 By  Theorem \ref{T:Unique_ergodicity_1}, let $T$ be  the unique  directed positive  $\ddc$-closed   current $T$ whose the Poincar\'e mass  is  equal to $1.$ 
 So  by  Theorems  \ref{thm_harmonic_currents_vs_measures},  the  measure $\mu=\Phi(T)$ given by Definition \ref{D:Phi} is 
 the unique  probability  harmonic measure of $\Fc.$ 
 \begin{definition}\label{D:Lyapunov}\rm
 The {\it Lyapunov  exponent}  of  the foliation $\Fc,$  denoted by $\chi(\Fc),$  is by  definition, the  real number  $\chi(\mu)$ given by  Theorem \ref{T:VA}.
 \end{definition}
 
 When  we  explore the dynamical system  associated to a foliation $\Fc,$  the sign of its Lyapunov   exponent is a crucial   information.
 Indeed, the positivity/negativity of $\chi(\Fc)$ corresponds to the   repelling/attracting   character of a typical leaf  along a typical Brownian trajectory.    
 Here is  our   second main result. 

\begin{theorem}\label{T:Negative_exponent}  {\rm  (\cite[Theorem B]{NguyenVietAnh18c}) } 
 Let $\Fc=(X,\Lc,E)$    be  
a   singular  holomorphic foliation   with   only hyperbolic singularities  in a    compact projective surface $X$ 
such  that $\Fc$ admits no directed positive closed current.
Then  $\chi(\Fc)$ is a negative  nonzero real number.
\end{theorem}

Roughly speaking,  Theorem \ref{T:Negative_exponent}  says that in the sense of ergodic theory, generic leaves have  the  tendancy to wrap together
towards the support of the unique probability  harmonic measure.
 
Recall from  Subsection \ref{SS:Local-theory} that the foliation $\Fc$ is given by  an open covering $\{\U_j\}$ of $X$ and  holomorphic vector fields $v_j\in H^0(\U_j,\Tan(X))$ 
with isolated singularities (i.e. isolated zeroes) such that
$$
v_j=g_{jk}v_k\qquad\text{on}\qquad \U_j\cap \U_k
$$
for some nonvanishing holomorphic functions $g_{jk}\in H^0(\U_j\cap \U_k, \Oc^*_X).$

The  functions $g_{jk}$ form a multiplicative cocycle and hence give a  cohomology class in $H^1(X,\Oc^*_X),$ that is  a holomorphic line  bundle on $X.$
This is  the {\it cotangent bundle }  $\Cotan(\Fc)$ of $\Fc,$  Its  dual $\Tan(\Fc),$ represented by the inverse  cocycle $\{g_{jk}^{-1}\}$ is called the {\it tangent  bundle } of $\Fc.$    
 For a complex  line  bundle $\E$ over $X,$ let  $c_1(\E)$ denote the cohomology Chern class of $\E.$
 This is an element  in $H^{1,1}(X).$


 The next  result gives cohomological formulas for $\chi(\mu)$ and $\|\mu\|$ in terms of the geometric quantity $T$ and some   characteristic classes
 of $\Fc.$
\begin{theorem}\label{T:Coho_exponent} {\rm  (\cite[Theorem A]{NguyenVietAnh18c}) }
Under the  assumption of Theorem  \ref{T:VA},
the following  identities hold
   \begin{subequations}\label{e:thm_A}
     \begin{alignat}{1}
\chi(\mu)&=  - c_1(\Nor(\Fc))\smile\{T\},\label{e:coho_exponent}\\ 
 \|\mu\|&= c_1(\Cotan(\Fc))\smile\{T\}\label{e:coho_mass}.
  \end{alignat}
\end{subequations}
Here  $\Nor(\Fc):=\Tan(X)/\Tan(\Fc)$ 
stands for the normal bundle 
of $\Fc,$  where $\Tan(X)$ (resp. $\Tan(\Fc)$ and $\Cotan(\Fc)$) is  as  usual the  tangent bundle of $X$ (resp. the tangent bundle  and the  cotangent bundle of $\Fc$).
\end{theorem}

 Now  we apply   the above results to  the family $\Fc_d(\P^2)$ of  singular holomorphic  foliations  on $\P^2$ with  a given degree $d>1,$
 which was previously  introduced in  Theorem  \ref{T:generic}.
Consequently, Theorem 
\ref{T:Coho_exponent}   give us the following result.

\begin{corollary}\label{C:Lyapunov_P2}
 Let $\Fc=(\P^2,\Lc,E)$    be  
a   singular  foliation by curves on the complex projective plane $\P^2.$ Assume that
all the singularities  are hyperbolic and  that  $\Fc$ has no invariant algebraic  curve. 
 Then  
 \begin{equation}\label{e:univ_exponent}\chi(\Fc)=-{d+2\over d-1}.
 \end{equation}
 
\end{corollary}  

 So for  a generic  foliation $\Fc$ of a given degree $d>1$ in $\P^2,$  we have $\chi(\Fc)=-{d+2\over d-1}.$

 \begin{remark}
  \rm  Theorem \ref{T:Negative_exponent} gives a complete answer to  Problem 7.7 in  our previous survey \cite{NguyenVietAnh18d} (see also Hussenot \cite{Hussenot}).
 \end{remark}

\begin{remark} \rm It is  of interest  to  know  whether  Theorems  \ref{T:Negative_exponent} and \ref{T:Coho_exponent} still hold if $X$ is merely a compact K\"ahler surface. This is definitely the case
 if  we can relax  the projectivity assumption in 
Theorem \ref{T:VA}, see Remark \ref{R:projective}. 
 The reader may find in  \cite{DinhSibony18b,NguyenVietAnh18c}  some other open questions in the ergodic theory of singular holomorphic foliations.
\end{remark}

\begin{problem}  \rm Study  the  sign of  Lyapunov  exponents for  a singular holomorphic  foliation on a compact projective  manifold of dimension $k>2$ (if possible,   more generally
on a compact K\"ahler  manifolds),  when   the  singularities  of the foliations are all  hyperbolic linearizable.
Can  we relax  the assumption that  the singularities are hyperbolic linearizable ?
\end{problem}

\begin{problem}  \rm Study  the existence  and  the  sign of  Lyapunov  exponents for  a singular holomorphic  foliation on a compact projective  manifold of dimension $k>2$ (if possible,   more generally
on a compact K\"ahler  manifolds),  when   the  singularities   are of dimension $\geq 1.$
\end{problem}

 \subsection{Normal  bundle, curvature  density   and Lyapunov exponent}
 \label{SS:Normal-bundle}

In this subsection we provide  preparatory results  needed  the proof of  formula \eqref{e:Lyapunov_exp}.
This  preparation gives  a  bridge  between  the  geometric aspect (normal bundle, curvature density) and the dynamical aspect (Lyapunov exponent) of a holomorphic  foliation  
emphasizing   the role of the  singularities.

Let $(L , h)$ be a singular Hermitian holomorphic line bundle on $X.$
If $e_L$ is a holomorphic frame of $L$ on some open set $U \subset X,$
then the function $\varphi $ defined  by  $|e_L |^2_h = \exp{( -2\varphi)} $  is called the {\it local weight} of the metric $h$ with respect to $e_L .$ 
If   the local weights $\varphi$ are in $L_\loc^1(U),$  then {\it (Chern) curvature current} of $(L,h)$ denoted by $c_1 (L , h)$ 
is  given by $c_1 (L , h)|_U = \ddc \varphi.$  This is   a $(1,1)$-closed current. Its class in $H^{1,1}(X)$ is  called the Chern class of $L.$
If  we fix a  smooth  Hermitian metric $h_0$ on $L,$ then every  singular metric $h$ on $L$ can be written $h=e^{-2\varphi} h_0$ for some  function $\varphi.$
We say that $\varphi$ is the {\it global weight} of $h$ with respect to $h_0.$ Clearly,  $c_1(L,h)=c_1(L,h_0)+\ddc\varphi.$

 Consider the normal bundle $\Nor(\Fc) = \Tan(X)/\Tan(\Fc)$ of $\Fc.$
 For $x\in X\setminus E$ and a  vector $u_x\in \Tan_x(X),$ let 
 $[u_x]$ denotes its class in $\Nor_x(\Fc).$
 We also identify
 $[u_x]$ with the  set $u_x +\Tan_x(\Fc)\subset \Tan_x(X)$
  Note that   a (local) smooth  section  $w$ of  $\Nor(\Fc)$ can be locally written as  $w_x=[u_x]$ for some
  smooth vector field $u.$
 
 Consider the  following  metric 
$g^\perp_X$ on the normal bundle $\Nor(\Fc):$  
 \begin{equation}\label{e:tran_metric}
  \|w_x\|_{g^\perp_X}:=\min_{u_x\in [w_x]} \|u_x\|_{g_X},\qquad\text{for}\qquad  w_x\in \Nor(\Fc)_x,\ x\in X\setminus E.
 \end{equation}
 Note  that $u_x$ achieving  the minimum in  \eqref{e:tran_metric} is uniquely determined by $[w_x].$
 $g^\perp_X$ is  called   {\it  transversal metric} associated with $\Fc$ and  the ambient metric $g_X.$

 Fix a  smooth Hermitian metric $g_0$ on $\Nor(\Fc).$ There is  a  global weight  function $\varphi$ on $X$  such that
 $  g^\perp_X=e^{-2\varphi} g_0. $

 Next, 
 we recall some notions and results from \cite[Section 9.1]{NguyenVietAnh17a}.
 Fix a  point $x\in X$   and let
   $\phi_{x}:\ \D\to L=L_{x}$   be the universal covering map    given in (\ref{e:covering_map}).    
    Consider  
      the  function  $\kappa_x:\  \D\to \R$ defined by
\begin{equation}\label{e:specialization}
\kappa_x(\zeta):= \log \| {\mathcal H}(\phi_x\circ\omega,1) \|,\qquad  \zeta \in \D,
\end{equation}
where  $\omega\in \Omega_0$ is any path  such that  $\omega(1)=\zeta.$ 
This   function is  well-defined   because $\mathcal H(\phi_x\circ \omega,t)$ depends only on the homotopy class of the path $\omega|_{[0,t]}$ and $\D$ is   simply connected.
Following \cite{NguyenVietAnh17a},
  $\kappa_x$ is  said  to be  the {\it specialization}  of the holonomy cocycle  $\mathcal H$ at $x.$

 The  following  two conversion rules  for changing  specializations in the  same leaf   are useful (see  \cite {NguyenVietAnh17a}). 
For this  purpose  let $y\in L_x$ and pick $\xi\in \phi_x^{-1}(y).$
Since  the holonomy cocycle is  multiplicative (see \eqref{eq_multiplicative_cocycles}), the  first  conversion rule (see \cite[identity (9.6)] {NguyenVietAnh17a}) states that 
\begin{equation}\label{e:change_spec}
\kappa_y\big({\zeta-\xi\over  1-\zeta\bar\xi}\big)=\kappa_x(\zeta)-\kappa_x(\xi),\qquad \zeta\in\D.
\end{equation}
Consequently, since $\Delta_P$ is invariant  with respect to the automorphisms of $\D,$  it follows that
\begin{equation}\label{e:change_spec_bis}
\Delta_P \kappa_y(0)=\Delta_P\kappa_x(\xi).
\end{equation}
  
By   \cite[identities (9.5) and  (9.8)]{NguyenVietAnh17a}, we have that
\begin{equation}\label{e:varphi_n_diffusion}
\kappa_x(0)=0  \quad\textrm{and}\quad  \E_x[\log {\|\mathcal H(\bullet,t)\|    }  ] =  
 (D_t \kappa_x)(0),\qquad t\in\R^+,
\end{equation}
where  $(D_t)_{t\in\R^+}$ is the family of  diffusion  operators associated with  $(\D,g_P).$

Consider the function $\kappa:\ X\setminus E\to\R$  defined by
\begin{equation}\label{e:kappa}
\kappa(x):= (\Delta_P\kappa_x)(0)\quad\text{for}\quad x\in X\setminus E.
\end{equation}

Now let $L= \Nor(\Fc).$
Suppose that $L$ is  trivial over a flow box $\U\simeq \B\times \Sigma,$ i.e.
$L|_\U\simeq \U\times\C\simeq \B\times\Sigma\times\C.$
Consider the holomorphic section $e_L$ on $\U$ defined by $e_L(x):=(x,1).$
Let $\varphi$ be the local weight of $(L,g^\perp_X)$ with respect to $e_L.$
 Equality \eqref{e:hol_cocycle_dist} is  rewritten as follows
 \begin{equation*}
 {\|e_L(y)\|_{g^\perp_X}\over \|e_L(x)\|_{g^\perp_X}    }  ={e^{-\varphi(y)}\over  e^{-\varphi(x)}} ,
  \end{equation*}
where $x$  and $y$ are  on the same plaque  in the    flow box $\U\simeq \B\times \Sigma.$
Consequently,
\begin{equation}\label{e:hol_cocycle_dist_bis}
\ddcy \log\|\mathcal H(\omega,t)\||_{L_x}=\Big(\ddcy \log\big({\|e_L(y)\|_{g^\perp_X}\over \|e_L(x)\|_{g^\perp_X}    }\big)\Big)|_{L_x}= -\ddcy \varphi(y)|_{L_x}.
\end{equation}
Combining   \eqref{e:specialization},  
   \eqref{e:change_spec}, \eqref{e:change_spec_bis}, \eqref{e:kappa} and \eqref{e:hol_cocycle_dist_bis}, we obtain that
 \begin{equation}\label{e:kappa-varphi}
 \kappa(x)=-(\Delta_P\varphi)(x)\qquad\text{for}\qquad x\in X\setminus E.
 \end{equation}
 The function $\kappa$ is  said  to be  the {\it  curvature density} of $\Nor(\Fc).$

 For $x=(z,w)\in\C^2,$ let
 $\|x\|:=\sqrt{|z|^2+|w|^2}$ be the  standard Euclidean norm of $x.$ Recall  that
 $\lof(\cdot):=1+|\log(\cdot)|.$  
 Using  Propositions \ref{P:holonomy} and \ref{P:Poincare},  we obtain the  following result which gives  precise variations up to order $2$ of   $\mathcal H$ near a singular point $a$ 
 using the local model  $(\D^2,\Lc,\{0\})$   discussed  in \eqref{e:Z}. 
\begin{proposition}\label{P:laplacian_hol}   {\rm  (\cite[Lemma 3.2]{NguyenVietAnh19})}
 Let $\D^2$ be  endowed with the Euclidean metric. Then
there is a constant $c>1$  such that
for every  $x=(z,w)\in({1\over2}\D)^2,$  we have that  
\begin{eqnarray*}
c^{-1}    \lof {\|(z,w)\|}&\leq& |d \kappa_x(0)|_P\leq c\lof {\|(z,w)\|},\\
-c  {|z|^2|w|^2\over (|z|^2+|w|^2)^2}  (\lof {\|(z,w)\|})^2&\leq& \Delta_P \kappa_x(0)\leq -c^{-1} {|z|^2|w|^2\over (|z|^2+|w|^2)^2}  (\lof {\|(z,w)\|})^2,
\end{eqnarray*}
where the function $\kappa_x$ is defined in \eqref{e:specialization}.
\end{proposition}

 Following  Proposition  \ref{P:laplacian_hol}, consider  a weight function $W:\ X\setminus E\to \R^+$ as follows. Let $x\in X\setminus E.$
 If $x$ belongs to  a regular flow box then $W(x):=1.$
 Otherwise, if  $x=(z,w)$ belongs to a singular flow box $\U_a,$ $a\in E,$  which is  identified with the local model with coordinates $(z,w),$ then
 \begin{equation}\label{e:W}
  W(x):=   \lof{\|(z,w)\|}+{|z|^2|w|^2\over (|z|^2+|w|^2)^2}  (\lof {\|(z,w)\|})^2.
 \end{equation}
Note that $$1\leq \lof \dist(x,E) \leq W(x)\leq 2 (\lof \dist(x,E))^2.$$


Inspired by  Definition 8.3  in Candel \cite{Candel03},  we have the  following
\begin{definition}\label{D:moderate_function} \rm A real-valued function $h$ defined on   $\D$
 is called    {\it weakly moderate} 
  if there is a   constant $c>0$ such that
$$
\log |h(\xi)-h(0)|\leq  c\,\dist_P(\xi,0)+c,\qquad  \xi\in \D.
$$  
  \end{definition}
  \begin{remark}
   \rm
   The notion of  weak moderateness is  weaker than the notion of moderateness
   given   in \cite{Candel03,NguyenVietAnh17b}.
  \end{remark}

 The usefulness of weakly moderate functions is illustrated by 
 the following Dynkin type formula.
\begin{lemma}\label{L:D_t_Delta} 
Let $f\in\Cc^2(\D)$  be such that  $f,$ $|df|_P$ (see the  definition after \eqref{e:heat_kernel}) and $\Delta_P f$  (see \eqref{e:Laplacian_disc})  are weakly moderate functions.   Then 
$$
(D_tf)(0) - f(0)=\int_0^t (D_s \Delta_P  f) (0) ds,\qquad t\in\R^+.
$$ 
\end{lemma}

Using  Propositions \ref{P:holonomy}, \ref{P:laplacian_hol}   and \ref{P:Poincare} as well as  Lemma \ref{L:D_t_Delta}, we 
 obtain  an estimate on  the  expansion rate up to order $2$ of  $\mathcal H(\omega,\cdot)$ in terms of   $\dist_P(\cdot,0)$
  and  the  distance  $\dist(x,E).$  
\begin{proposition}\label{P:expansion_rate}
  There  is a  constant $c>0$ such that for every $x\in X\setminus E$ and every
  $ \xi\in\D,$
 \begin{eqnarray*}  \big |  \kappa_x(\xi) -\kappa_x(0)   \big |
 &\leq & c \lof \dist(x,E)\ \cdot\  \exp{\big (c\, \dist_P(\xi, 0)\big)}, \\
  \big |  |d\kappa_x(\xi) |_P-|d\kappa_x(0)|_P   \big |
& \leq & c \lof \dist(x,E)\ \cdot\  \exp{\big (c\, \dist_P(\xi,0)\big)}, \\
  \big |  \Delta_P\kappa_x(\xi )-\Delta_P\kappa_x(0)   \big |
 &\leq & c (\lof \dist(x,E))^2\ \cdot\  \exp{\big (c\, \dist_P(\xi, 0)\big)}.
\end{eqnarray*}  
\end{proposition}

 We  keep the hypotheses and notation of Theorem \ref{T:VA}.
 For $t\in\R^+,$ consider  the function $F_t:\  X\setminus E \to\R$ defined by
 \begin{equation}\label{e:F}
  F_t(x):= \int_{\Omega} \log \|\mathcal{H}(\omega,t)\| dW_x(\omega)  \qquad\text{for}\qquad x\in X\setminus E.
 \end{equation}
By  \eqref{e:Lyapunov_exp}
 the  Lyapunov exponent $\chi(\mu)$ can be rewritten as
\begin{equation}\label{e:Lyapunov_exp_bis}
\chi(\mu):= \int_X F_1(x)d\mu(x)={1\over t}\int_X F_t(x)d\mu(x).
\end{equation} 

By Proposition 3.3 and Lemma 4.1 in \cite{NguyenVietAnh18b}, we infer that
 there is a constant $c>0$ such that
  \begin{equation}\label{e:growth_F}
   |F_1(x)|\leq  c  \lof \dist(x,E)\qquad \text{for}\qquad  x\in X\setminus E.
 \end{equation}

 \subsection{Sketchy proof of the cohomological formula  of Lyapunov exponent}
 \label{SS:Lyapunov_bis_bis}

 We only  give  here the proof of  formula \eqref{e:Lyapunov_exp} under the  assumption that  the ambient metric $g_X$ on $X$
 is  equal to the Euclidean metric in a local model near  every  singular point of $\Fc.$  For the proof of  the general case,  see  \cite{NguyenVietAnh18c}.

The  following  result  relates the Lyapunov exponent $\chi(T)$ to the function $\kappa$ defined in \eqref{e:kappa}. It plays  the key role in the proof of  formula \eqref{e:Lyapunov_exp}. 
\begin{proposition}\label{P:Lyapunov_vs_kappa}{\rm (\cite[Proposition 4.6]{NguyenVietAnh18c}) }
 Under the  hypotheses and notations  of Theorem \ref{T:VA},
  the integrals
 $\int_X|\kappa(x)|d\mu(x)$  and $\int_X W(x)d\mu(x)$  are  bounded,  and the following identity holds
 $$
 \chi(\mu)=\int_X\kappa(x)d\mu(x).
 $$ 
\end{proposition}

\begin{remark}\label{e:crucial-point}
 \rm
 In fact,  the $\mu$-integrability of the weight $W$  in  Proposition \ref{P:Lyapunov_vs_kappa} implies
    that of the curvature density $\kappa.$ 
 This  is a
 crucial point of the proof of  formula \eqref{e:Lyapunov_exp}.
 The novelty of this  proposition  is  that the weight $W(x)$ behaves like $(\lof \dist(x,E))^2$   when  $x=(z,w)$ satisfies $|z|\approx|w|,$  whereas      
 the new integrability condition \eqref{e:necessary_integrability} only provides the $\mu$-integrability  of  the less  singular  weight $\lof \dist(x,E).$
 However, for the sake of simplicity, we do not give  a full argument of this important point  in  the sketchy proof below.
\end{remark}

\proof[Sketchy proof]   
By Proposition \ref{P:expansion_rate}, $\kappa_x,$
$|d\kappa_x|_P$ and $\Delta_P\kappa_x$ are weakly moderate functions  on $\D.$
Consequently, applying Lemma \ref{L:D_t_Delta} yields that
\begin{equation}\label{e:D_t_Delta}
(D_1\kappa_x)(0) - \kappa_x(0)=\int_0^1 (D_s (\Delta_P \kappa_x) ) (0) ds.
\end{equation}
By \eqref{e:varphi_n_diffusion}  and \eqref{e:F}, the left-hand side of \eqref{e:D_t_Delta} is 
 equal to
$$\E_x[\log \mathcal H(\omega,1)]=F_1(x),$$  
which is  finite   because of \eqref{e:growth_F}.
On the other hand, by \eqref{e:change_spec} and \eqref{e:kappa}, the right-hand side  of \eqref{e:D_t_Delta} can be rewritten as
$$
\int_0^1 (D_s  \kappa) (x) ds.
$$
Consequently, integrating  both sides of   \eqref{e:D_t_Delta}  with respect to  $\mu,$ we get that
\begin{equation}\label{e:F_t_vs_kappa}
\int_X  F_1(x)d\mu(x)=\int_X\big(\int_0^1 (D_s \kappa)(x)  ds\big)d\mu(x).
\end{equation}
Since  we know by  \eqref{e:Lyapunov_exp_bis}, \eqref{e:growth_F} and \eqref{e:integrability}  that the left intergral is  
bounded and is  equal to $\chi(T),$ it follows that  right-side  double integral is also bounded. 
 
Next, we make a full  use  of   the  above boundedness and the  upper-bound of $\kappa(x)$ for $x$  close to the singularities (see Proposition \ref{P:laplacian_hol}).
Consequently, we can show that $\int_X|\kappa(x)|d\mu(x)<\infty$ and  $\int_X W(x)d\mu(x)<\infty.$  We do not  give here  the full explanation for the last  delicate argument, but refer the reader
to  \cite[Proposition 4.6]{NguyenVietAnh18c} for more details.

Hence,  we infer from this  and the fact that $\mu$ is weakly harmonic    that
$$\int_X\big(\int_0^1 (D_s \kappa)(x)  ds\big)d\mu(x)=\int_0^1\big(\int_X (D_s \kappa)(x) d\mu(x)\big)   ds= \int_0^1\big(\int_X  \kappa(x)d\mu(x)\big)ds=\int_X  \kappa(x)d\mu(x).$$
This, combined with \eqref{e:F_t_vs_kappa}
and  \ref{e:Lyapunov_exp_bis}, implies $$
 \chi(\mu)=\int_X  F_1(x)d\mu(x)     =\int_X\kappa(x)d\mu(x).
 $$
\endproof

\begin{corollary}\label{C:Lyapunov_vs_kappa}
 Under the  hypotheses and notations  of Theorem \ref{T:VA},   the following identities hold
 $$
  \kappa g_P=-c_1(\Nor(\Fc),g^\perp_X)\quad\text{on}\quad X\setminus E,\qquad\text{and}\qquad \int_X\kappa(x)d\mu(x)=-\int_X c_1(\Nor(\Fc),g^\perp_X)\wedge T.
 $$
\end{corollary}
\proof
 By \eqref{e:kappa-varphi}  we obtain that
 $$
 \kappa g_P=-(\Delta_P\varphi) g_P= -\ddc \varphi =-c_1(\Nor(\Fc),g^\perp_X)\qquad\text{on}\qquad \U.
 $$
 The first identity follows.  
Since 
 $\int_X|\kappa(x)|d\mu(x)<\infty$ by Proposition \ref{P:Lyapunov_vs_kappa},
  Integrating both sides of the first identity  over $X\setminus  E$ gives the second identity.  
\endproof

\proof[Proof of the first identity of   Theorem \ref{T:Coho_exponent} in a special case]
We  only consider  the special case where the ambient metric $g_X$ 
is  equal to the Euclidean metric in a local model near  every  singular point of $\Fc.$
Fix  a  smooth Hermitian metric $g_0$ on the normal bundle $\Nor(\Fc)$ of $\Fc.$
So  there is  a global weight function $f:\ X\to [-\infty,\infty)$ such that $g_X^\perp=g_0\exp{(-2f)}.$
We know that the weight function $f$ is   smooth outside $E.$ Now  we investigate  the behavior of $f$ near a singular point $a\in E.$
 
Consider  the local holomorphic  section $e_L$ given by  $(z,w)\mapsto {\partial\over\partial z}$ of $\Tan(X)$  over  $\U_a\simeq\D^2.$  
This  section induces  a holomorphic  section $\tilde e_L(x)=e_L(x)/\Tan(\Fc)_x$ of $\Nor(\Fc)$ over $\U_a.$

Our special assumption on   the ambient  metric  $g_X$ implies that on $\U_a,$ $g_X$  coincides with the Euclidean  metric. Therefore,
we  have, for $x=(z,w)\in\D\times (\D\setminus\{0\}),$
$$
\exp(-\varphi(x))=|\tilde e_L(x)|_{g^\perp_X}={1\over\sqrt{|z|^2+|\lambda w|^2} }|\lambda w|.
$$
Hence,   for $x=(z,w)\in\D^2\setminus\{(0,0)\},$
\begin{equation}\label{e:curvature_Nor}c_1(\Nor(\Fc), g^\perp_X)(x)=\ddc\varphi(x)=\ddczw\log{\sqrt{|z|^2+|\lambda w|^2}}.
\end{equation}
 Moreover,  in the local model with coordinates $(z,w)$ associated to the singular point
$E\ni a\simeq (0,0)\in\D^2,$ it follows  from \eqref{e:curvature_Nor} that
\begin{equation*}
c_1(\Nor(\Fc),g_0)+\ddc f= c_1(\Nor(\Fc), g^\perp_X)(x)=\ddc\varphi(x)={1\over 2}\ddczw\log{(|z|^2+|\lambda w|^2)}.
\end{equation*}
Consequently,   for  $x=(z,w)\in \U_a\simeq \D^2,$
\begin{equation}\label{e:f_smoothness}f(x)={1\over 2}\log{(|z|^2+|\lambda w|^2 ) }+\text{a smooth function in}\ x.
\end{equation}
Using  a finite partition of the unity on $X$ and  using \eqref{e:f_smoothness}),  we can construct  a  family  of smooth  functions  $(f_\epsilon)_{0<\epsilon\ll 1}$ on $X$  such that
$f_\epsilon$ converges uniformly to $f$ in $\Cc^2$-norm on  each regular flow box as $\epsilon\to 0$ and that
 in a local model with coordinates $(z,w)$ associated to each singular point
$a\in E,$  
\begin{equation}\label{e:f_epsilon}f_\epsilon-{1\over2}\log{(|z|^2+|\lambda w|^2  +\epsilon^2)}=f-{1\over2}\log{(|z|^2+|\lambda w|^2 )}\qquad\text{on}\qquad \D^2.
\end{equation}
For every $0<\epsilon\ll 1$ we endow  $\Nor(\Fc)$  with  the  metric $g_\epsilon:=g_0\exp{(-2f_\epsilon)}.$
Since $g_\epsilon$ is  smooth and the current $T$ is  $\ddc$-closed, it follows that
\begin{equation}\label{e:cohomo_g_eps}
c_1(\Nor(\Fc))\smile \{T\}=\int_Xc_1(\Nor(\Fc),g_\epsilon)\wedge T.
\end{equation}
 Let $\kappa_\epsilon:\ X\setminus E\to\R$ be the function defined by 
 \begin{equation}\label{e:kappa_epsilon}
 -c_1(\Nor(\Fc),g_\epsilon)(x)|_{ L_x}= \kappa_\epsilon(x)g_P(x).
\end{equation}
 This, combined with \eqref{e:mu}, implies that
\begin{equation}\label{e:kappa_epsilon_bis}
 -c_1(\Nor(\Fc),g_\epsilon)\wedge T= \kappa_\epsilon d\mu.
\end{equation}
Since $f_\epsilon$ converges uniformly to $f$ in $\Cc^2$-norm on    compact subsets of $X\setminus E$  as $\epsilon\to 0,$ 
it follows that  $\kappa_\epsilon$ converge pointwise to $\kappa$ $\mu$-almost everywhere.
Hence, we  get that
\begin{equation*}
-\int_{X\setminus(\bigcup_{a\in E}  \U_a)}c_1(\Nor(\Fc),g_\epsilon)\wedge T=\int_{X\setminus(\bigcup_{a\in E}  \U_a)}\kappa_\epsilon(x)d\mu(x)
\to \int_{X\setminus(\bigcup_{a\in E}  \U_a)}\kappa(x) d\mu(x)\quad\text{as}\quad \epsilon\to 0.
\end{equation*}
We will  show that  on each  singular flow box $\U_a\simeq \D^2,$
\begin{equation}\label{e:inside}
-\int_{ \U_a}c_1(\Nor(\Fc),g_\epsilon)\wedge T\to \int_{  \U_a} \kappa(x)d\mu(x)\qquad\text{as}\qquad \epsilon\to 0.
\end{equation}
Taking   \eqref{e:inside} for granted, we combine it with  the previous limit and get that
$$
-\int_{X}c_1(\Nor(\Fc),g_\epsilon)\wedge T\to \int_X\kappa(x)d\mu(x)=-\int_{X}c_1(\Nor(\Fc),g^\perp_X)\wedge T  \qquad\text{as}\qquad \epsilon\to 0,
$$
where the last equality follows from  Corollary \ref{C:Lyapunov_vs_kappa}.
We deduce from this and \eqref{e:cohomo_g_eps} that
$$
-c_1(\Nor(\Fc))\smile \{T\}=\int_X\kappa(x)d\mu(x).
$$
By  Proposition \ref{P:Lyapunov_vs_kappa}, the right hand side is $\chi(\mu).$
Hence, the last   equality  implies the desired identity of the theorem.

Now  it remains to prove \eqref{e:inside}.
We need the  following result which gives a precise behaviour of $\kappa_\epsilon$ near a singular point $a$ 
 using the local model  $(\D^2,\Lc,\{0\})$ introduced in  Subsection   \ref{SS:Existence-Lyapunov}. 
\begin{lemma}\label{L:laplacian_hol_bis}  
There is a constant $c>1$  such that for every $0<\epsilon\ll 1$ and 
for every  $x=(z,w)\in({1\over2}\D)^2,$  we have that  
\begin{multline*}
-c\big(  {|z|^2|w|^2\over (|z|^2+|w|^2+\epsilon^2)^2 }+  (|z|^2+|w|^2)\big)  (\lof {\|(z,w)\|})^2\leq  \kappa_\epsilon(x)\\
\leq \big(-c^{-1} {|z|^2|w|^2\over (|z|^2+|w|^2+\epsilon^2)^2} +c(|z|^2+|w|^2)\big)  (\lof {\|(z,w)\|})^2.
\end{multline*}  
\end{lemma}
\proof
Since  $g_\epsilon=g_0\exp{(-2f_\epsilon)}$ we get  that
$$
c_1(\Nor(\Fc),g_\epsilon)=c_1(\Nor(\Fc),g_0)+\ddc f_\epsilon=\ddc f_\epsilon+\text{a smooth $(1,1)$-form independent of $\epsilon$}.
$$
This,  together  with  \eqref{e:f_smoothness}, \eqref{e:f_epsilon} and \eqref{e:kappa_epsilon}, imply that
\begin{equation*}
 \kappa_\epsilon(x)g_P(x)= -{1\over 2} \ddc\log{(|z|^2+|\lambda w|^2+\epsilon^2)}(x)|_{ L_x}+\text{a smooth $(1,1)$-form independent of $\epsilon$}.
\end{equation*}
 Using the parametrization \eqref{e:leaf_equation} the pull-back of the first term of  the right-hand side  by $\psi_x$ is 
\begin{equation*}
 -{1\over 2}\ddc\log{(|ze^{i\zeta}|^2+|\lambda we^{i\lambda\zeta}|^2+\epsilon^2)}(0),
\end{equation*}
whereas  the pull-back of the second term of  the right-hand side  by $\psi_x$ is   $O(|z|^2+|w|^2) d\zeta\wedge d\bar\zeta.$ 
A straightforward computation  shows that the former  expression is  equal to
\begin{equation*}
-{|\lambda-1|^2\over 4\pi}  {|z|^2|w|^2\over (|z|^2+|\lambda w|^2+\epsilon^2)^2} id\zeta\wedge d\bar\zeta.
\end{equation*}  
On the other hand,  by \eqref{e:leaf_equation} we get that
$$
|d\psi_x(0)|\approx \|(z,w)\|=\dist(x,E).
$$
Using the last two estimates and applying Proposition \ref{P:Poincare}, the result follows.
\endproof

We resume the proof of  \eqref{e:inside}. By Proposition \ref{P:laplacian_hol} and Lemma \ref{L:laplacian_hol_bis}, there is  a  constant $c>1$ such that for $(z,w)\in\U_a\simeq\D^2$ that
$$
|\kappa_\epsilon(z,w)|\leq c \big({|z|^2|w|^2\over (|z|^2+|w|^2+\epsilon^2)^2} +(|z|^2+|w|^2)\big) (\lof {\|(z,w)\|})^2\leq c^2 |\kappa(z,w)|+c^2.$$
Recall from Proposition \ref{P:Lyapunov_vs_kappa} that  $\int_{\U_a}|\kappa(x)|d\mu(x)<\infty.$ 
On the other hand,  $\kappa_\epsilon$ converge pointwise to $\kappa$ $\mu$-almost everywhere as $\epsilon\to 0.$
Consequently, by Lebesgue dominated convergence, 
$$
\lim\limits_{\epsilon\to 0}\int_{\U_a} \kappa_\epsilon d\mu  =\int_{\U_a} \kappa d\mu.
$$
This and  \eqref{e:kappa_epsilon} imply   \eqref{e:inside}.
The proof of  formula  \eqref{e:Lyapunov_exp}  is thereby completed in the special case where the ambient metric $g_X$ 
is  equal to the Euclidean metric in a local model near  every  singular point of $\Fc.$
\endproof
\subsection{Geometric   characterization of Lyapunov exponents}

To find  a  geometric  interpretation  of these  characteristic  quantities,
 our idea  consists    in replacing the Brownian trajectories by the more appealing objects, namely, the {\it     unit-speed   geodesic rays.}  These paths   are  parameterized  by  their  length
 (with respect to the leafwise Poincar\'e metric $g_P$).   Therefore, we characterize 
 the Lyapunov exponents  in terms of  the  expansion rates of $\mathcal A$ along the geodesic rays.
 %
 
 Let $\Fc=(X,\Lc,E)$ be  a   Riemann surface lamination with singularities.
 Recall from (\ref{e:covering_map}) that
 $(\phi_x)_{x\in \Hyp(\Fc)}$ is  a given family of  universal covering maps $\phi_x:\ \D\to L_x$  with $\phi_x(0)=x.$
 For every $x\in \Hyp(\Fc),$  
  the  set of all unit-speed geodesic rays $\omega:\ [0,\infty)\to L_x$ starting at $x$ (that is,    $\omega(0)=x$),     can  be described   by  the family  $(\gamma_{x,\theta})_ {\theta\in [0,1)},$ where
 \begin{equation}\label{eq_geodesics}
 \gamma_{x,\theta}(R):=  \phi_x(e^{2\pi i\theta}r_R),\qquad  R\in\R^+,
 \end{equation}
 and  $r_R$  is uniquely determined  by the  equation $r_R\D=\D_R $  (see \eqref{e:radii_conversion}).
 The path $\gamma_{x,\theta}$ is called the {\it  unit-speed geodesic ray}   at $x$ with the leaf-direction $\theta.$
 Unless  otherwise  specified, the {\it space of leaf-directions} $[0,1)$ is  endowed with the Lebesgue measure. 
 This space   is  visibly identified, via the map $[0,1)\ni\theta\mapsto e^{2\pi i\theta},$ with the unit circle $\partial \D$ endowed with the normalized rotation measure.

 Set $\Omega:=\Omega(\Fc)$  as usual.  We  introduce the following   notions of expansion  rates for  cocycles.
 
  \begin{definition}\label{D:expansion_rate}\rm
 Let $\mathcal A:\  \Omega\times\R^+\to  \GL(d,\K)$ be a $\K$-valued   cocycle and  $R>0$  a  time and $x$  a point  in $\Hyp(\Fc).$ 

  The {\it expansion rate}  of $\mathcal A$  at $x$ in the leaf-direction $\theta$  at time  $R$ along the vector $v\in \K^d\setminus\{0\}$ 
 is the number
 $$\Ec(x,\theta,v,R):={1\over  R}  \log {\| \mathcal A(\gamma_{x,\theta},R)v   \|\over  \| v\|}.$$ 
 
 The {\it  expansion rate}  of $\mathcal A$   at  $x$ in the leaf-direction $\theta$ 
  at time  $R$
 is
 \begin{equation*}\begin{split}
 \Ec(x,\theta,R):= \sup\limits_{v \in \K^d\setminus\{0\}} \Ec(x,\theta,v,R)&=\sup_{v \in \K^d\setminus\{0\} } {1\over  R}  \log {\| \mathcal A(\gamma_{x,\theta},R)v   \|\over  \| v\|}\\
 &=
 {1\over  R}  \log {\| \mathcal A(\gamma_{x,\theta},R)  \|}.
 \end{split}
 \end{equation*}

 Given a $\K$-vector subspace   $\{0\}\not=H\subset \K^d,$ the {\it expansion rate}  of $\mathcal A$  at  $x$ at time  $R$
 along the vector space  $H$  
 is the interval $\Ec(x,H,R):=[a,b],$ where
 $$
 a:=  \inf_{v \in H\setminus \{0\}} \int_0^1   \Big ({1\over  R}  \log {\| \mathcal A(\gamma_{x,\theta},R)v   \|\over  \| v\|}\Big) d\theta\ \text{and}\
 b:=  \sup_{v \in H\setminus \{0\}} \int_0^1   \Big ({1\over  R}  \log {\| \mathcal A(\gamma_{x,\theta},R)v   \|\over  \| v\|}\Big) d\theta.
 $$
 \end{definition} 
  Notice  that    $ 
 \Ec(x,\theta,v,R) $ (resp. $
 \Ec(x,\theta,R)$)  expresses  geometrically the  expansion rate (resp. the maximal expansion rate) of the  cocycle 
when one travels along the unit-speed geodesic ray $\gamma_{x,\theta}$ up to time $R.$  
Moreover, the integral $\int_0^1 d\theta$ means that we  take  the average over all  possible directions $\theta.$
  On the other hand, $ 
 \Ec(x,H,R) $   represents the  smallest closed interval which  contains all numbers
$$\int_0^1   \Big ({1\over  R}  \log {\| \mathcal A(\gamma_{x,\theta},R)v   \|\over  \| v\|}\Big) d\theta,
$$
where  $v$ ranges over $H\setminus \{0\}.$
Note that  the above integral is the average of  the expansion rate  of the  cocycle 
when one travels along the unit-speed geodesic rays along the vector-direction $v\in H$   from $x$  to the  Poincar\'e circle  with radius $R$  and  center $x$ spanned  on $L_x.$ 

We say  that a sequence of  intervals $[a(R),b(R)]\subset \R$ indexed by $R\in\R^+$ converges to  a number $\chi\in \R$ and write $\lim_{R\to\infty} [a(R),b(R)]=\chi,$ if   
  $\lim_{R\to\infty} a(R)= \lim_{R\to\infty} b(R)=\chi.$
  
  Now  we are able  to state the main result of this  subsection.
  
  \begin{theorem} \label{T:Geometric_Lyapunov} {\rm (Nguyen \cite{NguyenVietAnh17b}).}
   Let $\Fc=(X,\Lc)$ be  a compact smooth hyperbolic Riemann surface lamination and $\mu$  a  harmonic measure
   which is also ergodic. 
Consider  a   smooth  cocycle
$\mathcal{A}:\ \Omega\times \R^+ \to  \GL(d,\K)  .    $ 
Then there is a leafwise saturated Borel set $Y$  of total $\mu$-measure 
which satisfies the conclusion of  Theorem \ref{T:VA_general}
and  the following additional geometric  properties: 
 \begin{itemize}
\item[(i)] For each $1\leq i\leq m$ and  for each $x\in Y,$ there is a  set $G_x\subset [0,1)$ of total Lebesgue measure 
 such that     for each  $v\in V_i(x)\setminus V_{i+1}(x),$ 
\begin{equation*}
 \lim_{R\to\infty}\Ec(x,\theta,v,R)= \chi_i,\qquad \theta\in G_x.
\end{equation*}
Moreover, the maximal Lyapunov exponent $\chi_1$ satisfies
 \begin{equation*}
 \lim_{R\to\infty}\Ec(x,\theta,R)= \chi_1,\qquad \theta\in G_x.
\end{equation*}  
\item[(ii)]  For  each  $1\leq i\leq m$ and each $x\in Y,$ 
\begin{equation*}
\lim_{R\to\infty}\Ec(x,H_i(x),R)=\chi_i.
\end{equation*} 
  \end{itemize}
Here
$\K^d=\oplus_{i=1}^m H_i(x),$ $x\in Y,$ is  the Oseledec  decomposition given by Theorem \ref{T:VA_general}   and      $\chi_m<\chi_{m-1}<\cdots
<\chi_2<\chi_1$ are the  corresponding     Lyapunov exponents.
\end{theorem}

 
 Theorem  \ref{T:Geometric_Lyapunov} gives a geometric meaning to  
 the stochastic formulas (\ref{e:Lyapunov})--(\ref{e:Lyapunov_max}).

 Let  $\Fc=(X,\Lc,E)$ be  a transversally smooth (resp.   holomorphic)  singular  foliation by    Riemann surfaces
in a Riemannian manifold (resp. Hermitian complex  manifold) $X.$
Consider  a leafwise saturated    set $M \subset X\setminus E$  which is compact in $X$ whose  leaves are all hyperbolic.  
So the  restriction of the foliation  $(X\setminus E,\Lc)$ to $M$ gives an inherited compact
  smooth/transversally holomorphic  hyperbolic Riemann lamination $(M,\Lc|_M).$ Moreover,
the   holonomy  cocycle of $(X\setminus E,\Lc)$ induces, by restriction, an inherited smooth cocycle on $(M,\Lc|_M).$
Hence, Theorem \ref{T:Geometric_Lyapunov} applies  to  the latter cocycle.   
In particular, when $(X,\Lc,E)$ is a singular  holomorphic     foliation  on a compact Hermitian complex  manifold $X$ of dimension $k,$   
 the last theorem
applies to the induced  holonomy cocycle of rank $k-1$ associated with every  minimal set  $M$ whose leaves are all  hyperbolic.

The proof of Theorem  \ref{T:Geometric_Lyapunov} (i)  relies on the theory of  Brownian trajectories  on hyperbolic spaces.
More  concretely, 
some quantitative results on  the boundary behavior of   Brownian trajectories  
by  Lyons \cite{Lyons} and Cranston \cite{Cranston} and on  
the shadow of   Brownian trajectories   by geodesic rays  
are  our main  ingredients.
This 
  allows us to replace a  Brownian  trajectory  by a unit-speed geodesic ray with uniformly  distributed leaf-direction.
Hence,   Part (i) of  Theorem  \ref{T:Geometric_Lyapunov}  will follow from  Theorem \ref{T:VA_general}. 

To establish Part (ii) of Theorem \ref{T:Geometric_Lyapunov} we  need  two steps.  
In the  first step we  adapt   to our context 
the  so-called {\it Ledrappier  type  characterization  of Lyapunov  spectrum} which  was introduced in  \cite{NguyenVietAnh17a}.  
 This  allows us to show that
 a  similar version of  Part (ii)  of Theorem \ref{T:Geometric_Lyapunov}  holds  when  the expansion  rates in terms of  geodesic rays
are replaced by some  heat diffusions associated  with the  cocycle.
The  second step shows that  the above  heat diffusions can be  approximated by  the expansion  rates. 
To  this end  we    establish a new  geometric estimate on the heat diffusions
which   relies on  the  proof of the geometric Birkhoff ergodic theorem  (Theorem \ref{T:geometric_ergodic}).

\begin{problem} \rm Is  Theorem \ref{T:Geometric_Lyapunov} still true if
   $\Fc=(X,\Lc)$  is  the whole regular part $(X'\setminus E', \Lc'|_{X'\setminus E'})$
of a  singular holomorphic  foliation $\Fc'=(X',\Lc' ,E')$ by hyperbolic Riemann surfaces on a compact complex manifold $X'$ and $\mathcal A$ is  the  holonomy cocycle  ? 
We  can begin  with  the case  where $X'$ is a  surface.
 \end{problem}


\bigskip

\small
\selectlanguage{english}

\end{document}